\documentclass[final,onefignum,onetabnum]{siamart251216}
\usepackage{bbm}
\usepackage[T1]{fontenc}
\usepackage{mathtools}
\usepackage{amssymb}
\usepackage{todonotes}
\usepackage{xcolor}
\usepackage{mathtools}
\usepackage{mathrsfs}
\usepackage{hyperref}
\usepackage{dsfont}
\usepackage{booktabs}

\usepackage{lipsum}
\usepackage{amsfonts}
\usepackage{graphicx}
\usepackage{epstopdf}
\usepackage{algorithmic}
\ifpdf
  \DeclareGraphicsExtensions{.eps,.pdf,.png,.jpg}
\else
  \DeclareGraphicsExtensions{.eps}
\fi

\usepackage[shortlabels]{enumitem}
\setlist[enumerate]{leftmargin=.5in}
\setlist[itemize]{leftmargin=.5in}

\newsiamremark{remark}{Remark}
\newsiamremark{hypothesis}{Hypothesis}
\crefname{hypothesis}{Hypothesis}{Hypotheses}
\newsiamthm{claim}{Claim}

\headers{Consistency of Learned Sparse Grid Quadrature}{H. Gottschalk, E. Partow, T. J. Riedlinger}

\title{Consistency of Learned Sparse Grid Quadrature Rules Using NeuralODEs\thanks{Submitted to the editors DATE. Authors listed in
alphabetical order.
}
}

\author{
Hanno Gottschalk\thanks{Technische Universität Berlin, Germany
  (\email{gottschalk@math.tu-berlin.de}).}
\and
Emil Partow\thanks{\textit{Corresponding author.}
  Ludwig-Maximilians-Universität München, Germany,
  and Munich Center for Machine Learning, Germany
  (\email{emil.partow@math.lmu.de}).}
\and
Tobias J.\ Riedlinger\thanks{Technische Universität Berlin, Germany
  (\email{riedlinger@math.tu-berlin.de}).}
}

\usepackage{amsopn}

\usepackage{mathrsfs}
\usepackage{graphicx}
\usepackage[caption=false]{subfig}
\usepackage{tikz}
\usepackage{pgfplots}
\pgfplotsset{compat=1.18}
\usepackage{afterpage}
\usepackage{comment}

\newcommand{\vk}{\boldsymbol k}

\newcommand{\valpha}{{\boldsymbol \alpha}}
\newcommand{\vbeta}{{\boldsymbol \beta}}

\newcommand{\evdist}[2]{ {\boldsymbol{\mathsf  E}}_{#1} \left[  #2 \right]}

\renewcommand{\d}[1]{\ensuremath{\operatorname{d}\!{#1}}}

\newtheorem{assumption}[theorem]{Assumption}

\numberwithin{equation}{section}

\newcommand{\R}{\mathbb{R}} 
\newcommand{\N}{\mathbb{N}} 

\DeclareMathOperator{\TV}{TV}
\DeclareMathOperator{\qoi}{QoI} 

\providecommand{\erf}{\operatorname{erf}}
\ifpdf
\hypersetup{
  pdftitle={Consistency of Learned Sparse Grid Quadrature Rules Using NeuralODEs},
  pdfauthor={H. Gottschalk, E. Partow, T. J. Riedlinger}
}
\fi

\begin{document}
\maketitle
\begin{abstract}
We prove consistency of a recently proposed scheme that evaluates expected
values by composing a learned transport map with Clenshaw--Curtis sparse-grid
quadrature on a tractable product source.
Our analysis hinges on the structural fact that composition of a $C^k_{\mathrm{mix}}$-regular function --- which carries the fast quadrature rate
$m^{-k}(\log m)^{(d-1)(k+1)}$ --- with a $C^1$-diffeomorphism can only be guaranteed to be $C^k_{\mathrm{mix}}$ itself, if the diffeomorphism is diagonal up to a permutation of coordinates.
The fast rate is therefore available exclusively for product targets, and the analysis splits into two regimes. In the \emph{general regime} of arbitrary targets, we learn the transport as the time-one flow of a $\mathrm{ReLU}^{k+1}$-neural ODE trained by maximum likelihood. The resulting flow lies in the isotropic space $C^k$ and yields the rate $m^{-k/d}(\log m)^{(d-1)(k/d+1)}$, with raising the density smoothness $k$ and the matched activation order $k+1$ mitigating the curse of dimensionality at the cost of harder optimization. In the \emph{diagonal regime} of product targets, the Knothe--Rosenblatt map is itself diagonal and we estimate it pointwise via empirical quantile transport, a lightweight alternative that recovers the full mixed-regularity rate. In both regimes, the resulting LtI estimator is PAC (probably approximately correct) consistent. With high probability the numerical integral approximates the true value to arbitrary accuracy as both the sample size $n$ and the quadrature budget $m$ tend to infinity.
\end{abstract}
\section{Introduction}
The numerical computation of integrals against probability measures $\mu$ is a
fundamental problem in numerical analysis \cite{Caflisch_1998,bungartz2004sparse,sloan_dick_kuo}, with applications in engineering,
the geosciences, and quantitative finance \cite{Glasserman2010-yd,sullivan2015introduction,Stuart_2010}. Often, the integrand is the output of a computationally expensive simulation that depends on multiple uncertain input parameters, and computing its expected value requires accurate quadrature in high dimensions. For several well-known distributions, such as the normal or the uniform, highly efficient sparse grid quadrature rules exist \cite{bungartz2004sparse, osti_1832293, smolyak1963, sullivan2015introduction}, which allow for efficient and accurate quadrature. However, the distributions encountered in practical applications do not necessarily fall into these classes and are typically far more complex.
In many relevant situations, there is not even explicit knowledge of the density of the distribution. Rather, the distribution is only represented as a set of samples of the parameters, e.g.\ obtained from a parametric bootstrap simulation \cite{davison1997bootstrap} or, in Bayesian methods, retrieved via Markov chain Monte Carlo \cite{gamerman2006markov}. These samples are often cheap to produce compared to the expensive simulation to which they serve as input, which makes naive Monte Carlo approaches inefficient.
Recently, the application of generative learning models has been proposed to overcome this difficulty by transforming the parameter distribution into a tractable distribution for which sparse grid quadrature rules are known \cite{ernst2025learning}.
Normalizing Flows (NF) \cite{papamakarios2021normalizing,rezende2015variational} learn a transport map \cite{Santambrogio2015-lp} that pushes a complicated multivariate distribution $\mu \in \mathcal{M}_1^+(\Omega)$ --- the space of probability measures on $\R^d$ supported on $\Omega$ --- to a simple noise distribution $\nu$, typically the multivariate standard normal or the uniform on the cube.
Concretely, NFs use neural network maps $\Phi^\theta$ that are easy to invert, so that both the generative direction $\Phi^\theta_*\nu \approx \mu$ and the normalizing direction $(\Phi^\theta)^{-1}_*\mu \approx \nu$ can be evaluated efficiently; here $\theta$ collects the weights and biases of the network and $\Phi^\theta_*\nu$ denotes the image measure of $\nu$ under $\Phi^\theta$, $\Phi^\theta_*\nu(B) = \nu((\Phi^\theta)^{-1}(B))$ for $B \subset \Omega$ Borel.
Existing models comprise affine coupling flows \cite{dinh2017density}, LU-net \cite{chan2023lu,rochau2024new}, and flows induced by neural ordinary differential equations (neuralODE) \cite{chen2019}, the latter being trainable either by maximum likelihood or via flow matching \cite{lipman2023flow}.
Across these architectures, impressive empirical results on modeling complex distributions have been obtained.
Alongside this empirical progress, a body of mathematical work studying the consistency of generative learning has emerged \cite{ehrhardt2025,marzouk2023,marzouk2025}.
These works provide convergence guarantees in the large-sample limit, typically by combining (optimal) transport theory for the existence of a sufficiently regular transport map \cite{ehrhardt2025,marzouk2023}, recent advances in the universal approximation of deep neural networks \cite{belomestny2022}, and methods from non-parametric statistics \cite{ehrhardt2025,marzouk2025}.
The present paper investigates the consistency of a recently proposed method called \emph{learning to integrate} (LtI) \cite{ernst2025learning}.
LtI uses a learned generative map $\Phi^\theta$ together with a sparse grid quadrature rule designed for the source $\nu$ to numerically integrate $\qoi$ over $\mu$, by integrating $\qoi \circ \Phi^\theta$ over $\nu$.
Its main benefit emerges when point evaluations of $\qoi$ are computationally costly while samples from $\mu$ are abundant. From such samples, a transport map $\Phi^\theta$ with $\Phi^\theta_*\nu \approx \mu$ can be learned, after which $\qoi$ is evaluated only on the sparse grid sites associated with $\nu$.
Our analysis of LtI starts from the convergence theory of Clenshaw--Curtis sparse grid quadrature on the $d$-dimensional unit cube.
With $m$ quadrature points, this rule attains the rate $m^{-k}(\log m)^{(d-1)(k+1)}$ on the space $C^k_{\mathrm{mix}}$ of functions with bounded mixed derivatives, but only the slower rate $m^{-k/d}(\log m)^{(d-1)(k/d+1)}$ on the larger space $C^k$ of functions with isotropic smoothness \cite{Clenshaw1960,Novak1996,Novak1997,Novak1999}.
In LtI, the integrand to which the rule is applied is the composition $\qoi \circ \Phi^\theta$, and the regularity of \emph{both} the quantity of interest and the learned transport enters the rate.
As a structural observation we prove that mixed regularity is preserved under composition with a $C^1$-diffeomorphism only when the latter is diagonal up to a permutation of coordinates; the fast mixed-regularity rate is therefore available for general LtI only when $\mu$ is a product measure.
The analysis splits accordingly into two regimes.
In the \emph{general regime}, where $\mu$ is arbitrary, we propose to learn the transport as the time-one flow of a neural ordinary differential equation with $\mathrm{ReLU}^s$ activations of arbitrary order $s \geq 2$, trained by maximum likelihood.
The activation order $s$ governs the regularity of the resulting flow, $\mathrm{ReLU}^s$ activations yielding $\Phi^\theta \in C^{s-1}$ and hence the isotropic quadrature rate exponent $(s-1)/d$, while the smoothness $k$ of the target density governs the statistical learning rate via nonparametric M-estimation.
The statistical learning theory underlying this construction is provided in \cite{ehrhardt2025,marzouk2023,marzouk2025} for the case $s = 2$ of $\mathrm{ReQU}$ activations. We extend it to general $\mathrm{ReLU}^s$, $s \geq 2$.
In the \emph{diagonal regime}, where $\mu$ is a product measure, we propose empirical quantile transport as a lightweight alternative. The marginal cumulative distribution functions of $\mu$ are estimated from data and their generalized inverses are taken as transport components.
In both regimes we prove that the integration scheme is probably approximately correct (PAC) consistent in the sense of \cite{shalev2014understanding}.
The setting analyzed in this paper differs in a few aspects from the numerical approach of \cite{ernst2025learning}, motivated by the requirements of the consistency analysis and the structural split into two regimes.
We work on the unit cube $\Omega = [0,1]^d$ rather than on $\R^d$, since (a) universal approximation of neural networks is more readily available on compact domains, and (b) the convergence theory of Clenshaw--Curtis sparse grids is well understood on the cube \cite{Clenshaw1960,Novak1996,Novak1997,Novak1999}.
For the source and target measures we assume continuous densities bounded away from zero, which is a standard setting in statistical learning theory for generative models \cite{marzouk2023,marzouk2025,ehrhardt2025}.
For the general regime we choose neuralODE among the available NF architectures because the statistical learning theory for this model is the most developed \cite{marzouk2023,marzouk2025}; the analysis of other architectures used in \cite{ernst2025learning} is left for future work.
For the diagonal regime, we propose empirical quantile transport rather than a neural transport. When $\mu$ is a product measure, the Knothe--Rosenblatt map reduces to a coordinatewise composition of inverse marginal cumulative distribution functions, which is directly accessible from the data and admits a clean concentration analysis.

The remainder of the paper is organized as follows.
Section~\ref{sec:LtI} introduces the LtI framework and the relevant notions from statistical learning theory.
Section~\ref{sec:quadrature_rates} introduces Clenshaw--Curtis rules and states the quadrature error bounds on $C^k$ and $C^k_{\mathrm{mix}}$ that govern the rates available in each regime.
Section~\ref{sec:transport} establishes the structural result that mixed regularity is preserved under composition only by transports that are diagonal up to a permutation of coordinates, motivating the split into the two regimes.
Section~\ref{sec:general_regime} treats the general regime: the transport is the time-one flow of a $\mathrm{ReLU}^s$-neuralODE trained by maximum likelihood, leading to PAC consistency at the isotropic rate $m^{-(s-1)/d}$.
Section~\ref{sec:diagonal_regime} treats the diagonal regime: the marginals of $\mu$ are matched via empirical quantile transport, leading to PAC consistency at the full mixed-regularity rate.
Section~\ref{sec:numerics} illustrates the theoretical results numerically on a range of Genz test integrands and product and non-product target measures.
Section~\ref{sec:Discussion_Outlook} concludes with a discussion of the results and an outlook on open questions.
The extension of the existing $\mathrm{ReQU}$ statistical learning theory to general $\mathrm{ReLU}^s$ activations, on which the analysis of the general regime relies, is presented in Appendix~\ref{app:relus}.
\section{Learning to Integrate via Generative Models}
\label{sec:LtI}
\subsection{Generative Learning}
\label{subsec:generative_learning}

The goal of generative learning is to approximate an unknown \emph{target
distribution} $\mu \in \mathcal{M}_1^+(\Omega)$ from i.i.d.\ samples
$X_1, \dots, X_n \sim \mu$. Throughout this paper  $\mathcal{M}_1^+(\Omega)$ denotes the space of probability measures
supported on $\Omega$. Rather than estimating the high-dimensional density of $\mu$ directly,
generative models represent $\mu$ implicitly through a \emph{transport map}
$\Phi \colon \Omega \to \Omega$ that pushes a known, easily sampled
\emph{source distribution} $\nu \in \mathcal{M}_1^+(\Omega)$ onto $\mu$, in
the sense that $\Phi_* \nu \approx \mu$. Sampling from the model then reduces
to pushing $\nu$-samples through $\Phi$. In practice, $\Phi$ is parametrized by a function class
$\{\Phi^\theta\}_{\theta \in \Theta}$, inducing a family of model
distributions $\mu_\theta := \Phi^\theta_* \nu$. The parameter $\theta$ is
chosen so that $\mu_\theta$ is close to $\mu$ in some divergence
$\mathcal{D} \colon \mathcal{M}_1^+(\Omega) \times \mathcal{M}_1^+(\Omega) \to
[0, \infty]$.
\subsection{Quadrature Rules}
\label{subsec:quadrature}
A \emph{quadrature rule} for a probability measure
$\rho \in \mathcal{M}_1^+(\Omega)$ is a finite collection of nodes and weights
$(\xi_j, w_j)_{j=1}^m \subset \Omega \times \R$ approximating the expectation of
a function $g \colon \Omega \to \R$ under $\rho$ by a weighted sum,
\begin{align}
\evdist{\rho}{g} = \int_\Omega g(x) \, \d\rho(x) \;\approx\; \sum_{j=1}^m w_j \, g(\xi_j).
\label{eq:quad_generic}
\end{align}
The number $m$ of nodes is the \emph{budget} of the rule. Both nodes and
weights depend explicitly on $\rho$. Their construction presupposes $\rho$ to
be given in closed form, and efficient high-order rules are typically
available only when $\rho$ is a product of standard one-dimensional
measures, such as the uniform or the standard normal
\cite{Novak1997,Novak1996,Novak1999}.
\subsection{Learning to Integrate}
\label{subsec:learning_to_integrate}
Let $\nu, \mu \in \mathcal{M}_1^+(\Omega)$ be probability distributions on $\Omega$, and suppose a transport map $\Phi \colon \Omega \to \Omega$ pushes $\nu$ to $\mu$, that is, $\Phi_*\nu = \mu$.
Then, by the change-of-variables formula, the expected value of a quantity of interest $\qoi \colon \Omega \to \R$ under $\mu$ can be rewritten as an expected value under $\nu$,
\begin{align}
\evdist{\mu}{\qoi} = \int_{\Omega} \qoi(x) \, \d\mu(x) = \int_{\Omega} \qoi\big(\Phi(z)\big) \, \d\nu(z).
\label{eq:change_of_variables}
\end{align}
The identity \eqref{eq:change_of_variables} motivates the LtI approach to numerical integration against $\mu$. A quadrature rule $(\xi_j, w_j)_{j=1}^m$ for $\nu$, in the sense of Section~\ref{subsec:quadrature}, induces via $\Phi$ an integration rule for $\mu$,
\begin{align}
\evdist{\mu}{\qoi} \approx \sum_{j=1}^m w_j \, \qoi\big(\Phi(\xi_j)\big),
\label{eq:lti_rule}
\end{align}
which is implementable whenever $\nu$ is tractable for quadrature and $\Phi$ is explicitly available; cf. fig.~\ref{fig:lti_schematic}.
\begin{figure}[ht]
    \centering
    \includegraphics[width=0.7\linewidth]{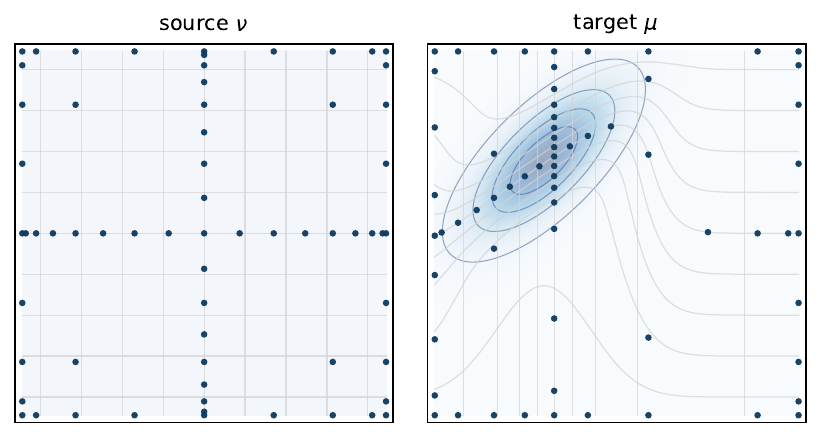}
    \caption{LtI in $d=2$: quadrature nodes $\xi_j$ for uniform source $\nu$ (left) transported to quadrature nodes $\Phi(\xi_j)$ for a  target $\mu$ (right).}
    \label{fig:lti_schematic}
\end{figure}
In the setting considered in this paper, $\mu$ is unknown and accessible only through i.i.d.\ samples $X_1, \dots, X_n \sim \mu$.
The transport map $\Phi$ is therefore unavailable in closed form and is replaced by an estimator $\Phi^{\hat{\theta}_n}$ learned from the data, with $\Phi^{\hat{\theta}_n}_*\nu \approx \mu$, in the sense of Section~\ref{subsec:generative_learning}.
Substituting $\Phi^{\hat{\theta}_n}$ for $\Phi$ in \eqref{eq:lti_rule} yields the \emph{learning to integrate} (LtI) approximation
\begin{align}
\evdist{\mu}{\qoi} \approx \sum_{j=1}^m w_j \, \qoi\big(\Phi^{\hat{\theta}_n}(\xi_j)\big).
\label{eq:lti_estimator}
\end{align}
\subsection{PAC Consistency}
\label{subsec:pac}
The accuracy of the LtI approximation \eqref{eq:lti_estimator} depends on the random sample $\chi_n := (X_1, \dots, X_n)$, $X_i \sim \mu$, and the finite quadrature budget $m$.
To formulate consistency in this setting, we first recall the standard notion of statistical learning of a distribution from samples, cf.~\cite{shalev2014understanding}, and then specialize it to the integration error of the LtI scheme.
Let $\mathcal{T} \subset \mathcal{M}_1^+(\Omega)$ denote a \emph{target class} of admissible distributions.
A family of estimators $\{\hat\mu_n\}_{n \in \N}$, where each $\hat\mu_n(\chi_n) \in \mathcal{M}_1^+(\Omega)$ depends measurably only on the observed samples, is said to \emph{learn} the target class $\mathcal{T}$ with respect to a divergence $\mathcal{D} \colon \mathcal{M}_1^+(\Omega) \times \mathcal{M}_1^+(\Omega) \to [0, \infty]$ if
\[
\mathcal{D}(\mu \,\|\, \hat\mu_n(\chi_n)) \;\xrightarrow[n \to \infty]{\mathbb{P}}\; 0 \qquad \text{for all } \mu \in \mathcal{T}.
\]
A stronger notion is \emph{probably approximately correct} (PAC) learnability. The class $\mathcal{T}$ is PAC-learnable with respect to $\mathcal{D}$ if, for all $\varepsilon, \delta > 0$, there exists a sample size threshold $n(\varepsilon, \delta) \in \N$ such that
\begin{align}
\mathbb{P}\!\left(\mathcal{D}(\mu \,\|\, \hat\mu_n(\chi_n)) > \varepsilon\right) \leq \delta \qquad \text{for all } \mu \in \mathcal{T} \text{ and all } n \geq n(\varepsilon, \delta).
\label{eq:pac_distribution}
\end{align}
In the LtI context, the relevant quantity is not the distribution
$\hat\mu_n := (\hat\Phi_n)_* \nu$ itself but the \emph{integration error}
\begin{align}
\varepsilon^{\mathrm{tot}}_{n,m}(\qoi)
:= \left| \evdist{\mu}{\qoi} - \sum_{j=1}^m w_j \, \qoi\big(\hat\Phi_n(\xi_j)\big) \right|,
\label{eq:total_error}
\end{align}
which depends on both the sample size $n$ and the quadrature budget $m$.
We say that the LtI scheme is \emph{PAC consistent} on the target class
$\mathcal{T}$ if, for all $\varepsilon, \delta > 0$, there exist a quadrature
budget threshold $m(\varepsilon, \delta) \in \N$ and a sample size threshold
$n(\varepsilon, \delta, m) \in \N$, the latter potentially depending on
the quadrature budget, such that
\begin{align}
\mathbb{P}\!\left(\varepsilon^{\mathrm{tot}}_{n,m}(\qoi) > \varepsilon\right)
\leq \delta \qquad \text{for all } \mu \in \mathcal{T},\
m \geq m(\varepsilon, \delta),\
n \geq n(\varepsilon, \delta, m).
\label{eq:pac_lti}
\end{align}
Allowing $n$ to depend on $m$ reflects the fact, that refining the quadrature without increasing the sample size may
amplify the contribution of $\hat\Phi_n$'s estimation error to the total
integration error. Whether such an amplification actually
occurs depends on the specific transport estimator and quadrature rule.

\begin{remark}[Naive Monte Carlo as baseline]
\label{rmk:mc_baseline}
The natural competitor for the LtI scheme is naive Monte Carlo, which estimates $\evdist{\mu}{\qoi}$ by averaging $\qoi$ over i.i.d.\ samples drawn directly from $\mu$,
\[
\hat I_n^{\mathrm{MC}} := \frac{1}{n} \sum_{j=1}^n \qoi(Y_j), \qquad Y_1, \dots, Y_n \stackrel{\mathrm{iid}}{\sim} \mu.
\]
For bounded $\qoi$ with $\|\qoi\|_\infty \leq M$, Hoeffding's inequality \cite{Hoeffding01031963} yields
\[
\mathbb{P}\!\left(\big|\hat I_n^{\mathrm{MC}} - \evdist{\mu}{\qoi}\big| > \varepsilon\right) \leq 2 \exp\!\left(-\frac{n \varepsilon^2}{2 M^2}\right),
\]
so that achieving accuracy $\varepsilon$ with confidence $1 - \delta$ requires
\begin{equation}
n_{\mathrm{MC}}(\varepsilon, \delta) \;=\; \frac{2 M^2}{\varepsilon^2} \log\!\left(\frac{2}{\delta}\right) \;=\; \mathcal{O}\!\left(\varepsilon^{-2} \log(1/\delta)\right)
\label{eq:mc_rate}
\end{equation}
function evaluations of $\qoi$. In contrast to Monte Carlo, the LtI framework decouples the costs into a
quadrature budget $m$ and a learning budget $n$, with $m$ chosen freely. This split can be advantageous
whenever sampling from $\mu$ is cheap compared with evaluating $\qoi$.
\end{remark}
\section{Sparse Grid Quadrature on the Cube}
\label{sec:quadrature_rates}
This section gathers the \emph{sparse grid} quadrature rules used throughout the paper, together with their convergence rates. In the following, let $\Omega := [0,1]^d$.
\subsection{Sparse Grid Construction}
\label{subsec:sparse_grid}
The sparse grid construction assembles a quadrature rule on the cube $[0,1]^d$ from a family of univariate rules on $[0,1]$, in such a way that the number of multivariate function evaluations grows only polynomially in the dimension $d$.
For each direction $i = 1, \dots, d$ and each level $l \in \N$, let $\{(\xi_{j,l}^{(i)}, w_{j,l}^{(i)})\}_{j=1}^{m_l}$ denote a univariate $m_l$-point quadrature rule on $[0,1]$ for the marginal $\nu_i$ of a product source $\nu = \bigotimes_{i=1}^d \nu_i$ with density $f_\nu(x) = \prod_{i=1}^d f_{\nu_i}(x_i)$.
The function $l \mapsto m_l \in \N$ is called the \emph{growth rule}.
For a multi-index $\vk = (k_1, \dots, k_d) \in \N^d$, the corresponding \emph{tensorized quadrature operator} is
\[
I_{\vk}^d(f) := \sum_{j_1=1}^{m_{k_1}} \cdots \sum_{j_d=1}^{m_{k_d}}
\prod_{i=1}^d w_{j_i, k_i}^{(i)} \cdot f\!\left(\xi_{j_1, k_1}^{(1)}, \dots, \xi_{j_d, k_d}^{(d)}\right),
\]
defined for functions $f \colon [0,1]^d \to \R$.
Evaluating $I_{\vk}^d$ requires $\prod_{i=1}^d m_{k_i}$ function evaluations, which grows exponentially with the dimension $d$. The \emph{curse of dimensionality} thus severely limits the practical applicability of full tensorization in high-dimensional settings.
To mitigate this exponential growth, Smolyak~\cite{smolyak1963} introduced a sparse tensor product construction that combines lower-dimensional tensorized rules.
The \emph{Smolyak sparse quadrature operator} is, for $q \geq d$,
\begin{align}
\mathcal{S}_q^d(f) := \sum_{q - d + 1 \leq |{\vk}| \leq q} (-1)^{q - |{\vk}|} \binom{d-1}{q - |{\vk}|}\, I_{\vk}^d(f),
\label{eq:smolyak}
\end{align}
where $|{\vk}| = k_1 + \cdots + k_d$ and $\ell := q - d$ is called the \emph{sparsity level} \cite{wasilikowski1995}.
For the \emph{closed nonlinear growth rule}
\begin{align}
m_1 = 1, \qquad m_l = 2^{l-1} + 1 \quad (l \geq 2),
\label{eq:growth_rule}
\end{align}
the number of function evaluations required for $\mathcal{S}_q^d = \mathcal{S}_{\ell+d}^d$ satisfies the asymptotic bound
\[
m(\ell + d, d) \;\simeq\; \frac{2^\ell}{\ell!}\, d^\ell \qquad (d \to \infty, \text{ $\ell$ fixed}),
\]
see \cite{Novak1997, Novak1999}.
For fixed sparsity level $\ell$, the computational cost thus exhibits polynomial growth in $d$ of degree $\ell$, in contrast to the exponential growth of the full tensorized rule. Once the univariate rules and the growth rule are fixed, we adopt the notation
\[
( \xi_j^{(\ell)}, w_j^{(\ell)} )_{j=1}^{m(\ell + d, d)} \subset [0,1]^d \times \R
\]
for the full collection of nodes and weights of the sparse grid rule $\mathcal{S}_{\ell + d}^d$, indexed arbitrarily.
When the sparsity level $\ell$ is fixed, we may also drop the superscript $(\ell)$.
\subsection{Clenshaw--Curtis Quadrature in One Dimension}
\label{subsec:clenshaw_curtis_rule}
The univariate building block of our sparse grid construction is the \emph{Clenshaw--Curtis rule}~\cite{Clenshaw1960}, designed for the uniform measure on $[-1,1]$.
The underlying idea is the substitution $x = \cos z$, which transforms the integration of $f \colon [-1,1] \to \R$ into the integration of $f(\cos z) \sin z$ on $[0, \pi]$.
Expanding $f(\cos z)$ in a Fourier cosine series $f(\cos z) = \tfrac{a_0}{2} + \sum_{j=1}^\infty a_j \cos(j z)$ yields the closed-form expression
\[
\int_{-1}^{1} f(x)\, \d x = \int_0^\pi f(\cos z) \sin z \, \d z = a_0 + \sum_{j=1}^\infty \frac{2\, a_{2j}}{1 - (2j)^2},
\]
so that the integration reduces to the computation of the cosine coefficients $a_j$.
The Clenshaw--Curtis $m$-point rule truncates this series at $j = m - 1$ and approximates the remaining coefficients $a_j$ by a discrete cosine transform of the values of $f$ at the Chebyshev extrema $\xi_j = \cos\!\big(\tfrac{(j-1)\pi}{m-1}\big)$ for $j = 1, \dots, m$, leading to a quadrature rule of the form $\sum_{j=1}^m w_j f(\xi_j)$ with weights $w_j$ that depend only on $m$.
Equivalently, the rule integrates exactly the unique polynomial of degree at most $m-1$ that interpolates $f$ at these nodes \cite[Section 2]{Sloan1978}. The weights $\{w_j\}_{j=1}^m$ implementing the Clenshaw--Curtis rule can be (pre-)computed in $\mathcal{O}(m \log m)$ operations \cite{SOMMARIVA2013682, Waldvogel2006}.
For integration on a general interval $[a, b]$, the nodes and weights are mapped affinely as $\xi_j \mapsto \tfrac{a+b}{2} + \tfrac{b-a}{2}\, \xi_j$ and $w_j \mapsto \tfrac{b-a}{2}\, w_j$.
Clenshaw--Curtis rules exhibit a range of desirable properties, which make them a practical default choice in a wide range of settings.
E.g., the Chebyshev extrema are \emph{nested}. That is, the $m_l$-point rule with $m_l = 2^{l-1} + 1$ is contained in the $m_{l+1}$-point rule, so that function evaluations on coarser levels are reused on finer ones.
\begin{remark}[Source measure for the univariate building block]
\label{rmk:cc_for_uniform}
The Clenshaw--Curtis rule above is designed for the Lebesgue/uniform measure on $[0,1]$. For a source $\nu$ with smooth density $f_\nu$, one can apply the same rule to the modified integrand $g \cdot f_\nu$, since $\int g\, \d\nu = \int g\, f_\nu\, \d\lambda^d$, and the convergence rates carry over to this setting.
\end{remark}
\subsection{Quadrature Error Bounds}
\label{subsec:quadrature_error_bounds}
The convergence rate of the Clenshaw--Curtis sparse grid quadrature \eqref{eq:smolyak} depends on the smoothness of the integrand.
We state the bound for both isotropic $C^k$-smoothness and the smaller class $C^k_{\mathrm{mix}}$ of functions with bounded mixed derivatives.
\subsubsection{Smoothness Spaces}
\label{subsubsec:smoothness_spaces}
Let $U \subset \R^{d_1}$ be an open and bounded set, and denote its closure by $\overline U$.
For a nonnegative integer $k$ and $\sigma \in \{\mathrm{iso}, \mathrm{mix}\}$, the smoothness space $C^k_\sigma(\overline U; \R^{d_2})$ consists of all functions $f \colon U \to \R^{d_2}$ whose partial derivatives $D^{\vbeta} f$ exist on $U$ for every multi-index $\vbeta \in \N_0^{d_1}$ with $|\vbeta|_\sigma \leq k$ and admit continuous extensions to $\overline U$, equipped with the norm
\begin{align}
\|f\|_{C^k_\sigma(\overline U; \R^{d_2})} := \max_{|\vbeta|_\sigma \leq k}\, \sup_{x \in \overline U} \, \big\| D^{\vbeta} f(x) \big\|_2,
\label{eq:ck_norm}
\end{align}
where $|\vbeta|_{\mathrm{iso}} := \sum_{i=1}^{d_1} \beta_i$ and $|\vbeta|_{\mathrm{mix}} := \max_i \beta_i$. We follow standard convention and abbreviate $C^k := C^k_{\mathrm{iso}}$ when no confusion arises.
Since $|\vbeta|_{\mathrm{mix}} \leq |\vbeta|_{\mathrm{iso}} \leq d_1 \cdot |\vbeta|_{\mathrm{mix}}$, the mixed space is strictly smaller than the isotropic one for $d_1 \geq 2$ and $k \geq 1$, $C^k_{\mathrm{mix}} \subsetneq C^k$, and the mixed norm dominates the isotropic norm.
We will write $C^k_\sigma(U; V)$ for functions whose image lies entirely within a subset $V \subset \R^{d_2}$.
\subsubsection{Error Bound}
\label{subsubsec:error_bound}
We recall the classical convergence results.
\begin{theorem}[Sparse grid quadrature error]
\label{thm:quadrature_error}
Let $d, k \in \N$ and $\sigma \in \{\mathrm{iso}, \mathrm{mix}\}$.
Let $\nu \in \mathcal{M}_1^+([0,1]^d)$ be a product probability measure with marginal densities $f_{\nu_i} \in C^k([0,1])$, and let $(\xi_j, w_j)_{j=1}^m$ denote the Clenshaw--Curtis sparse grid rule for $\nu$ with closed nonlinear growth \eqref{eq:growth_rule}, in the sense of Remark~\ref{rmk:cc_for_uniform}.
Suppose $g \in C^k_\sigma([0,1]^d; \R)$.
Then there exists a constant $c_{d,k}^\sigma > 0$, depending only on $d$, $k$, $\sigma$, and $\max_i \|f_{\nu_i}\|_{C^k}$, such that
\begin{align}
\bigg|\int_{[0,1]^d} g(x)\, \d\nu(x) - \sum_{j=1}^m w_j\, g(\xi_j)\bigg|
\leq c_{d,k}^\sigma\, r_\sigma(m, d, k)\, \|g\|_{C^k_\sigma},
\label{eq:quadrature_error}
\end{align}
where the rates are
\[
r_{\mathrm{iso}}(m, d, k) := m^{-k/d} (\log m)^{(d-1)(k/d + 1)},
\quad
r_{\mathrm{mix}}(m, d, k) := m^{-k} (\log m)^{(d-1)(k+1)}.
\]
\end{theorem}
\begin{proof}
For $g \in C^k_\sigma([-1,1]^d; \R)$ and $\lambda^d = \mathrm{Uniform}([-1,1]^d)$, 
the Clenshaw--Curtis sparse grid quadrature with Lebesgue weights 
$\{w_j^\lambda\}_{j=1}^m$ at the nodes $\{\xi_j\}_{j=1}^m$ satisfies
\[
\bigg|\int_{[-1,1]^d} g\, \d\lambda^d - \sum_{j=1}^m w_j^\lambda\, g(\xi_j)\bigg|
\;\leq\; \hat c_{d,k}^\sigma\, r_\sigma(m, d, k)\, \|g\|_{C^k_\sigma},
\]
with universal constants $\hat c_{d,k}^\sigma > 0$; see 
\cite[Theorem and Remark~2]{Novak1997} for $\sigma = \mathrm{iso}$ and 
\cite[Corollary and Remark~1]{Novak1997} for $\sigma = \mathrm{mix}$. 
The estimate extends to $[0,1]^d$ via the affine reparametrization of 
Section~\ref{subsec:clenshaw_curtis_rule}. For a non-uniform product source $\nu$ with $f_{\nu_i} \in C^k([0,1])$, 
applying this rule to $g \cdot f_\nu$, after reweighting the Lebesgue 
weights $w_j^\lambda$ via $w_j := w_j^\lambda\, f_\nu(\xi_j)$, yields
\[
\sum_{j=1}^m w_j\, g(\xi_j) \;=\; \sum_{j=1}^m w_j^\lambda\, (g f_\nu)(\xi_j) 
\;\approx\; \int g f_\nu\, \d\lambda^d \;=\; \int g\, \d\nu,
\]
at the same rate with $\|g \cdot f_\nu\|_{C^k_\sigma} \leq c_{d,k}\,\|g\|_{C^k_\sigma}\,\prod_i \|f_{\nu_i}\|_{C^k}$ 
by Leibniz, which is absorbed into $c_{d,k}^\sigma$.
\end{proof}
\begin{remark}
\label{rmk:other_quadratures}
As pointed out in \cite[Remark~1, Remark~3]{Novak1996}, error bounds of the form \eqref{eq:quadrature_error} are not inherently restricted to Clenshaw--Curtis quadrature. Analogous estimates hold for sparse grid rules built from other univariate building blocks, so that Theorem~\ref{thm:quadrature_error} extends to a broader class of sparse grid integration schemes.
\end{remark}
\section{Transport Theory}
\label{sec:transport}
We now introduce the Knothe--Rosenblatt (KR) transport, as a constructive existence statement for a transport map between $\nu$ and $\mu$.
Throughout this and subsequent sections several transport-related symbols appear in close proximity. To orient the reader, we adopt the following convention. $\Phi \colon [0,1]^d \to [0,1]^d$ denotes a generic transport map, $T$ the (deterministic) Knothe--Rosenblatt transport defined in \eqref{eq:kr_trafo}, and $\Phi^\theta$ a parametric ansatz (typically the time-one flow of a neuralODE with vector field $v^\theta$), $\Phi_t$ or $\Phi_{0,t}$ the time-$t$ flow map of an ODE. Finally, $\hat\Phi_n$ denotes the estimator learned from the sample $\chi_n$, which is either the empirical MLE flow (general regime, Section~\ref{sec:general_regime}) or the empirical quantile transport (diagonal regime, Section~\ref{sec:diagonal_regime}); the precise meaning is fixed by context.
\subsection{Knothe--Rosenblatt Transport}
\label{subsec:kr}
The \emph{Knothe--Rosenblatt transport} is a triangular map that recursively matches the marginals of $\nu$ to those of $\mu$, while preserving the alignment of previously matched coordinates via conditional distributions.
Following \cite{Santambrogio2015-lp} we restrict the construction to the cube $[0,1]^d$ to avoid additional technical complications and assume that the source and target measures $\nu, \mu \in \mathcal{M}_1^+([0,1]^d)$ admit continuous densities $f_\nu, f_\mu$ satisfying
\begin{align}
f_\nu(x),\; f_\mu(x) \geq \kappa > 0 \qquad \text{for all } x \in [0,1]^d
\label{eq:density_lower_bound}
\end{align}
for some constant $\kappa > 0$. For $\bullet \in \{\nu, \mu\}$ and $1 \leq k \leq d$, the $k$-dimensional marginal density is
\[
\hat f_{\bullet, k}(x_1, \dots, x_k) := \int_{[0,1]^{d-k}} f_\bullet(x_1, \dots, x_d) \, \d\lambda^{d-k}(x_{k+1}, \dots, x_d),
\]
with the convention $\hat f_{\bullet, 0} \equiv 1$.
The corresponding conditional densities and conditional CDFs are
\[
f_{\bullet, k}(x \mid x_1, \dots, x_{k-1}) := \frac{\hat f_{\bullet, k}(x_1, \dots, x_{k-1}, x)}{\hat f_{\bullet, k-1}(x_1, \dots, x_{k-1})},
\]
and $F_{\bullet, k}(x \mid x_1, \dots, x_{k-1}) := \int_0^x f_{\bullet, k}(z \mid x_1, \dots, x_{k-1}) \, \d z$ for $k = 1, \dots, d$.
The Knothe--Rosenblatt transport is then defined component-wise. The first coordinate is
\[
T_1(x_1) := F_{\mu, 1}^{-1} \circ F_{\nu, 1}(x_1),
\]
and for $2 \leq k \leq d$, the remaining components are constructed recursively as
\[
T_k(x_1, \dots, x_k) := F_{\mu, k}^{-1}\!\left(F_{\nu, k}(x_k \mid x_1, \dots, x_{k-1}) \,\big|\, T_1(x_1), \dots, T_{k-1}(x_1, \dots, x_{k-1})\right).
\]
The resulting triangular map $T \colon [0,1]^d \to [0,1]^d$ is
\begin{align}
T(x_1, \dots, x_d) := \big(T_1(x_1),\ T_2(x_1, x_2),\ \dots,\ T_d(x_1, \dots, x_d)\big)^\top.
\label{eq:kr_trafo}
\end{align}
\begin{theorem}[Knothe--Rosenblatt transport]
\label{thm:kr}
Let $\nu, \mu \in \mathcal{M}_1^+([0,1]^d)$ admit continuous densities satisfying \eqref{eq:density_lower_bound}.
The Knothe--Rosenblatt transport $T$ defined in \eqref{eq:kr_trafo} satisfies
\[
T_* \nu = \mu, \qquad f_\nu(x) = f_\mu\big(T(x)\big)\, \big|\!\det DT(x)\big| \quad \text{for all } x \in [0,1]^d.
\]
\end{theorem}
\subsection{NeuralODE}
\label{subsec:neuralode}
The underlying idea of neuralODE is to model
the transport $\Phi$ implicitly through a continuous deformation of the
source distribution into the target \cite{chen2019}.
Concretely, one specifies a time-dependent vector field $v_t \colon \R^d \to \R^d$, $t \in [0,1]$, that prescribes at each instant in which direction every point should move; integrating these instantaneous velocities over the time interval $[0,1]$ then transports a starting position $y_0 \sim \nu$ along a continuous trajectory to a final position $y(1)$ whose distribution is intended to match $\mu$.
The transport map $\Phi$ is thus the time-one endpoint of the flow induced by the ordinary differential equation
\begin{align}
\frac{\mathrm{d}}{\mathrm{d} t} y(t) = v\big(y(t), t\big), \qquad y(t_0) = y_0,
\label{eq:flow_ode}
\end{align}
where $v \colon \R^d \times [0,1] \to \R^d$ is the chosen vector field, for which we use the shorthand notation $v_t(x) := v(x, t)$.
Assuming $v$ is Lipschitz continuous in the spatial variable and continuous in time \cite{hartmann2002}, the \emph{flow map} $\Phi_{t_0, t} \colon \R^d \to \R^d$ is well-defined by $\Phi_{t_0, t}(y_0) := y(t)$, where $y(t)$ is the unique solution of \eqref{eq:flow_ode}.
Without loss of generality we restrict to the unit time interval and write $\Phi_t := \Phi_{0, t}$ for flows starting at $t_0 = 0$, and $\Phi := \Phi_{0, 1}$ for the flow map evaluated at $t = 1$.
Representing the transport as an ODE flow offers two practical advantages.
First, under the regularity assumptions above, the flow endpoint $\Phi$ is invertible. A sample $z$ from $\Phi_*\nu$ is generated by drawing $X \sim \nu$ and solving \eqref{eq:flow_ode} forward in time with $y_0 = X$ to obtain $x = \Phi(X)$, and the inverse $\Phi^{-1}$ is obtained by solving the same ODE backward in time.
Second, when the vector field satisfies $v \in C^1(\R^d \times [0,1]; \R^d)$, Liouville's formula provides an efficient integral representation for the log-determinant of the Jacobian, which yields the change-of-variables identity
\begin{align}
\log f_{\Phi_*\nu}(y) = \log f_\nu\big(\Phi^{-1}(y)\big) - \int_0^1 \operatorname{div}_y v\big(\Phi_t(\Phi^{-1}(y)), t\big) \, \d t,
\label{eq:liouville_log_det}
\end{align}
see \cite[Lemma 2.1]{ehrhardt2025}.
This expression is central to likelihood-based training, as it avoids explicit computation of the Jacobian determinant.
\subsubsection{Training via Maximum Likelihood Estimation}
To turn the abstract construction of the previous paragraph into a learnable procedure, the vector field $v$ is modeled through a parametric family $\{v^\theta\}_{\theta \in \Theta}$ of neural networks, with corresponding flow endpoint $\Phi^\theta$, which gives rise to the name \emph{neuralODE}.
A standard way to learn the parameter $\theta$ from data is maximum likelihood estimation. Given an i.i.d.\ sample $\chi_n := (X_1, \dots, X_n)$ from $\mu$, define the empirical negative log-likelihood
\begin{align}
\hat L_n(\theta, \chi_n) := -\frac{1}{n} \sum_{j=1}^n \log f_{\Phi^\theta_*\nu}(X_j),
\label{eq:neural_ode_log_likelihood}
\end{align}
and minimize $\hat L_n(\cdot, \chi_n)$ over $\Theta$ to obtain the estimator $\hat\theta_n$.
The representation \eqref{eq:liouville_log_det} enables efficient computation of the densities $f_{\Phi^\theta_*\nu}(X_j)$ along the flow \cite{chen2019}.
\begin{remark}[Flow matching]
A recent alternative training approach is \emph{Flow Matching}, which avoids explicit integration by regressing a time-dependent vector field $v^\theta$ onto a constructed reference dynamics; see \cite{lipman2023flow}.
In particular, the ODE solves required at every step of maximum likelihood training \eqref{eq:neural_ode_log_likelihood} are eliminated.
\end{remark}
\subsubsection{ODE Representation of the Knothe--Rosenblatt Transport}
It turns out that the Knothe--Rosenblatt transport $T$ admits a representation as the time-one endpoint of an ODE flow of the form \eqref{eq:flow_ode}, which gives a theoretical guarantee that an estimate $\Phi^\theta$ can in principle reach a desired transport between $\nu$ and $\mu$, whenever condition \eqref{eq:density_lower_bound} is met.
Following \cite{marzouk2023, marzouk2025}, the construction is based on the straight-line interpolation between the identity and $T$,
\begin{align}
I_t(x) := t\, T(x) + (1 - t)\, x, \qquad (x, t) \in [0,1]^d \times [0,1],
\label{eq:displacement_interpolation}
\end{align}
which traces a path from $x$ at time $t = 0$ to $T(x)$ at time $t = 1$.
By \cite[Theorem 3.4]{marzouk2025}, $I_t$ is a diffeomorphism of $[0,1]^d$ for each $t \in [0,1]$, so that the inverse interpolation $G(x, t) := I_t^{-1}(x)$ identifies the initial position $x_0 = G(x, t)$ reaching $x$ at time $t$ along the interpolation path.
The associated \emph{target vector field}
\begin{align}
u^\mu_t(y) := T\big(G(y, t)\big) - G(y, t), \qquad (y, t) \in [0,1]^d \times [0,1],
\label{eq:target_vector_field}
\end{align}
points in the direction of the remaining displacement toward $T$ and generates the flow $\Phi_t = I_t$ with endpoint $\Phi_1 = T$, in the sense that
\[
\frac{\d{}}{\d t}\Phi_t(x) = u^\mu_t\big(\Phi_t(x)\big), \qquad \Phi_0(x) = x,
\]
see \cite[Theorem 3.4]{marzouk2025}; cf. fig.~\ref{fig:kr_transport}.
\begin{figure}[ht]
\centering
\includegraphics[width=0.85\linewidth]{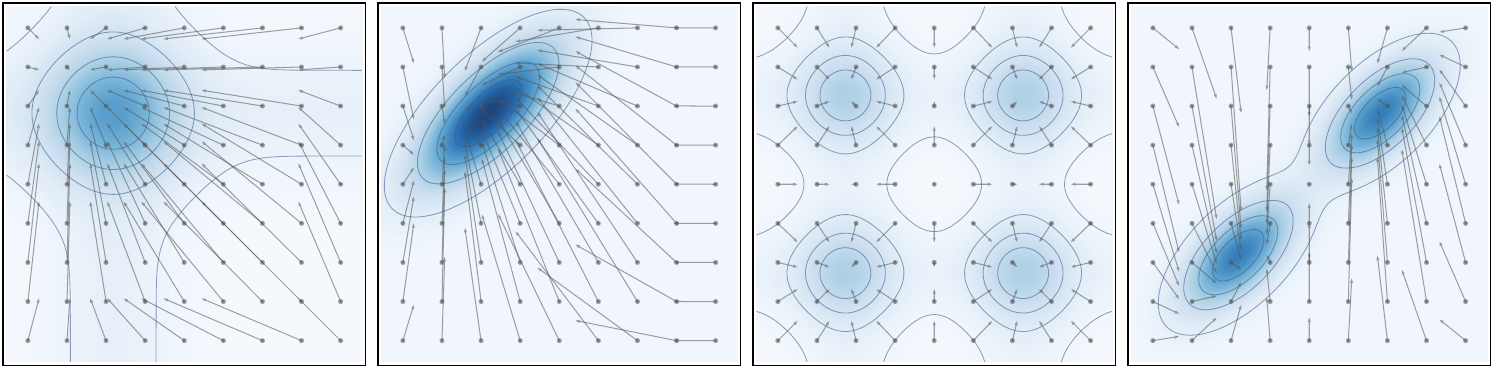}
\caption{Displacement interpolation between uniform source $\nu = \mathrm{Uniform}([0,1]^2)$ and four representative target measures $\mu$ on $[0,1]^2$. Grey points show a uniform grid in the source, with the trajectories $t \mapsto I_t(x) = t\,T(x) + (1-t)\,x$ traced out by the displacement interpolation \eqref{eq:displacement_interpolation} and arrowheads marking the endpoints $T(x)$ in the target.}
\label{fig:kr_transport}
\end{figure}
\subsection{Regularity-Preserving Diffeomorphisms}
\label{subsec:regularity_preserving}
When performing LtI, we integrate a composition $\qoi \circ \Phi$ of the quantity of interest with a learned transport $\Phi$, cf. Section~\ref{subsec:learning_to_integrate}. To apply Theorem~\ref{thm:quadrature_error}, it is crucial to analyze when this composition is of a given regularity.
There is in general no reason to expect that composition improves the regularity of a function. By the \emph{Faà di Bruno formula} \cite{Ma2009-rj}, the derivatives of $\qoi \circ \Phi$ depend on the derivatives of both $\qoi$ and $\Phi$, so that any irregularity in either factor propagates to the composition.
The regularity of the quantity of interest is typically assumed as part of the modeling setup, and we therefore ask under which conditions composition with $\Phi$ \emph{preserves} this regularity.
\subsubsection{Preservation of Isotropic Regularity}
\label{subsubsec:isotropic_composition}
Once both $\qoi$ and $\Phi$ are of class $C^k$, their composition is $C^k$ as well, with a quantitative bound on the norm.
\begin{theorem}[Isotropic composition bound]
\label{thm:iso_composition}
Let $d, k \in \N$, $\qoi \in C^k([0,1]^d; \R)$, and $\Phi \in C^k([0,1]^d; [0,1]^d)$.
Then $\qoi \circ \Phi \in C^k([0,1]^d; \R)$, and there exists a constant $\dot c_{d,k} > 0$ depending only on $d$ and $k$ such that
\begin{align}
\|\qoi \circ \Phi\|_{C^k} \leq \dot c_{d,k}\, \|\qoi\|_{C^k} \big(1 + \|\Phi\|_{C^k}\big)^k.
\label{eq:iso_composition_bound}
\end{align}
\end{theorem}
\begin{proof}
Apply the multivariate Faà di Bruno formula \cite[Corollary 12]{Ma2009-rj}.
\end{proof}
One can prove that the flow from~\eqref{eq:flow_ode} is smooth, if the underlying velocity field is.
\begin{theorem}[Smoothness of the flow]
\label{thm:flow_smoothness}
Let $k \in \N$ and let $v \in C^k([0,1]^d \times [0,1]; \R^d)$ be a time-dependent vector field whose flow $\Phi_t$ generated by \eqref{eq:flow_ode} preserves $[0,1]^d$ for all $t \in [0,1]$.
Then the time-one flow map $\Phi := \Phi_1 \in C^k([0,1]^d; [0,1]^d)$.
\end{theorem}
\begin{proof}
This is a classical result; see, e.g., \cite[Theorem V.4.1]{hartmann2002}.
\end{proof}
Theorems~\ref{thm:flow_smoothness} and~\ref{thm:iso_composition} together show that the isotropic rate $r_{\mathrm{iso}}(m, d, j)$ of Theorem~\ref{thm:quadrature_error} is available for any $C^j$-flow $\Phi$ of a $C^j$-vector field, $j \in \N$. In particular, the neuralODE construction with $\mathrm{ReLU}^s$ activations yields $\Phi \in C^{s-1}([0,1]^d; [0,1]^d)$, thus achieving the rate $r_{\mathrm{iso}}(m, d, s-1) = m^{-(s-1)/d} (\log m)^{(d-1)((s-1)/d + 1)}$ for $\qoi \in C^{s-1}([0,1]^d; \R)$.
\subsubsection{Preservation of Mixed Regularity Forces Diagonal Maps}
\label{subsubsec:diagonal_composition}
The faster rate $r_{\mathrm{mix}}(m, d, k)$ requires the stronger condition $\qoi \circ \Phi \in C^k_{\mathrm{mix}}$.
The next result shows that this is a strong structural constraint on $\Phi$. Only diagonal diffeomorphisms (up to a permutation of coordinates) preserve mixed regularity.
\begin{proposition}[Mixed regularity forces diagonality]
\label{prop:diagonal}
Let $\Phi \colon [0,1]^d \to [0,1]^d$ be a $C^1$-diffeomorphism such that $f \circ \Phi \in C^k_{\mathrm{mix}}([0,1]^d)$ for every $f \in C^k_{\mathrm{mix}}([0,1]^d)$.
Then there exists a permutation $\pi \in S_d$ and one-dimensional $C^1$-diffeomorphisms $\phi_1, \dots, \phi_d \colon [0,1] \to [0,1]$ such that
\[
\Phi(x) = \big(\phi_1(x_{\pi(1)}), \dots, \phi_d(x_{\pi(d)})\big).
\]
\end{proposition}
\begin{proof}[Sketch]
The hypothesis applied to the coordinate projections $\pi_i \in C^k_{\mathrm{mix}}$
yields $\Phi_i = \pi_i \circ \Phi \in C^k_{\mathrm{mix}}$, so all mixed
derivatives of $\Phi$ up to order $k$ exist. By the closed graph theorem, the composition operator
$T_\Phi \colon f \mapsto f \circ \Phi$ is bounded from
$C^k_{\mathrm{mix}}([0,1]^d)$ into itself. Testing $T_\Phi$ on the
oscillating functions $f_N(y) := N^{-k} \sin(N y_i)$ and
$\tilde f_N(y) := N^{-k} \cos(N y_i)$ for each coordinate~$i$ and
expanding the mixed top-order derivative
$\partial_j^k \partial_l^k (f_N \circ \Phi)$ via the Faà di Bruno
formula, the leading term is proportional to
$N^k\, (\partial_j \Phi_i)^k\,(\partial_l \Phi_i)^k$. The Pythagorean
identity $\sin^2 + \cos^2 = 1$ eliminates the oscillatory factor and
yields a uniform bound on $J(x)^2 := (\partial_j \Phi_i(x))^{2k} (\partial_l \Phi_i(x))^{2k}$ that vanishes as $N\to\infty$. Since $k \geq 1$, this yields
$\partial_j \Phi_i \cdot \partial_l \Phi_i \equiv 0$ for all $i$ and
all $j \neq l$. Each row of $D\Phi$ therefore has at most one nonzero
entry, and since $\det D\Phi \neq 0$, the matrix $D\Phi$ is a
generalized permutation matrix at every point. Continuity and
connectedness of $[0,1]^d$ then force the permutation to be globally
constant. The full argument is given in
Appendix~\ref{app:diagonality_proof}.
\end{proof}
By Proposition~\ref{prop:diagonal} applied to $\qoi \in C^k_{\mathrm{mix}}$,
the mixed-regularity rate $r_{\mathrm{mix}}$ is available only if $\Phi$ is
diagonal up to a permutation of coordinates. Since $\nu$ is a product measure
and diagonal maps preserve product structure, the target $\mu = \Phi_*\nu$
must then itself be a product measure.
However, in the case of product targets, the neuralODE machinery is unnecessarily heavy, as the KR map, cf. \eqref{eq:kr_trafo},  reduces to a coordinatewise composition of inverse marginal CDFs, fully determined by the marginals of $\mu$, which can be estimated directly from samples by their empirical counterparts.
\subsection{Learning Product Measures by Empirical Quantile Transport}
\label{subsec:quantile_estimator}
Following \eqref{eq:kr_trafo}, when both source and target $\nu, \mu \in \mathcal{M}_1^+([0,1]^d)$ are product measures with continuous, strictly positive densities $\nu = \bigotimes_{i=1}^d \nu_i,\mu = \bigotimes_{i=1}^d \mu_i$,
the conditional CDFs reduce to the marginal CDFs and the Knothe--Rosenblatt transport simplifies to the coordinatewise composition of inverse marginal CDFs,
\begin{align}
T(x) = \big(F_{\mu_1}^{-1} \circ F_{\nu_1}(x_1), \dots, F_{\mu_d}^{-1} \circ F_{\nu_d}(x_d)\big)^\top, \qquad x \in [0,1]^d,
\label{eq:quantile_transport}
\end{align}
where $F_{\bullet, i}(y) := \bullet_i([0, y])$ denotes the $i$-th marginal CDF of $\bullet \in \{\nu, \mu\}$.
Each component of $T$ depends only on a single input coordinate and is fully determined by the corresponding pair of one-dimensional marginals $(\nu_i, \mu_i)$.
The source marginals $F_{\nu_i}$ are known in closed form by assumption, while the target marginals $F_{\mu_i}$ are unknown and must be estimated from data.
Given i.i.d.\ samples $\chi_n := (X_1, \dots, X_n)$ with $X_j = (X_j^{(1)}, \dots, X_j^{(d)}) \sim \mu$, the marginal $\mu_i$ is estimated by the \emph{empirical CDF}
\begin{align}
\widehat F_{\mu_i, n}(y) := \frac{1}{n} \sum_{j=1}^n \mathds{1}\big\{X_j^{(i)} \leq y\big\}, \qquad y \in [0,1].
\label{eq:empirical_cdf}
\end{align}
Since the empirical CDF is a step function and not strictly increasing, its inverse is taken in the generalized sense as the \emph{empirical quantile function}
\begin{align}
\widehat F_{\mu_i, n}^{-1}(u) := \inf\!\left\{y \in [0,1] \,:\, \widehat F_{\mu_i, n}(y) \geq u\right\}, \qquad u \in (0, 1].
\label{eq:empirical_quantile}
\end{align}
The \emph{empirical quantile transport} is then defined coordinatewise by composing the empirical quantile of $\mu_i$ with the known source CDF $F_{\nu_i}$,
\begin{align}
\hat\Phi_n(x) := \big(\widehat F_{\mu_1, n}^{-1} \circ F_{\nu_1}(x_1), \dots, \widehat F_{\mu_d, n}^{-1} \circ F_{\nu_d}(x_d)\big)^\top, \qquad x \in [0,1]^d.
\label{eq:empirical_quantile_transport}
\end{align}
\begin{figure}[ht]
    \centering
    \includegraphics[width=0.85\linewidth]{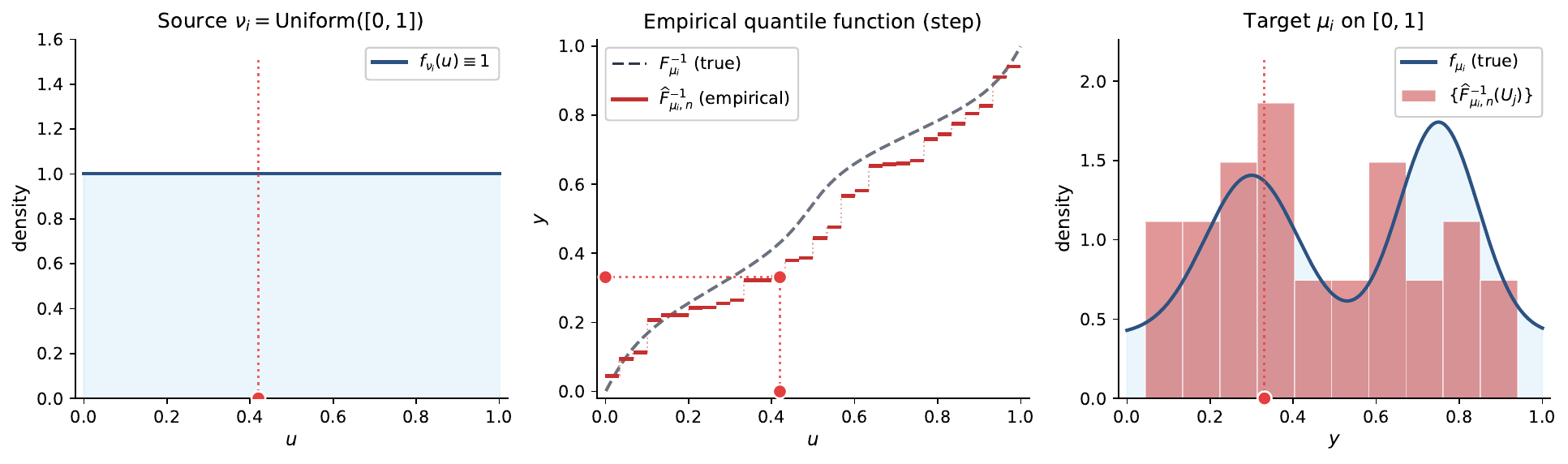}
    \caption{Empirical quantile transport in one coordinate. A uniform sample (left) is mapped through the empirical inverse CDF $\widehat F_{\mu_i, n}^{-1}$ (center, red), an approximation of $F_{\mu_i}^{-1}$ (dashed), onto the target $\mu_i$ (right). The full transport \eqref{eq:empirical_quantile_transport} applies this construction coordinatewise.}
    \label{fig:eq-transport}
\end{figure}
In practice, evaluating $\widehat F_{\mu_i, n}^{-1}$ amounts to looking up an order statistic of the samples $X_1^{(i)}, \dots, X_n^{(i)}$, cf. fig.~\ref{fig:eq-transport}. Sorting the marginal sample produces all empirical quantiles simultaneously in $\mathcal{O}(n \log n)$ operations per coordinate, after which any evaluation $\widehat F_{\mu_i, n}^{-1}(u)$ reduces to a binary search in $\mathcal{O}(\log n)$.
\section{PAC Learnability of Sparse Grid Integration using NeuralODE}
\label{sec:general_regime}
In this section we treat the general regime, in which the target $\mu$ is an arbitrary distribution on $[0,1]^d$.
\subsection{Hypothesis Space and Estimator}
\label{subsubsec:general_hypothesis}
\begin{assumption}
\label{ass:general_regime}
Fix $d, k \in \N$. The source $\nu$ and target $\mu$ in $\mathcal{M}_1^+([0,1]^d)$ satisfy:
\begin{enumerate}
  \item[\textup{(A1)}] The source $\nu$ is a product measure,
    $f_\nu(x) = \prod_{j=1}^d f_{\nu_j}(x_j)$, with marginal densities
    $f_{\nu_j} \in C^k([0,1])$.
  \item[\textup{(A2)}] The target $\mu$ has density
    $f_\mu \in C^k([0,1]^d)$.
  \item[\textup{(A3)}] There exist constants
    $0 < \kappa \le \mathcal K < \infty$ with
    $\kappa \le f_\nu, f_\mu \le \mathcal K$ on $[0,1]^d$.
  \item[\textup{(A4)}] There exists a constant $M \ge 1$ with
    $\|f_\mu\|_{C^k([0,1]^d)} \le M$ and
    $\|f_{\nu_j}\|_{C^k([0,1])} \le M$ for all $j = 1, \dots, d$.
\end{enumerate}
\end{assumption}
We denote by $\mathcal{T}$ the class of target distributions $\mu \in \mathcal{M}_1^+([0,1]^d)$ for which Assumption~\ref{ass:general_regime} holds, with $\nu$ and the constants $(\kappa, \mathcal{K}, M)$ fixed. The idea is to approximate a transport $\Phi$ by the time-one flow of a neuralODE whose vector field is parameterized by a $\mathrm{ReLU}^s$-network of activation order $s \geq 2$, ensuring that the resulting flow lies in $C^{s-1}$. 
We follow the construction of \cite[Sections 4.1--4.3]{marzouk2023}.
For integers $d_1, d_2, L, W, S \in \N$ and $B \geq 1$, let $\Phi_s^{d_1, d_2}(L, W, S, B)$ denote the class of fully connected $\mathrm{ReLU}^s$-networks $f_{\mathrm{NN}} \colon [0,1]^{d_1} \to \R^{d_2}$ of depth at most $L$, width at most $W$, with at most $S$ non-zero weights and biases, each of absolute value at most $B$ (full definition in Appendix~\ref{app:relus}). To enforce that the flow $\Phi_t^\theta$ preserves the cube $[0,1]^d$ across the entire time interval $[0,1]$, we impose a no-flux boundary condition on the vector field $v^\theta$. Following \cite[Definition 4.7]{marzouk2023}, this is achieved by multiplying the network output coordinatewise with the cut-off function $\chi_d(x) := (x_1(1-x_1), \dots, x_d(1-x_d))^\top$, whose $i$-th component vanishes precisely when $x_i \in \{0, 1\}$. Setting $\Omega := [0,1]^d \times [0,1]$, the resulting ansatz class is
\begin{equation*}
\mathcal{F}^{\mathrm{ansatz}}_s(L, W, S, B) := \bigl\{(x, t) \mapsto f_{\mathrm{NN}}(x, t) \odot \chi_d(x) \,:\, f_{\mathrm{NN}} \in \Phi_s^{d+1, d}(L, W, S, B)\bigr\},
\label{eq:ansatz_main}
\end{equation*}
where $\odot$ denotes coordinatewise multiplication. By construction, every element of $\mathcal{F}^{\mathrm{ansatz}}_s(L, W, S, B)$ satisfies the no-flux condition $v_i(x, t) = 0$ whenever $x_i \in \{0, 1\}$, for all $i = 1, \dots, d$ and all $t \in [0,1]$, and lies in $C^{s-1}(\Omega; \R^d)$. In particular, $[0,1]^d$ is invariant under the flow $\Phi_t^\theta$, and the closed-domain regularity $\Phi^\theta \in C^{s-1}([0,1]^d; [0,1]^d)$ follows by differentiating the flow ODE~\eqref{eq:flow_ode} on the closed cube, where $v^\theta$ is $C^{s-1}$.
To control the smoothness of the resulting flow $\Phi^\theta$ uniformly across the parameter space, we follow \cite{marzouk2023,marzouk2025} and additionally truncate the vector field in the $C^{s-1}$-norm. For a truncation radius $r > 0$, the hypothesis class is
\begin{equation}
\mathcal{F}_s^{L, W, S, B, r} := \mathcal{F}^{\mathrm{ansatz}}_s(L, W, S, B) \cap \bigl\{f \in C^{s-1}(\Omega) \,:\, \|f\|_{C^{s-1}(\Omega)} \leq r,\ \|f\|_{W^{2,\infty}(\Omega)} \leq r\bigr\}.
\label{eq:hypothesis_main}
\end{equation}
The joint $C^{s-1}$/$W^{2,\infty}$-bound on $v^\theta$ serves a twofold purpose. On the one hand, the explicit $W^{2,\infty}$-bound is required for the velocity-field class definition of \cite[Definition 4.7]{marzouk2023} and the $C^1$-stability theorems \cite[Theorem 4.4]{marzouk2023} feeding into the metric-entropy estimates underlying the statistical convergence theory in Appendix~\ref{app:relus}. For $s \geq 3$ this $W^{2,\infty}$-bound is implied by the $C^{s-1}$-bound via the trivial inclusion $C^{s-1}(\Omega) \subset C^{2}(\Omega) \subset W^{2,\infty}(\Omega)$, so the constraint is redundant; for $s = 2$ ($\mathrm{ReQU}$), networks in $\mathcal{F}_2^{\mathrm{ansatz}}$ are piecewise polynomial of degree depending on $L$, so their Hessians are piecewise
polynomials of bounded magnitude on compact domains for every choice of $(L, W, S, B)$. Thus, the $W^{2,\infty}$-radius~$r$ is an additional explicit truncation that is satisfiable for every choice of $(L, W, S, B)$ and absorbed into the approximation construction of Appendix~\ref{app:relus}. On the other hand, the $C^{s-1}$-bound propagates via Gronwall's inequality applied to successive derivatives of the flow ODE \eqref{eq:flow_ode} (cf.\ Theorem~\ref{thm:flow_smoothness}) to a uniform $C^{s-1}$-bound on the flow $\Phi^\theta$, which is needed to control $\|\qoi \circ \Phi^\theta\|_{C^{s-1}}$ uniformly in $\theta$.

The estimator is the empirical maximum likelihood estimator over $\mathcal{F}_s^{L, W, S, B, r}$ and samples $X_j \sim \mu$ in the sense of eq. \eqref{eq:neural_ode_log_likelihood}, i.e.
\begin{equation}
\hat\theta_n \in \arg\min_{\theta \in [-B, B]^S} \hat L_n(\theta, \chi_n) = \arg\min_{\theta \in [-B, B]^S}-\frac{1}{n} \sum_{j=1}^n \log f_{\Phi^\theta_*\nu}(X_j), \label{eq:erm_learner}
\end{equation}
with associated transport estimator $\Phi^{\hat\theta_n}$. Here we identify $\theta \in [-B,B]^S$ with elements of $\mathcal{F}_s^{L,W,S,B,r}$ via a fixed sparsity pattern of the $\mathrm{ReLU}^s$-network. We also note that the minimization in \eqref{eq:erm_learner} is implicitly restricted to those
$\theta \in [-B, B]^S$ whose corresponding network satisfies both
$\|f_{\mathrm{NN}}^\theta \odot \chi_d\|_{C^{s-1}(\Omega)} \le r$ and
$\|f_{\mathrm{NN}}^\theta \odot \chi_d\|_{W^{2,\infty}(\Omega)} \le r$. This feasible set is the intersection of the compact box $[-B,B]^S$ with two closed convex constraints. Together with continuity of $\theta \mapsto \hat L_n(\theta, \chi_n)$ this guarantees existence of a minimizer. Non-emptiness of the feasible set for the parameter choices of Theorem~\ref{thm:pac_learning_error_2} follows from the approximation construction in Appendix~\ref{app:relus}, specifically the verification at the end of the \emph{approximation error} paragraph in the proof of Theorem~\ref{thm:pac_learning_error_2}, where the approximating network $\hat v$ is shown to lie in $\mathcal{F}_s^{L,W,S,B,r}$ for $r = \mathcal{O}(1)$.
\subsection{Error Decomposition}
The total integration error $\varepsilon^{\mathrm{tot}}_{n,m}(\qoi)$ defined in \eqref{eq:total_error} combines two distinct sources of error. The \emph{learning error}, due to the discrepancy between the true target $\mu$ and the pushforward $\Phi^{\hat\theta_n}_*\nu$, and the \emph{quadrature error}, due to the finite budget of the sparse grid rule applied to the composed integrand $\qoi \circ \Phi^{\hat\theta_n}$.
Following Section~\ref{subsec:generative_learning} we choose the divergence $\mathcal{D}$ as the \emph{total variation distance}
\[
\TV(\mu, \Phi_* \nu) := \sup_{\|f\|_\infty \leq 1} \left| \int f \, \d\mu - \int f \, \d(\Phi_* \nu) \right|,
\]
which requires no regularity assumptions beyond boundedness of the integrand. With this dual definition one has $\TV(\mu, \nu) = \int |f_\mu - f_\nu|\,\d\lambda^d$, twice the standard probabilistic convention $\tfrac{1}{2}\int |f_\mu - f_\nu|\,\d\lambda^d$.
Stronger divergences such as the \emph{Kullback--Leibler divergence} or the (squared) \emph{Hellinger distance} dominate $\TV$; cf.~\cite{Polyanskiy_Wu_2025}, so any bound on the learning error in these divergences yields a corresponding bound in $\TV$. The following theorem makes the splitting between learning error and quadrature error precise.
\begin{theorem}[Decomposition of Total Error]
\label{thm:decomposition_of_total_error}
Let $d \in \N$, and let $\nu, \mu \in \mathcal{M}_1^+([0,1]^d)$. Let $\qoi \colon [0,1]^d \to \R$ be (essentially) bounded, and let $\Phi \colon [0,1]^d \to [0,1]^d$ be measurable. Given a quadrature rule $(w_j, \xi_j)_{j=1}^{m} \subset \R \times [0,1]^d$, the total error $\varepsilon^{\mathrm{tot}}_{n,m}(\qoi)$ satisfies the decomposition
\begin{align*}
\varepsilon^{\mathrm{tot}}_{n,m}(\qoi) \leq \|\qoi\|_\infty \,
\TV(\mu, \Phi_* \nu)
+ \bigg| \int_{[0,1]^d} \qoi(\Phi(x)) \, \d\nu(x) - \sum_{j = 1}^m w_j \, \qoi(\Phi(\xi_j)) \bigg|.
\end{align*}
\end{theorem}
\begin{proof}
The statement follows immediately from the triangle inequality
\begin{align*}
\varepsilon^{\mathrm{tot}}_{n,m}(\qoi)
&\leq \bigg| \int_{[0,1]^d} \qoi(x) \, \d\mu(x) - \int_{[0,1]^d} \qoi(x) \, \d(\Phi_* \nu)(x) \bigg| \\
&\hspace{2em}+ \bigg| \int_{[0,1]^d} \qoi(\Phi(x)) \, \d\nu(x) - \sum_{j = 1}^{m} w_j \, \qoi(\Phi(\xi_j)) \bigg|.
\end{align*}
\end{proof}
\subsection{Learning Error}
The control of the learning error $\TV(\mu, \Phi^{\hat\theta_n}_*\nu)$ rests on the statistical convergence theory for neuralODEs developed in \cite{marzouk2023}, which provides a PAC bound in the (squared) Hellinger distance; defined for absolutely continuous $\mu,\mu'$ with densities $f_\mu,f_{\mu'}$ by
\[
\mathcal{H}^2(\mu, \mu') := \tfrac{1}{2} \int \bigl(\sqrt{f_\mu} - \sqrt{f_{\mu'}}\bigr)^2 \, \d\lambda^d;
\]
for the case $s = 2$ of $\mathrm{ReQU}$ activations.
We extend this theory to general $\mathrm{ReLU}^s$ activations with $s \geq 2$ in Appendix~\ref{app:relus} by generalizing \cite[Theorem 4.8]{marzouk2023} to accommodate the higher-order derivatives induced by $\eta_s(y) = \max(y, 0)^s$.
\subsubsection{PAC Bound on Learning Error}
\label{subsubsec:learning_error_pac}
We summarize the resulting Hellinger PAC bound, with proof deferred to the appendix.
\begin{theorem}[PAC bound on the learning error in Hellinger; proof in Appendix~\ref{app:relus}]
\label{thm:pac_learning_error_2}
Let Assumption~\ref{ass:general_regime} hold, and fix $s, k \in \N$ with $k \geq 2$ and $s \in \{2, \ldots, k+1\}$. Set
\begin{equation}
\eta \;:=\; \frac{2(k - 1)}{d + 1 + 2(k - 1)}.
\label{eq:eta}
\end{equation}
There exist parameter choices $L = \mathcal{O}(1)$, $W = S = \mathcal{O}(n^{(d+1)/(d+1+2(k-1))})$, $B = \mathcal{O}(n^{1/(d+1+2(k-1))})$, $r = \mathcal{O}(1)$ in the hypothesis class $\mathcal{F}_s^{L, W, S, B, r}$ from Section~\ref{subsubsec:general_hypothesis}, and constants $C_1, C_2 > 0$ depending only on $(d, s, k, \kappa, \mathcal{K}, M)$, such that for all $n \in \N$, all $\delta \in (0, 1)$, and all $\mu \in \mathcal{T}$, the MLE $\hat\theta_n$ from \eqref{eq:erm_learner} satisfies
\begin{equation}
\mathbb{P}\!\Bigl(\,\mathcal{H}^2\!\bigl(\mu, \Phi^{\hat\theta_n}_*\nu\bigr)
\;\le\; C_1\,n^{-\eta} \log n
\;+\; C_2\,n^{-1} \log(1/\delta)\Bigr)
\;\ge\; 1 - \delta.
\label{eq:pac_hellinger}
\end{equation}
\end{theorem}
The Hellinger bound transfers to total variation via the standard inequality $\TV \leq 2\sqrt{2}\, \mathcal{H}$, cf.~\cite{Polyanskiy_Wu_2025}, yielding the following sample complexity for the learning error.
\begin{corollary}[PAC bound in total variation; sample complexity]
\label{cor:pac_tv_learning_error}
Under the assumptions of Theorem~\ref{thm:pac_learning_error_2}, with the same hypothesis class and constants $\tilde{C}_1 = 2\sqrt{2}\,\sqrt{C_1}$, $\tilde{C}_2 = 2\sqrt{2}\,\sqrt{C_2}$,
\[
\mathbb{P}\!\Bigl(\,\TV\!\bigl(\mu, \Phi^{\hat\theta_n}_*\nu\bigr)
\;\le\; \tilde{C}_1\, n^{-\eta/2} \sqrt{\log n}
\;+\; \tilde{C}_2\, \sqrt{\log(1/\delta)/n}\Bigr) \ge 1 - \delta.
\]
Equivalently, for all $\varepsilon, \delta \in (0, 1)$, there exists
\begin{equation}
n_{\mathrm{TV}}(\varepsilon, \delta)
\;=\; \mathcal{O}\!\left(\varepsilon^{-2/\eta} \bigl(\log(1/\varepsilon)\bigr)^{1/\eta} + \varepsilon^{-2} \log(1/\delta)\right)
\label{eq:companion_n_tv}
\end{equation}
samples such that $\TV(\mu, \Phi^{\hat\theta_n}_*\nu) \le \varepsilon$ holds for all $n \geq n_{\mathrm{TV}}(\varepsilon, \delta)$ with probability at least $1 - \delta$.
\end{corollary}
\begin{remark}[On the learning rate]
\label{rmk:minimax}
Up to log-factors, the learning-error rate $n^{-\eta/2}$ with $\eta = 2(k-1)/(d+1+2(k-1))$ matches the minimax rate for nonparametric estimation of a $(k-1)$-smooth density on a $(d+1)$-dimensional domain in Hellinger or $L^2$ loss. The appearance of $d+1$ rather than $d$ in the denominator reflects that the velocity field is time-dependent; see \cite[Remark 4.9]{marzouk2023} for a detailed discussion.
\end{remark}
\subsection{Main PAC Bound}
\label{subsec:general_pac_bound}
We now combine the error decomposition (Theorem~\ref{thm:decomposition_of_total_error}), the learning error bound (Corollary~\ref{cor:pac_tv_learning_error}), and the quadrature error bound (Theorem~\ref{thm:quadrature_error}) into a single PAC consistency statement in the general regime.
\begin{theorem}[PAC consistency of LtI in the general regime]
\label{thm:pac_lti_general}
Let Assumption~\ref{ass:general_regime} hold with $k \geq 2$, fix any activation order $s \in \{2, \ldots, k+1\}$, and let $\qoi \in C^{s-1}([0,1]^d; \R)$. Let $\hat\Phi_n := \Phi^{\hat\theta_n}$ denote the neuralODE transport estimator from Theorem~\ref{thm:pac_learning_error_2} and let $(\xi_j, w_j)_{j=1}^m$ denote the Clenshaw--Curtis sparse grid rule for $\nu$ from Section~\ref{subsec:sparse_grid}.
Then there exist constants $C_1, C_2, C_3 > 0$ depending only on $(d, s, k, \kappa, \mathcal{K}, M, \qoi)$ such that, for all $n, m \in \N$ and $\delta \in (0, 1)$, the total error satisfies, with probability at least $1 - \delta$,
\begin{equation*}
\begin{aligned}
\varepsilon^{\mathrm{tot}}_{n,m}(\qoi)
\;\leq\;
&\underbrace{C_1\, n^{-\eta/2} \sqrt{\log n} + C_2\, \sqrt{\log(1/\delta)/n}}_{\text{learning error}} \\
&+ \underbrace{C_3\, m^{-(s-1)/d} (\log m)^{(d-1)((s-1)/d + 1)}}_{\text{quadrature error}},
\end{aligned}
\label{eq:pac_lti_general}
\end{equation*}
where $\eta = 2(k-1)/(d+1+2(k-1))$ is given by \eqref{eq:eta}.
Equivalently, for all $\varepsilon, \delta \in (0, 1)$, there exist sample and budget thresholds
\begin{align*}
n(\varepsilon, \delta)
&= \mathcal{O}\!\bigl(\varepsilon^{-2/\eta} (\log(1/\varepsilon))^{1/\eta} + \varepsilon^{-2} \log(1/\delta)\bigr), \\
m(\varepsilon)
&= \mathcal{O}\!\bigl(\varepsilon^{-d/(s-1)} (\log(1/\varepsilon))^{(d-1) + d(d-1)/(s-1)}\bigr),
\end{align*}
such that $\varepsilon^{\mathrm{tot}}_{n,m}(\qoi) \leq \varepsilon$ holds for all $n \geq n(\varepsilon, \delta)$ and $m \geq m(\varepsilon)$ with probability at least $1 - \delta$.
\end{theorem}
\begin{proof}
By Theorem~\ref{thm:decomposition_of_total_error},
\begin{align*}
\varepsilon^{\mathrm{tot}}_{n,m}(\qoi)
&\leq \|\qoi\|_\infty \TV(\mu, \hat\Phi_{n,*}\nu) \\
&\quad + \bigg|\int_{[0,1]^d} \qoi(\hat\Phi_n(x)) \,\d\nu(x) - \sum_{j=1}^m w_j\, \qoi(\hat\Phi_n(\xi_j))\bigg|.
\end{align*}
The first term is bounded by Corollary~\ref{cor:pac_tv_learning_error} with probability at least $1 - \delta$, contributing the learning error of \eqref{eq:pac_lti_general}.
For the second term, we control $\|\qoi \circ \hat\Phi_n\|_{C^{s-1}}$ uniformly in $n$ in three steps.
\textit{(a) Vector field bound.} The truncation in the hypothesis class $\mathcal{F}_s^{L,W,S,B,r}$ defined in \eqref{eq:hypothesis_main} enforces $\|v^{\hat\theta_n}\|_{C^{s-1}(\Omega)} \leq r$ uniformly in $n$, with $r = \mathcal{O}(1)$ as ensured by Theorem~\ref{thm:pac_learning_error_2}.
\textit{(b) Flow bound.} The no-flux condition $v^{\hat\theta_n} \cdot \nu_x \equiv 0$
on $\partial[0,1]^d \times [0,1]$ enforced by the cut-off $\chi_d$ ensures that the
flow $\hat\Phi_n$ preserves $[0,1]^d$; trajectories starting in $[0,1]^d$ remain
in $[0,1]^d$, and regularity propagates from $v^{\hat\theta_n}$ to $\hat\Phi_n$ by
differentiating the flow ODE~\eqref{eq:flow_ode} in $x$ up to and including the
boundary, since $v^{\hat\theta_n}$ is $C^{s-1}$ on the closed domain $\Omega$. In
particular, by Theorem~\ref{thm:flow_smoothness}, $v^{\hat\theta_n} \in C^{s-1}(\Omega; \R^d)$
implies $\hat\Phi_n \in C^{s-1}([0,1]^d; [0,1]^d)$. Quantitatively, applying
\cite[Appendix~C, Lemma~7]{TrouveYounes2005} with $B = C^{s-1}([0,1]^d; \R^d)$,
$p = s-1$, and $T = 1$, the $C^{s-1}$-truncation $\|v^{\hat\theta_n}\|_{C^{s-1}(\Omega)} \leq r$
enforced by~\eqref{eq:hypothesis_main} propagates via Gronwall's inequality
applied to the variational equations satisfied by successive derivatives of the
flow to a uniform-in-$n$ bound
\begin{equation}\label{eq:flow_Cs_bound}
    \|\hat\Phi_n\|_{C^{s-1}([0,1]^d; [0,1]^d)} \;\leq\; C\, e^{C' r},
\end{equation}
with constants $C, C' > 0$ depending only on $s$ and $d$.
\textit{(c) Composition bound.} By Theorem~\ref{thm:iso_composition} (multivariate Faà di Bruno formula), the composition $\qoi \circ \hat\Phi_n$ satisfies
\[
\|\qoi \circ \hat\Phi_n\|_{C^{s-1}} \;\leq\; \dot c_{d, s-1}\, \|\qoi\|_{C^{s-1}}\, \bigl(1 + \|\hat\Phi_n\|_{C^{s-1}}\bigr)^{s-1},
\]
which is uniformly bounded in $n$ by step (b). Theorem~\ref{thm:quadrature_error} applied with $\sigma = \mathrm{iso}$ and smoothness index $s-1$ then yields the quadrature error of \eqref{eq:pac_lti_general}.
The sample and budget complexities follow by inverting each error term in $\varepsilon/2$ and $\delta$.
\end{proof}

\begin{remark}[Comparison with naive Monte Carlo]
\label{rmk:mc_vs_lti_general}
Comparing the asymptotic complexities $m(\varepsilon) = \mathcal{O}(\varepsilon^{-d/(s-1)})$ from Theorem~\ref{thm:pac_lti_general} and $n_{\mathrm{MC}}(\varepsilon, \delta) = \mathcal{O}(\varepsilon^{-2})$ from \eqref{eq:mc_rate}, LtI asymptotically outperforms Monte Carlo in $\qoi$-evaluations precisely when $s - 1 > d/2$. 
\end{remark}

\begin{remark}[Practical reach of the asymptotic rates]
\label{rem:constants}
The constants $C_1, C_2, C_3$ in Theorem~\ref{thm:pac_lti_general} are stated in
their asymptotic form. A closer look at the proof reveals three compounded
sources of dimensional dependence. The Faà di Bruno bound of
Theorem~\ref{thm:iso_composition} contributes a constant
combinatorial in $(d, s)$; the sparse-grid prefactor $c^{\mathrm{iso}}_{d, s-1}$
of Theorem~\ref{thm:quadrature_error} carries a factor of order $(d-1)!$;
and the iterated-Gronwall step~\eqref{eq:flow_Cs_bound} provides a constant that grows
exponentially in $d$ through successive variational equations. Together with
the $(\log m)^{(d-1)((s-1)/d + 1)}$ prefactor of the quadrature itself,
these factors imply that an asymptotic advantage over Monte Carlo
(Remark~\ref{rmk:mc_vs_lti_general}) may be practically relevant only for
moderate dimensions and sufficiently large activation order $s$; consistent also with \cite{bungartz2004sparse}.
\end{remark}

\section{PAC Learnability of Sparse Grid Integration via Empirical Quantiles}
\label{sec:diagonal_regime}
In this section we treat the diagonal regime, in which the target $\mu$ is a product measure on $[0,1]^d$.
By Proposition~\ref{prop:diagonal} together with the discussion in Section~\ref{subsubsec:diagonal_composition}, this is arguably the only regime in which the mixed-regularity rate $r_{\mathrm{mix}}(m, d, k)$ of Theorem~\ref{thm:quadrature_error} is available without further assumptions on $\qoi \in C_{\mathrm{mix}}^k$.
The transport reduces to the coordinatewise empirical quantile estimator $\hat\Phi_n$ from Section~\ref{subsec:quantile_estimator}.
\subsection{Assumptions}
\label{subsec:diagonal_assumptions}
We adopt the setup of Assumption~\ref{ass:general_regime}, with the additional structural restriction that $\mu$ is a product measure.
\begin{assumption}
\label{ass:diagonal_regime}
Fix $d, k \in \N$. The source $\nu$ and target $\mu$ in $\mathcal{M}_1^+([0,1]^d)$ satisfy:
\begin{enumerate}
  \item[\textup{(D1)}] Assumption~\ref{ass:general_regime} holds.
  \item[\textup{(D2)}] The target $\mu$ is a product measure, $\mu = \bigotimes_{i=1}^d \mu_i$, with marginal densities $f_{\mu_i} \in C^k([0,1])$.
\end{enumerate}
\end{assumption}
We denote by $\mathcal{T}_{\mathrm{diag}}$ the class of target distributions $\mu \in \mathcal{M}_1^+([0,1]^d)$ for which Assumption~\ref{ass:diagonal_regime} holds, with $\nu$ and the constants $(\kappa, \mathcal{K}, M)$ fixed. In particular, since $\mu$ is a product measure with $\|f_\mu\|_{C^k([0,1]^d)} \le M$ by (D1)+(A4), each marginal density satisfies $\|f_{\mu_i}\|_{C^k([0,1])} \le M\,\kappa^{-(d-1)}$, which we absorb into the class-dependent constants. The standard choice $\nu = \mathrm{Uniform}([0,1]^d)$ is recovered as the special case $f_{\nu_j} \equiv 1$, but Assumption~\ref{ass:diagonal_regime} accommodates any product source with $C^k$ marginal densities bounded away from zero.
\subsection{Error Decomposition}
\label{subsec:diagonal_decomposition}
The total-variation decomposition of Theorem~\ref{thm:decomposition_of_total_error} is not the natural starting point in the diagonal regime.
Unlike the neuralODE estimator of Section~\ref{subsec:neuralode}, the empirical quantile transport~$\hat\Phi_n$ in~\eqref{eq:empirical_quantile_transport} is a direct, non-parametric estimator constructed coordinatewise from order statistics.
In particular, $\hat\Phi_n$ is a step function in each coordinate, cf. fig.~\ref{fig:eq-transport}, and admits no pointwise smoothness, which precludes the statistical learning machinery underlying the TV bound of Corollary~\ref{cor:pac_tv_learning_error}.
We therefore work with an alternative decomposition, which directly exploits the uniform convergence of $\hat\Phi_n$ .
\begin{lemma}[Error decomposition for the diagonal regime]
\label{lem:split}
Let $\qoi \in C^1([0,1]^d)$, and let $T$ and $\hat\Phi_n$ denote the Knothe--Rosenblatt transport~\eqref{eq:quantile_transport} and the empirical quantile transport~\eqref{eq:empirical_quantile_transport}, respectively. Then
\begin{align}
\label{eq:split}
\varepsilon^{\mathrm{tot}}_{n,m}(\qoi) &\leq
\underbrace{\left|\int_{[0,1]^d} \qoi(T(z))\,\d\nu(z) - \sum_{j=1}^m w_j\, \qoi(T(\xi_j))\right|}_{%
\displaystyle =:\, \varepsilon^{\mathrm{quad}}}
\\
&\quad \;+\;
\underbrace{d^2 \cdot \|\qoi\|_{C^1} \cdot \|\hat\Phi_n - T\|_\infty \cdot \|Q_m\|_1}_{%
\displaystyle =:\, \varepsilon^{\mathrm{stat}}}\!,
\end{align}
where $\|Q_m\|_1 := \sum_{j=1}^m |w_j|$ is the $\ell^1$-norm of the quadrature weights.
\end{lemma}
\begin{proof}
By triangle inequality,
\begin{align*}
\varepsilon^{\mathrm{tot}}_{n,m}(\qoi)
&\leq \left|\int_{[0,1]^d} \qoi(T(z))\,\d\nu(z) - \sum_{j=1}^m w_j\, \qoi(T(\xi_j))\right|
\\
&\quad + \left|\sum_{j=1}^m w_j\, \bigl(\qoi(T(\xi_j)) - \qoi(\hat\Phi_n(\xi_j))\bigr)\right|.
\end{align*}
For the second term, apply the mean value to $\qoi$ on the segment between $a = T(\xi_j)$ and $b = \hat\Phi_n(\xi_j)$. Denoting by $\nabla$ the gradient, we obtain $|\nabla\qoi \cdot (a-b)| \le \|\nabla\qoi\|_\infty \|a-b\|_1 \le d\,\|\nabla\qoi\|_\infty \|a-b\|_\infty$,
\[
|\qoi(T(\xi_j)) - \qoi(\hat\Phi_n(\xi_j))|
\;\leq\; d\,\|\nabla \qoi\|_\infty \cdot \|T(\xi_j) - \hat\Phi_n(\xi_j)\|_\infty.
\]
Combining, we obtain
\begin{align*}
\varepsilon^{\mathrm{tot}}_{n,m}(\qoi)
&\leq \varepsilon^{\mathrm{quad}}
+ \sum_{j=1}^m |w_j| \cdot d\,\|\nabla \qoi\|_\infty \cdot \|T(\xi_j) - \hat\Phi_n(\xi_j)\|_\infty \\
&\leq \varepsilon^{\mathrm{quad}}
+ d^2 \,\|\qoi\|_{C^1} \cdot \|Q_m\|_1 \cdot \|\hat\Phi_n - T\|_\infty,
\end{align*}
where the last inequality uses $\|\nabla \qoi\|_\infty \leq d \|\qoi\|_{C^1}$.
\end{proof}
\subsection{Learning Error}
\label{subsec:diagonal_learning}
The empirical quantile transport $\hat\Phi_n$ is constructed coordinatewise from the empirical CDFs $\widehat F_{\mu_i, n}$ via \eqref{eq:empirical_quantile_transport}, so its uniform error decomposes into $d$ independent univariate problems.
For each marginal, the Dvoretzky--Kiefer--Wolfowitz inequality \cite{Dvoretzky1956-oq, Massart1990} provides a non-asymptotic concentration bound for the empirical CDF in the supremum norm.
Combining this with the Lipschitz continuity of the inverse marginal CDFs $F_{\mu_i}^{-1}$, which follows from the density lower bound $f_{\mu_i} \geq \kappa$ in Assumption~\ref{ass:diagonal_regime}, yields a PAC bound for the learning error of $\hat\Phi_n$.
\begin{lemma}[PAC bound on the learning error in $L^\infty$]
\label{lem:dkw}
Under Assumption~\ref{ass:diagonal_regime}, the empirical quantile transport $\hat\Phi_n$ from \eqref{eq:empirical_quantile_transport} satisfies, for all $t > 0$,
\begin{equation}
\mathbb{P}\!\left(\|\hat\Phi_n - T\|_\infty > t\right)
\;\leq\; 2 d \, \exp\!\left(-2 n \kappa^2 t^2\right).
\label{eq:dkw_pac}
\end{equation}
Equivalently, for all $\varepsilon, \delta \in (0, 1)$, there exists a sample size threshold
\begin{equation}
n_\infty(\varepsilon, \delta) \;=\; \left\lceil \frac{1}{2 \kappa^2 \varepsilon^2} \log\!\left(\frac{2d}{\delta}\right) \right\rceil
\;=\; \mathcal{O}\!\left(\varepsilon^{-2} \bigl(\log d + \log(1/\delta)\bigr)\right)
\label{eq:dkw_sample_complexity}
\end{equation}
such that $\|\hat\Phi_n - T\|_\infty \leq \varepsilon$ holds for all $n \geq n_\infty(\varepsilon, \delta)$ with probability at least $1 - \delta$.
\end{lemma}
\begin{proof}
For each $i \in \{1, \dots, d\}$, the Dvoretzky--Kiefer--Wolfowitz inequality \cite{Dvoretzky1956-oq, Massart1990} gives, for all $\tau > 0$,
\begin{equation}
\mathbb{P}\!\left(\|\widehat F_{\mu_i, n} - F_{\mu_i}\|_\infty > \tau\right)
\;\leq\; 2 \, e^{-2 n \tau^2}.
\label{eq:dkw_cdf}
\end{equation}
We translate this into a deviation of the empirical quantile transport.
\smallskip
\textit{Step (i): Reduction to inverse CDF deviation.}
By the definitions of $\hat\Phi_n$ and $T$,
\begin{equation}
|\hat\Phi_{n,i}(x_i) - T_i(x_i)| \;=\; \bigl|\widehat F_{\mu_i, n}^{-1}\!\bigl(F_{\nu_i}(x_i)\bigr) - F_{\mu_i}^{-1}\!\bigl(F_{\nu_i}(x_i)\bigr)\bigr|.
\label{eq:dkw_pointwise}
\end{equation}
By Assumption~\ref{ass:diagonal_regime}~(D1)+(A3), the marginal density $f_{\nu_i} \in C^k([0,1])$ is continuous and bounded below by $\kappa > 0$, so $F_{\nu_i}$ is continuous and strictly increasing; together with $F_{\nu_i}(0) = 0$, $F_{\nu_i}(1) = 1$ this yields $F_{\nu_i}([0,1]) = [0,1]$ bijectively. Substituting $u := F_{\nu_i}(x_i)$ in \eqref{eq:dkw_pointwise} and taking suprema gives
\begin{equation}
\|\hat\Phi_{n,i} - T_i\|_\infty
\;=\; \sup_{u \in [0,1]} \bigl|\widehat F_{\mu_i, n}^{-1}(u) - F_{\mu_i}^{-1}(u)\bigr|
\;=\; \|\widehat F_{\mu_i, n}^{-1} - F_{\mu_i}^{-1}\|_\infty.
\label{eq:dkw_reduction}
\end{equation}
The bound is therefore independent of the source marginal $\nu_i$.
\smallskip
\textit{Step (ii): Inverse CDF sandwich including the boundary.}
Set $\varepsilon_n := \|\widehat F_{\mu_i, n} - F_{\mu_i}\|_\infty$ and extend $F_{\mu_i}^{-1}$ to all of $\R$ by clamping, $F_{\mu_i}^{-1}(v) := 0$ for $v \leq 0$ and $F_{\mu_i}^{-1}(v) := 1$ for $v \geq 1$.
For any $u \in (0, 1]$, the pointwise inequality $\widehat F_{\mu_i, n}(x) \geq F_{\mu_i}(x) - \varepsilon_n$ for all $x \in [0,1]$ gives the inclusion
\[
\bigl\{x \in [0,1] : F_{\mu_i}(x) \geq u + \varepsilon_n\bigr\}
\;\subseteq\;
\bigl\{x \in [0,1] : \widehat F_{\mu_i, n}(x) \geq u\bigr\}.
\]
If the left-hand side is empty (i.e.\ $u + \varepsilon_n > 1$), then by the clamping convention $F_{\mu_i}^{-1}(\min(u+\varepsilon_n, 1)) = F_{\mu_i}^{-1}(1) = 1$, and $\widehat F_{\mu_i, n}^{-1}(u) \le 1$ holds trivially. Otherwise, taking infima yields $\widehat F_{\mu_i, n}^{-1}(u) \leq F_{\mu_i}^{-1}(\min(u + \varepsilon_n, 1))$, valid for all $u \in (0, 1]$. The symmetric inclusion based on $\widehat F_{\mu_i, n}(x) \leq F_{\mu_i}(x) + \varepsilon_n$ yields $\widehat F_{\mu_i, n}^{-1}(u) \geq F_{\mu_i}^{-1}(\max(u - \varepsilon_n, 0))$. Combining,
\begin{equation}
F_{\mu_i}^{-1}\!\bigl(\max(u - \varepsilon_n, 0)\bigr)
\;\leq\; \widehat F_{\mu_i, n}^{-1}(u)
\;\leq\; F_{\mu_i}^{-1}\!\bigl(\min(u + \varepsilon_n, 1)\bigr)
\qquad \forall\, u \in (0, 1].
\label{eq:quantile_sandwich}
\end{equation}
By Assumption~\ref{ass:diagonal_regime}~(D1)+(D2), the marginal density satisfies
\[
f_{\mu_i}(x) = \int_{[0,1]^{d-1}} f_\mu(x, x_{-i})\, \d x_{-i} \geq \kappa
\]
on $[0,1]$, where the bound uses $f_\mu \geq \kappa$ from~(A3) together with the unit volume of the integration domain. Hence $F_{\mu_i}^{-1}$ restricted to $[0, 1]$ is Lipschitz with constant $1/\kappa$, and the clamping extension preserves Lipschitz continuity on $\R$ (the extension is constant outside $[0,1]$ and matches the boundary values $F_{\mu_i}^{-1}(0) = 0$, $F_{\mu_i}^{-1}(1) = 1$). Applying the Lipschitz bound to \eqref{eq:quantile_sandwich},
\begin{align*}
F_{\mu_i}^{-1}\!\bigl(\min(u + \varepsilon_n, 1)\bigr) - F_{\mu_i}^{-1}(u) &\leq (\min(u + \varepsilon_n, 1) - u)/\kappa \leq \varepsilon_n/\kappa, \\
F_{\mu_i}^{-1}(u) - F_{\mu_i}^{-1}\!\bigl(\max(u - \varepsilon_n, 0)\bigr) &\leq (u - \max(u - \varepsilon_n, 0))/\kappa \leq \varepsilon_n/\kappa,
\end{align*}
and the supremum over $u \in (0, 1]$ yields
\begin{equation}
\|\widehat F_{\mu_i, n}^{-1} - F_{\mu_i}^{-1}\|_\infty \;\leq\; \frac{1}{\kappa}\, \|\widehat F_{\mu_i, n} - F_{\mu_i}\|_\infty.
\label{eq:cdf_to_quantile}
\end{equation}
\smallskip
\textit{Step (iii): Union bound.}
Combining \eqref{eq:dkw_reduction}, \eqref{eq:cdf_to_quantile} and substituting $\tau = \kappa t$ in \eqref{eq:dkw_cdf} gives the marginal bound $\mathbb{P}(\|\hat\Phi_{n,i} - T_i\|_\infty > t) \leq 2 e^{-2 n \kappa^2 t^2}$. The full bound \eqref{eq:dkw_pac} follows by a union bound over the $d$ marginals. The sample complexity \eqref{eq:dkw_sample_complexity} is obtained by inverting \eqref{eq:dkw_pac} in $t$.
\end{proof}
\subsection{Main PAC Bound}
\label{subsec:diagonal_pac_bound}
To combine the error decomposition of Lemma~\ref{lem:split} with the DKW-based learning error of Lemma~\ref{lem:dkw} and the quadrature error of Theorem~\ref{thm:quadrature_error}, we need a uniform-in-$n$ bound on the $\ell^1$-norm of the Smolyak quadrature weights, also referred to as the \emph{stability constant} of the quadrature rule \cite{glaubnitz2020,bungartz2004sparse}, which we establish below.
\begin{lemma}[Quadrature norm of the CC-Smolyak operator]
\label{lem:quad_norm}
For each dimension $i = 1, \ldots, d$ and level $l \in \N$, let
$\{(w_{j,l}^{(i),\lambda}, \xi_{j,l}^{(i)})\}_{j=1}^{m_l}$ denote the
univariate Clenshaw--Curtis rule on $[0,1]$ for the Lebesgue measure with
closed nonlinear growth $m_1 = 1$, $m_l = 2^{l-1} + 1$ for $l > 1$.
Let $\mathcal{S}_q^{d,\lambda}$ denote the Smolyak operator built from these
weights via \eqref{eq:smolyak}, and let $\mathcal{S}_q^{d,\nu}$ be its
$\nu$-reweighted variant from Remark~\ref{rmk:cc_for_uniform}. Then
\begin{equation}
\label{eq:quad_norm_bound}
\|\mathcal{S}_q^{d,\lambda}\|_1 \;\leq\; \sum_{s=0}^{d-1} \binom{d-1}{s}\binom{q-1-s}{d-1},
\qquad
\|\mathcal{S}_q^{d,\nu}\|_1 \;\leq\; \|f_\nu\|_\infty\,\|\mathcal{S}_q^{d,\lambda}\|_1.
\end{equation}
In particular, the standard sparse grid asymptotic
$m = m(\ell+d, d) \simeq 2^\ell \ell^{d-1}/(d-1)!$ for fixed $d$ as
$\ell \to \infty$ \cite{bungartz2004sparse} yields $\ell = \mathcal{O}(\log m)$ and
\[
\|\mathcal{S}_q^{d,\lambda}\|_1,\; \|\mathcal{S}_q^{d,\nu}\|_1
\;=\; \mathcal{O}\bigl((\log m)^{d-1}\bigr)
\qquad (m \to \infty,\ d \text{ fixed}).
\]
\end{lemma}

\begin{proof}
The Lebesgue Clenshaw--Curtis weights are non-negative with
$\sum_{j=1}^{m_l} w_{j,l}^{(i),\lambda} = 1$ \cite{Imhof1963-lh}, so the
$\ell^1$-norm of each tensorized operator satisfies
\[
\|I_{\vk}^{d,\lambda}\|_1
\;=\; \sum_{j_1, \ldots, j_d} \prod_{i=1}^d w_{j_i, k_i}^{(i),\lambda}
\;=\; \prod_{i=1}^d \sum_{j_i} w_{j_i, k_i}^{(i),\lambda}
\;=\; 1.
\]
Substituting $s = q - |\vk|_1$ in \eqref{eq:smolyak} and applying the triangle
inequality,
\[
\|\mathcal{S}_q^{d,\lambda}\|_1
\;\leq\; \sum_{s=0}^{d-1} \binom{d-1}{s}
\sum_{\substack{\vk \in \N^d \\ |\vk| = q-s}}
\|I_{\vk}^{d,\lambda}\|_1
\;=\; \sum_{s=0}^{d-1} \binom{d-1}{s}\binom{q-1-s}{d-1},
\]
where the last equality uses the identity
$\#\{\vk \in \N_{\geq 1}^d : |\vk| = n\} = \binom{n-1}{d-1}$.
Since $\binom{q-1-s}{d-1}$ is non-increasing in $s$ and
$\sum_{s=0}^{d-1}\binom{d-1}{s} = 2^{d-1}$,
\[
\sum_{s=0}^{d-1}\binom{d-1}{s}\binom{q-1-s}{d-1}
\;\leq\; 2^{d-1}\binom{q-1}{d-1}
\;\leq\; \frac{2^{d-1}}{(d-1)!}\,(q-1)^{d-1}
\;=\; \mathcal{O}(q^{d-1}).
\]
The asymptotic $m \simeq 2^\ell \ell^{d-1}/(d-1)!$ yields
$\log m = \ell \log 2 + (d-1)\log\ell + \mathcal{O}(1)$, hence
$\ell, q = \mathcal{O}(\log m)$ and therefore
$\|\mathcal{S}_q^{d,\lambda}\|_1 = \mathcal{O}((\log m)^{d-1})$.

The $\nu$-bound follows from $|w_j^\nu| = |w_j^\lambda f_\nu(\xi_j)| \leq \|f_\nu\|_\infty\,|w_j^\lambda|$
applied termwise to the Smolyak sum.
\end{proof}
For the composed integrand $\qoi \circ T$ that the Smolyak rule will be applied to, we also need control of its mixed-regularity norm in terms of the mixed-regularity norm of $\qoi$ and the (isotropic) $C^k$-norms of the diagonal components $T_i$. The general isotropic composition bound of Theorem~\ref{thm:iso_composition} does not yield this, since composition with a general $C^k$-map does not preserve mixed regularity (cf.\ Proposition~\ref{prop:diagonal}). The following \emph{diagonal Faà di Bruno} bound makes the diagonal exception precise.
\begin{lemma}[Diagonal Faà di Bruno bound]
\label{lem:diag_fdb}
Let $d, k \in \N$, let $T \colon [0,1]^d \to [0,1]^d$ be a diagonal map, $T(x) = (T_1(x_1), \dots, T_d(x_d))$ with $T_i \in C^k([0,1])$, and let $\qoi \in C^k_{\mathrm{mix}}([0,1]^d; \R)$. Then $\qoi \circ T \in C^k_{\mathrm{mix}}([0,1]^d; \R)$ and there exists a constant $\bar c_{d,k} > 0$ depending only on $d$ and $k$ such that
\begin{equation}
\|\qoi \circ T\|_{C^k_{\mathrm{mix}}}
\;\leq\; \bar c_{d,k}\, \|\qoi\|_{C^k_{\mathrm{mix}}}\, \prod_{i=1}^d \bigl(1 + \|T_i\|_{C^k([0,1])}\bigr)^{k}.
\label{eq:diag_fdb}
\end{equation}
\end{lemma}
\begin{proof}
For any multi-index $\valpha \in \{0, \dots, k\}^d$, the chain rule applied coordinate-by-coordinate to the diagonal $T$ gives
\[
\partial^{\valpha}(\qoi \circ T)(x)
= \sum_{\vbeta \le \valpha} (\partial^{\vbeta} \qoi)(T(x))\, \prod_{i = 1}^d P_{\alpha_i, \beta_i}\!\bigl(T_i'(x_i), \dots, T_i^{(\alpha_i)}(x_i)\bigr),
\]
where each $P_{\alpha_i, \beta_i}$ is the (univariate) Bell polynomial of weight $\alpha_i$ in the derivatives of $T_i$ up to order $\alpha_i$, of total degree $\beta_i$ \cite[Corollary 12]{Ma2009-rj}. In particular, $\vbeta \le \valpha$ componentwise implies $|\vbeta|_{\mathrm{mix}} \le k$, so $(\partial^{\vbeta} \qoi) \in C^0$ is controlled by $\|\qoi\|_{C^k_{\mathrm{mix}}}$. The Bell polynomial $P_{\alpha_i, \beta_i}$ is a sum of finitely many monomials in $T_i', \dots, T_i^{(\alpha_i)}$ of total degree $\beta_i \le \alpha_i \le k$, hence bounded by $c_k (1 + \|T_i\|_{C^k})^{k}$ for a constant $c_k$ depending only on $k$. Taking the sup over $x$ and the max over $\valpha$ with $|\valpha|_{\mathrm{mix}} \le k$ yields \eqref{eq:diag_fdb} with $\bar c_{d,k}$ absorbing the combinatorial constants and the number of $\vbeta \le \valpha$.
\end{proof}
We are now in a position to combine all ingredients into the main PAC consistency result for the diagonal regime.
\begin{theorem}[PAC consistency of LtI in the diagonal regime]
\label{thm:pac_lti_diagonal}
Let Assumption~\ref{ass:diagonal_regime} hold with $k \geq 1$, and let $\qoi \in C^k_{\mathrm{mix}}([0,1]^d; \R)$. Let $\hat\Phi_n$ denote the empirical quantile transport from \eqref{eq:empirical_quantile_transport}, and let $(\xi_j, w_j)_{j=1}^m$ denote the Clenshaw--Curtis sparse grid rule for the product source $\nu$ from Section~\ref{subsec:sparse_grid}.
Then there exist constants $C_1, C_2 > 0$ depending only on $(d, k, \kappa, \mathcal{K}, M, \qoi)$ such that, for all $n, m \in \N$ and all $\delta \in (0, 1)$, with probability at least $1 - \delta$,
\begin{equation}
\begin{aligned}
\varepsilon^{\mathrm{tot}}_{n,m}(\qoi)
\;\leq\;
&\underbrace{C_1\, m^{-k} (\log m)^{(d-1)(k+1)}}_{\text{quadrature error}} \\
&+ \underbrace{C_2\, (\log m)^{d-1} \cdot n^{-1/2} \cdot \sqrt{\log(2d/\delta)}}_{\text{learning error}}.
\end{aligned}
\label{eq:pac_lti_diagonal}
\end{equation}
Equivalently, for all $\varepsilon, \delta \in (0, 1)$, there exist sample and budget thresholds
\[
n(\varepsilon, \delta, m) \;=\; \mathcal{O}\!\bigl(\varepsilon^{-2} (\log m)^{2(d-1)} \log(d/\delta)\bigr),
\ \
m(\varepsilon) \;=\; \mathcal{O}\!\bigl(\varepsilon^{-1/k} (\log(1/\varepsilon))^{(d-1)(k+1)/k}\bigr),
\]
such that $\varepsilon^{\mathrm{tot}}_{n,m}(\qoi) \leq \varepsilon$ holds for all $n \geq n(\varepsilon, \delta, m)$ and $m \geq m(\varepsilon)$ with probability at least $1 - \delta$.
\end{theorem}
\begin{proof}
By Lemma~\ref{lem:split}, applied to $\qoi \in C^1$ (which holds since $C^k_{\mathrm{mix}} \subset C^1$ for $k \geq 1$),
\[
\varepsilon^{\mathrm{tot}}_{n,m}(\qoi) \;\leq\; \varepsilon^{\mathrm{quad}} + d^2\,\|\qoi\|_{C^1} \cdot \|\hat\Phi_n - T\|_\infty \cdot \|Q_m\|_1.
\]
The Knothe--Rosenblatt transport $T$ is diagonal with components $T_i = F_{\mu_i}^{-1} \circ F_{\nu_i}$. By (D1) and (D2), both $f_{\nu_i}$ and $f_{\mu_i}$ lie in $C^k([0,1])$ and are bounded below by $\kappa$ on $[0,1]$, so $F_{\nu_i}, F_{\mu_i} \in C^{k+1}([0,1])$ are $C^{k+1}$-diffeomorphisms of $[0,1]$. Hence $T_i = F_{\mu_i}^{-1} \circ F_{\nu_i} \in C^{k+1}([0,1]) \subset C^k([0,1])$, with $C^k$-norm bounded in terms of $(k, \kappa, \|f_{\nu_i}\|_{C^k}, \|f_{\mu_i}\|_{C^k})$; see~\cite[Theorem 4.11]{marzouk2025}. Under Assumption~\ref{ass:diagonal_regime} we have $\|f_{\nu_i}\|_{C^k} \le M$ and $\|f_{\mu_i}\|_{C^k} \le M\kappa^{-(d-1)}$, so $\|T_i\|_{C^k}$ is bounded by a constant depending only on $(d, k, \kappa, M)$.

Since $\qoi \in C^k_{\mathrm{mix}}$ and $T$ is diagonal with $C^k$-components, Lemma~\ref{lem:diag_fdb} yields $\qoi \circ T \in C^k_{\mathrm{mix}}([0,1]^d; \R)$ with $\|\qoi \circ T\|_{C^k_{\mathrm{mix}}} \le \bar c_{d,k}\,\|\qoi\|_{C^k_{\mathrm{mix}}} \prod_i (1 + \|T_i\|_{C^k})^{k}$.
Theorem~\ref{thm:quadrature_error} applied with $\sigma = \mathrm{mix}$ then yields
\[
\varepsilon^{\mathrm{quad}} \;\leq\; C_1\, m^{-k} (\log m)^{(d-1)(k+1)},
\]
where $C_1$ absorbs $\|\qoi \circ T\|_{C^k_{\mathrm{mix}}}$.
For the statistical term, Lemma~\ref{lem:dkw} gives, with probability at least $1 - \delta$,
\[
\|\hat\Phi_n - T\|_\infty \;\leq\; \frac{1}{\kappa\sqrt{2 n}} \, \sqrt{\log(2d/\delta)},
\]
and Lemma~\ref{lem:quad_norm} gives $\|Q_m\|_1 \leq \tilde C \, (\log m)^{d-1}$ with $\tilde C$ depending only on $d$.
Combining these bounds yields \eqref{eq:pac_lti_diagonal} with $C_2 := d^2 \cdot \tilde C \cdot \|\qoi\|_{C^1} / (\kappa \sqrt{2})$.
The sample and budget complexities follow by inverting each error term in $\varepsilon/2$.
\end{proof}
\begin{remark}[Comparison with naive Monte Carlo]
\label{rmk:mc_vs_lti_diagonal}
The diagonal regime requires only $m(\varepsilon) = \mathcal{O}(\varepsilon^{-1/k} \cdot \mathrm{polylog})$ evaluations of $\qoi$, compared to $n_{\mathrm{MC}}(\varepsilon, \delta) = \mathcal{O}(\varepsilon^{-2})$ for naive Monte Carlo (Remark~\ref{rmk:mc_baseline}).
\end{remark}
\section{Numerical Illustration}
\label{sec:numerics}
In this section we numerically illustrate the consistency results
established in Sections~\ref{sec:general_regime}
and~\ref{sec:diagonal_regime}.
\subsection{One-Dimensional Comparison}
\label{subsec:numerics_genz_1d}
We begin with a one-dimensional illustration; cf. fig.~\ref{fig:genz_1d}. We compare the LtI scheme \eqref{eq:lti_estimator} with empirical quantile transport against naive Monte Carlo (Remark~\ref{rmk:mc_baseline}) on three Genz test integrands; \cite{Genz1984}. The oscillatory $\qoi_1 \in C^\infty$, the Gaussian peak $\qoi_4 \in C^\infty$, and the discontinuous $\qoi_6 \in L^1$. As target $\mu$ we use a uniform-plus-Gaussian-mixture density on $[0,1]$ in three configurations exhibiting different degrees of concentration. The closed-form reference values $\evdist{\mu}{\qoi}$ and the precise parametrizations of $\mu$ and the $\qoi_i$ are deferred to Appendix~\ref{app:genz_reference}.
Figure~\ref{fig:genz_1d} shows the median absolute integration error against the quadrature budget $m$ for sample sizes $n \in \{10^2, 10^3, 10^4, 10^5\}$. Empirical-quantile errors are reported as the median over $6$ independent runs of the empirical quantile transport; naive Monte Carlo errors are reported as the median over $80$ independent runs at each budget. The red curve labeled \emph{LtI (ideal $T$)} corresponds to the LtI scheme using the exact KR transport \eqref{eq:quantile_transport}.
The plots exhibit the two-stage error profile predicted by Theorem~\ref{thm:pac_lti_diagonal}. In the first stage, the quadrature error dominates and LtI tracks the ideal-transport curve at a rate substantially faster than the $1/\sqrt m$ slope of Monte Carlo. In the second stage, the learning error from \eqref{eq:dkw_pac} saturates at a sample-size-dependent plateau. Increasing $n$ pushes this plateau downward, consistent with the $1/\sqrt n$ rate of the DKW bound. For sufficiently large $n$, LtI outperforms naive Monte Carlo by several orders of magnitude across all tests. The advantage is most pronounced for the smooth integrands $\qoi_1$ and $\qoi_4$, which fall within the regularity hypotheses of Theorem~\ref{thm:pac_lti_diagonal}. The $\qoi_6$ lies outside the $C^k$ hypothesis of Theorem~\ref{thm:pac_lti_diagonal}, but its empirical advantage over MC might be explainable by Bounded Variation and Sobolev analogues of the sparse-grid theory \cite{bungartz2004sparse}.
\begin{figure}[ht]
\centering
\includegraphics[width=0.8\linewidth]{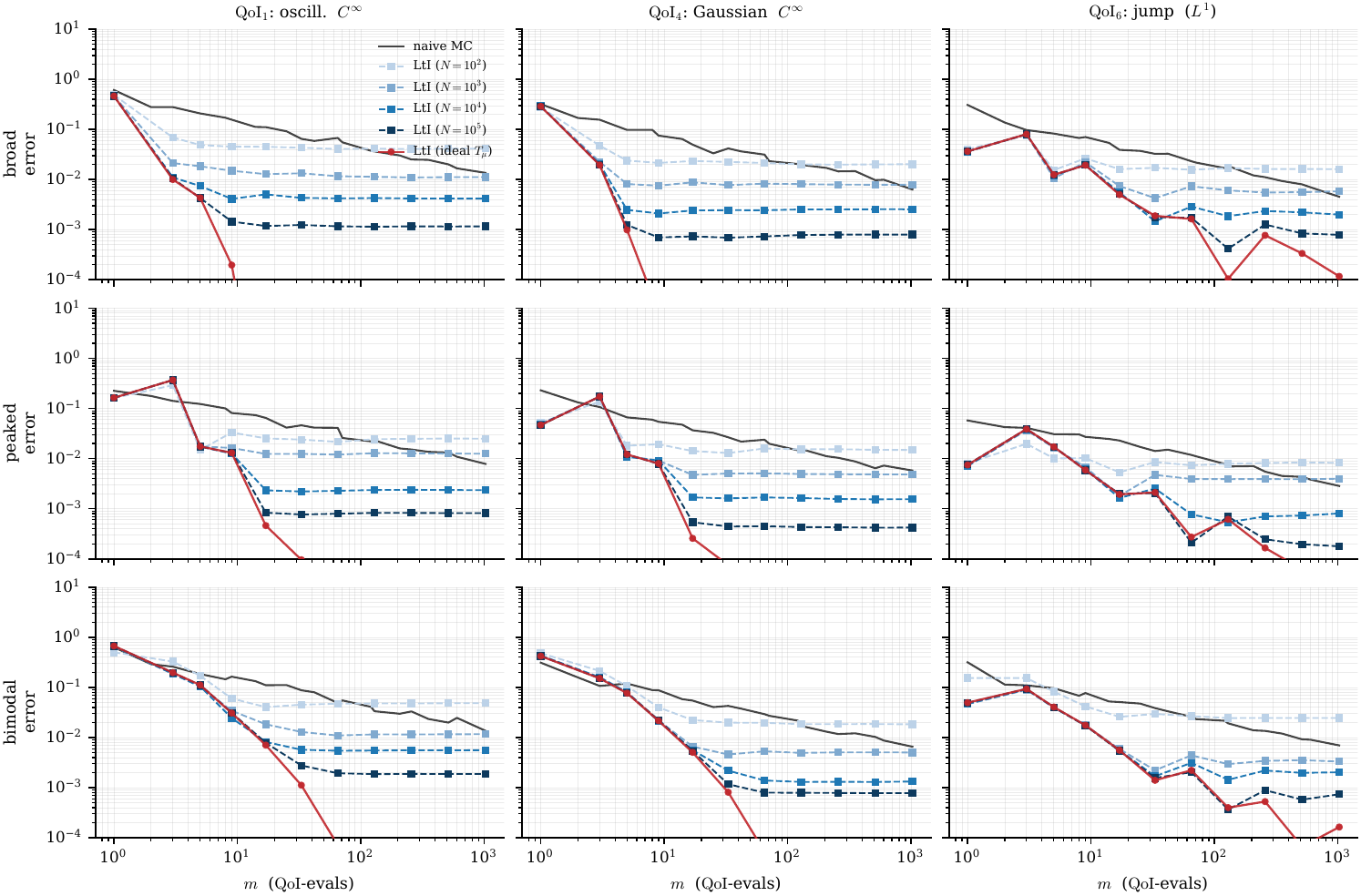}
\caption{One-dimensional illustration. Each row corresponds to a target distribution; each column to a Genz integrand. The horizontal axis is the quadrature budget $m$, the vertical axis is the median absolute integration error. The black curve is naive Monte Carlo at matching budget; blue curve are LtI with empirical quantile transport at different sample sizes; the red curve is LtI with the exact Knothe--Rosenblatt transport $T$ from~\eqref{eq:quantile_transport}.}
\label{fig:genz_1d}
\end{figure}
\subsection{Multi-Dimensional Integrands in the Diagonal Regime}
\label{subsec:numerics_genz_multid}
We now extend the illustration to higher dimensions $d \in \{2, 5, 10, 15\}$; cf. fig.~\ref{fig:genz_multid}. The target $\mu = \bigotimes_{j=1}^d \mu_j$ is the product of $d$ identical \emph{broad} marginals from Section~\ref{subsec:numerics_genz_1d}; see also \eqref{eq:mixture_marginal_density}. The empirical quantile transport $\hat\Phi_n$ is constructed coordinatewise as in~\eqref{eq:empirical_quantile_transport}.The integrands $\qoi_1$ and $\qoi_4$ are the canonical $d$-dimensional extensions of their 1D counterparts; see~\eqref{eq:fdef_dim} in Appendix~\ref{app:genz_reference} for the precise form and parameter scaling.
Figure~\ref{fig:genz_multid} shows the median integration error against the quadrature budget $m$. The qualitative picture of Section~\ref{subsec:numerics_genz_1d} persists. LtI outperforms Monte Carlo by several orders of magnitude as soon as $n$ is large enough to push the DKW plateau below the dominant quadrature error. The cross-over budget at which LtI overtakes naive Monte Carlo depends on both the dimension and the integrand. This is also consistent with the observation that sparse grid quadrature mitigates the curse of dimensionality only \emph{to some extent}; cf.\ \cite{bungartz2004sparse}. The asymptotic rate $m^{-k}$ improves on $m^{-1/2}$, but the cross-over depends on the problem and \emph{can} lie outside practical reach for higher dimensions.
\begin{figure}[ht]
\centering
\includegraphics[width=0.85\linewidth]{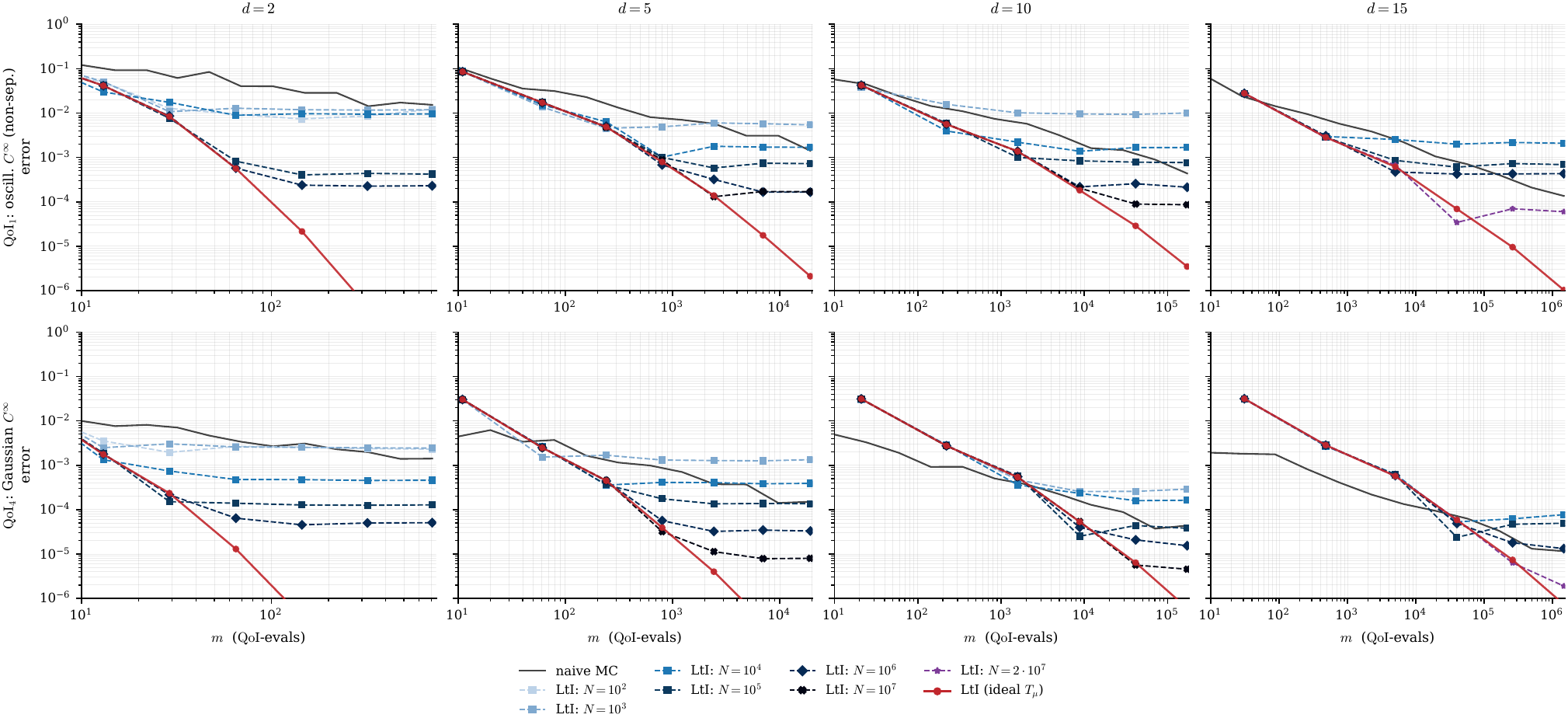}
\caption{Multi-dimensional illustration in the diagonal regime, $d \in \{2, 5, 10, 15\}$. Top row: oscillatory $\qoi_1 \in C^\infty$ (non-separable). Bottom row: Gaussian peak $\qoi_4 \in C^\infty$. Curves are colored as in Figure~\ref{fig:genz_1d}. Reference values are computed in closed form; see Appendix~\ref{app:genz_reference}.}
\label{fig:genz_multid}
\end{figure}

\subsection{End-to-End LtI on a Trained \texorpdfstring{$\mathrm{ReLU}^s$}{ReLUs} NeuralODE Flow}
\label{subsec:numerics_trained}
We conduct an end-to-end LtI experiment; cf. fig.~\ref{fig:trained_lti_2stage}. A $\mathrm{ReLU}^s$ neural ODE trained by maximum likelihood, then evaluated as a sparse-grid quadrature rule on the test integrands $\qoi_1, \qoi_4$. The target is a two-bump mixture on the diagonal of $[0,1]^2$ with uniform floor; see Appendix~\ref{app:training_setup_relus}. For each $s \in \{2, 3\}$ we train one flow, then push the Clenshaw--Curtis sparse-grid nodes through the time-1 flow and evaluate the LtI estimator.
\begin{figure}[ht]
\centering
\includegraphics[width=0.7\linewidth]{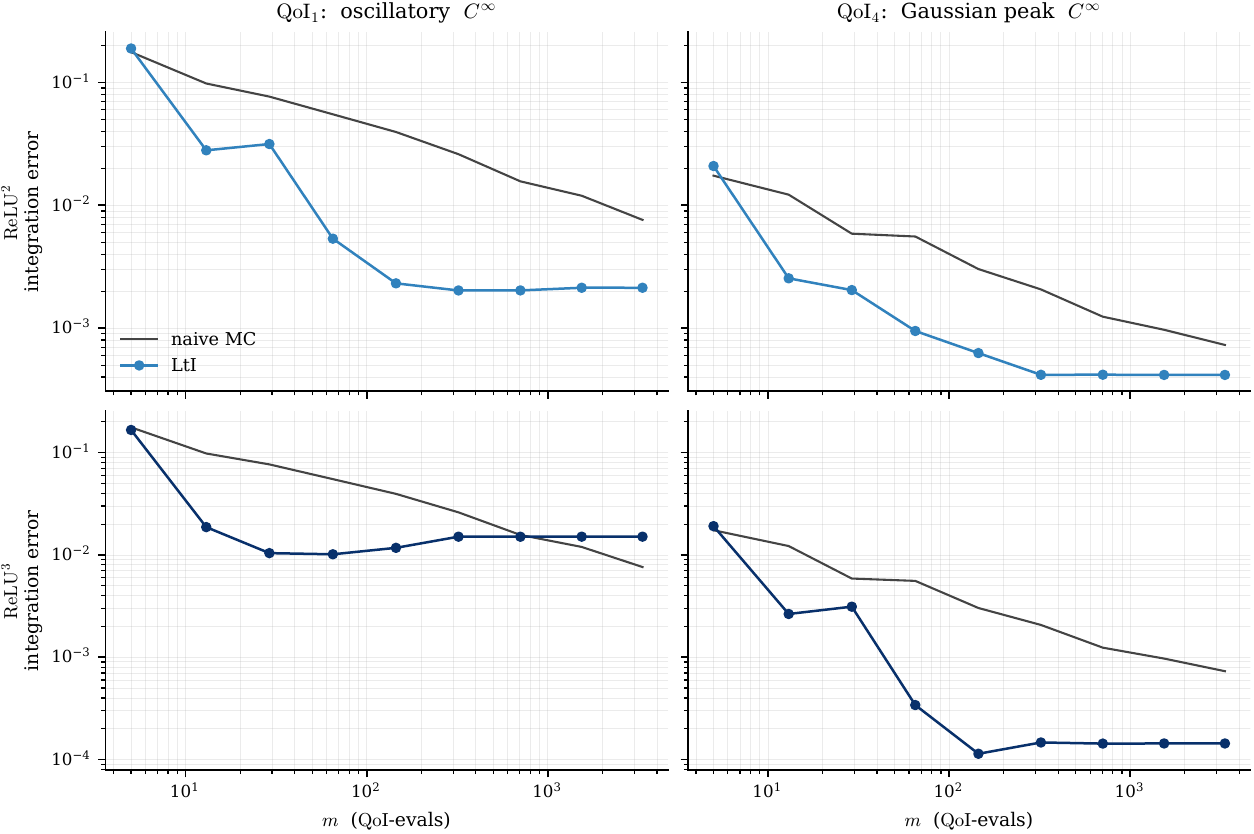}
\caption{LtI integration error on a $\mathrm{ReLU}^s$ neural ODE flow (blue) versus naive Monte Carlo at matching budget (grey), for activation orders $s \in \{2, 3\}$ (rows). All four panels exhibit the two-stage error profile predicted by Theorem~\ref{thm:pac_lti_general}.}
\label{fig:trained_lti_2stage}
\end{figure}
The plateau heights vary across panels because they depend on both the
(activation) regularity and the specific norms of $\qoi_k$ to which the trained
flow couples. In all cases the LtI scheme outperforms naive Monte Carlo
over the budget range. The quadrature term governs the early
regime, the learning (flow) term governs the late one. We caution that the
plateau reflects the practical difficulty of training $\mathrm{ReLU}^s$ vector fields. Throughout our experiments we observed that $\mathrm{ReLU}^s$-networks are harder to optimize as $s$ increases, in line e.g. with the findings of \cite{drygala2025learning}. The experiments
reported here already rely on several stabilization techniques to make
training tractable; see
Appendix~\ref{app:training_setup_relus} for details. Even so, we expect
better-tuned training pipelines or alternative architectures to lower the
plateau substantially. Moreover, we note that the activation regularity entering
Theorem~\ref{thm:pac_lti_general} is a worst-case bound on the
hypothesis class; characterizing the effective regularity of the specific
flow $\Phi^{\hat\theta_n}$ selected by maximum likelihood training is an
interesting direction, left for future work.

\subsection{Activation Regularity: \texorpdfstring{$\mathrm{ReLU}^s$}{ReLUs} Convergence Rates}
\label{subsec:numerics_relus}
We isolate the dependence of the Smolyak rate on the activation order $s$ on a controlled synthetic LtI experiment; cf. fig.~\ref{fig:synthetic_quadrature}. For $s \in \{2, 3, 5, 8\}$ and $32$ random seeds per $s$, we draw a one-hidden-layer $\mathrm{ReLU}^s$ vector field $v^\theta(x, t)$ on $[0,1]^2 \times [0,1]$ with the boundary cut-off $\chi_d$ from Section~\ref{subsec:neuralode}, integrate the time-1 flow $\Phi^\theta$ via RK4, and evaluate the LtI Smolyak quadrature error on the smooth test integrands $\qoi_1, \qoi_4$ composed with the flow. Full setup in Appendix~\ref{app:synthetic_quadrature_details}.
\begin{figure}[ht]
\centering
\includegraphics[width=0.85\linewidth]{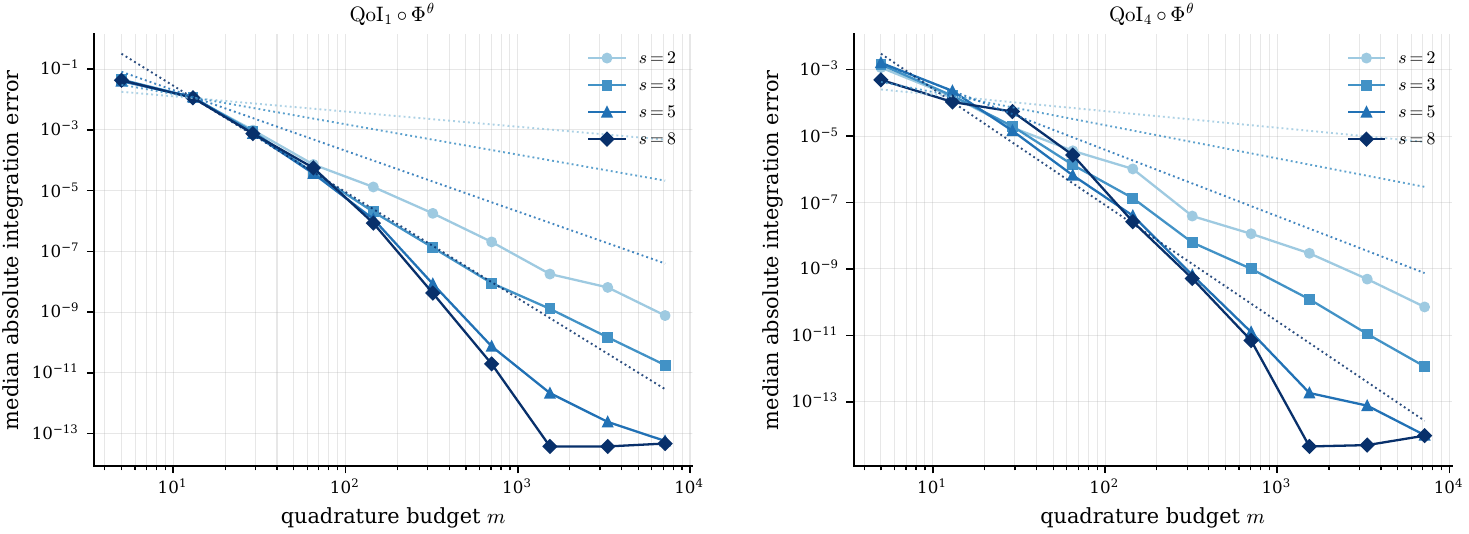}
\caption{Quadrature error of $\qoi_1 \circ \Phi^\theta$ over $32$ random $\mathrm{ReLU}^s$ neural-ODE flows per $s$. Dotted lines: rate $m^{-(s-1)/d}$ from Theorem~\ref{thm:pac_lti_general}. Empirical slopes are systematically steeper.}
\label{fig:synthetic_quadrature}
\end{figure}
Empirical slopes are systematically steeper than the rate predicted by Theorem~\ref{thm:pac_lti_general}. Any random $\mathrm{ReLU}^s$ network in fact carries higher Sobolev regularity than $C^{s-1}$, on which sparse-grid quadrature converges at the sharper effective rate $m^{-{s/d}}$, see~\cite{bungartz2004sparse}. A Sobolev-rate analysis of the LtI scheme reflecting this effect would tighten Theorem~\ref{thm:pac_lti_general}; we leave it as an open direction. The crossover at which higher $s$ overtakes lower $s$ is delayed, consistent with Remark~\ref{rem:constants}. 

\section{Discussion and Outlook}
\label{sec:Discussion_Outlook}
In this paper, we proved the consistency of the \emph{learning to integrate}
method on the $d$-dimensional unit cube in two structurally distinct regimes.
The split is not a matter of analytical convenience but is forced by a
rigidity statement of independent interest. Under composition with a
$C^1$-diffeomorphism, mixed regularity $C^k_{\mathrm{mix}}$ is preserved only
when the diffeomorphism is diagonal up to a permutation of coordinates
(Proposition~\ref{prop:diagonal}). The fast mixed-regularity sparse-grid rate
is therefore intrinsically unavailable for non-product targets, regardless of
the transport architecture, and the analysis splits accordingly.
In the \emph{general regime} of arbitrary targets, we combined the convergence
analysis for Clenshaw--Curtis sparse grids, universal approximation for
$\mathrm{ReLU}^s$-neural networks, and statistical learning theory for
neuralODE into a PAC-learning result at the isotropic rate. In the
\emph{diagonal regime} of product targets, we estimated each marginal CDF
empirically, recovering the full mixed-regularity rate via the
Dvoretzky--Kiefer--Wolfowitz inequality.

Despite these first results showing how the LtI method works in principle,
some open questions have to be left for future research. The setting
analyzed here differs in a few aspects from the numerical investigations
in \cite{ernst2025learning}, and bridging this gap remains a natural
direction. The illustrative experiments from Section~\ref{sec:numerics}
confirm the two-stage error profile predicted by
Theorems~\ref{thm:pac_lti_general} and~\ref{thm:pac_lti_diagonal}, but a
systematic benchmark in higher dimensions, across activation orders $s$,
and over different architectures remains future work.
A convergence analysis for sparse grid quadratures with respect to the normal
distribution has been conducted recently \cite{kazashi2025}. It would,
therefore, be of interest to extend our analysis to this case where
especially the approximation theory of neural networks has to be revised.
Also note that $\mathrm{ReLU}^s$-networks for $s > 2$ are not globally
Lipschitz, leading to questions concerning the existence of the
flow map $\Phi^\theta$ in this case.
The analysis we have given for the general regime is for neuralODE, while
other normalizing flow architectures show
good numerical properties as well \cite{chan2023lu, rochau2024new,ernst2025learning}. It would be of interest to repeat the
analysis given here for these architectures, as well as for flow matching
\cite{lipman2023flow} as an alternative training procedure that avoids the
cost of ODE integration during training.
Finally, the PAC bounds of Theorems~\ref{thm:pac_lti_general} and
\ref{thm:pac_lti_diagonal} assume $\qoi \in C^{s-1}$ and
$\qoi \in C^k_{\mathrm{mix}}$ respectively, whereas in many
uncertainty quantification applications the quantity of interest is itself
the output of a forward model with limited regularity. Quantifying how the
present rates degrade under reduced QoI smoothness, for instance by
combining the present analysis with bounded-variation or mixed-Sobolev
variants of the sparse-grid theory~\cite{bungartz2004sparse}, would broaden
the practical scope of LtI.

\paragraph{Acknowledgments}
The authors thank Oliver G.\ Ernst, Emily C.\ Erhardt, Toni Kowalewitz and Patrick Krüger for interesting discussions.
Emil Partow acknowledges support from the Munich Center for
Machine Learning (MCML).
Emil Partow acknowledges support from the Research Training
Group GRK~3081/1 (project number~534429653) of the Deutsche
Forschungsgemeinschaft (DFG, German Research Foundation).
The authors used large language models to polish written text for spelling, grammar, and style.
\bibliographystyle{siamplain}
\bibliography{slt-lti}

@article{Novak1996,
author={Novak, Erich
and Ritter, Klaus},
title={High dimensional integration of smooth functions over cubes},
journal={Numerische Mathematik},
year={1996},
month={11},
day={01},
volume={75},
number={1},
pages={79-97},
issn={0945-3245},
doi={10.1007/s002110050231}
}

@article{Suzuki2018AdaptivityOD,
  title={Adaptivity of deep ReLU network for learning in Besov and mixed smooth Besov spaces: optimal rate and curse of dimensionality},
  author={Taiji Suzuki},
  journal={ArXiv},
  year={2018},
  volume={abs/1810.08033},
  url={https://api.semanticscholar.org/CorpusID:53015027}
}

@inproceedings{Genz1984, author = {Genz, Alan}, title = {Testing multidimensional integration routines}, year = {1984}, isbn = {0444875700}, publisher = {Elsevier North-Holland, Inc.}, address = {USA}, booktitle = {Proc. of International Conference on Tools, Methods and Languages for Scientific and Engineering Computation}, pages = {81–94}, numpages = {14}, location = {Paris, France} }

@article{Clenshaw1960,
author={Clenshaw, C. W.
and Curtis, A. R.},
title={A method for numerical integration on an automatic computer},
journal={Numerische Mathematik},
year={1960},
month={1},
day={01},
volume={2},
number={1},
pages={197-205},
issn={0945-3245},
doi={10.1007/BF01386223}
}

@BOOK{Glasserman2010-yd,
  title     = "Monte Carlo Methods in Financial Engineering",
  author    = "Glasserman, Paul",
  publisher = "Springer",
  series    = "Stochastic Modelling and Applied Probability",
  month     =  oct,
  year      =  2010,
  address   = "New York, NY",
  language  = "en"
}

@article{TrouveYounes2005,
author = {Trouv\'{e}, Alain and Younes, Laurent},
title = {Local Geometry of Deformable Templates},
journal = {SIAM Journal on Mathematical Analysis},
volume = {37},
number = {1},
pages = {17-59},
year = {2005},
doi = {10.1137/S0036141002404838},
URL = {    
        https://doi.org/10.1137/S0036141002404838
},
eprint = { 
        https://doi.org/10.1137/S0036141002404838
}
}

@article{Stuart_2010, title={Inverse problems: A Bayesian perspective}, volume={19}, DOI={10.1017/S0962492910000061}, journal={Acta Numerica}, author={Stuart, A. M.}, year={2010}, pages={451–559}}

@article{sloan_dick_kuo,
author = {Dick, Josef and Kuo, Frances and Sloan, Ian},
year = {2013},
month = {05},
pages = {},
title = {High-dimensional integration: The quasi-Monte Carlo way},
volume = {22},
journal = {Acta Numerica},
doi = {10.1017/S0962492913000044}
}

@article{Caflisch_1998, title={Monte Carlo and quasi-Monte Carlo methods}, volume={7}, DOI={10.1017/S0962492900002804}, journal={Acta Numerica}, author={Caflisch, Russel E.}, year={1998}, pages={1–49}}

@ARTICLE{Imhof1963-lh,
  title     = "On the method for numerical integration of Clenshaw and Curtis",
  author    = "Imhof, J P",
  journal   = "Numer. Math. (Heidelb.)",
  publisher = "Springer Nature",
  volume    =  5,
  number    =  1,
  pages     = "138--141",
  month     =  dec,
  year      =  1963,
  language  = "en"
}

@article{Massart1990,
author = {P. Massart},
title = {{The Tight Constant in the Dvoretzky-Kiefer-Wolfowitz Inequality}},
volume = {18},
journal = {The Annals of Probability},
number = {3},
publisher = {Institute of Mathematical Statistics},
pages = {1269 -- 1283},
year = {1990},
doi = {10.1214/aop/1176990746},
URL = {https://doi.org/10.1214/aop/1176990746}
}

@ARTICLE{Dvoretzky1956-oq,
  title     = "Asymptotic minimax character of the sample distribution function
               and of the classical multinomial estimator",
  author    = "Dvoretzky, A and Kiefer, J and Wolfowitz, J",
  journal   = "Ann. Math. Stat.",
  publisher = "Institute of Mathematical Statistics",
  volume    =  27,
  number    =  3,
  pages     = "642--669",
  month     =  sep,
  year      =  1956
}

@article{Novak1999,
  author    = {Novak, Erich and Ritter, Klaus},
  title     = {Simple Cubature Formulas with High Polynomial Exactness},
  journal   = {Constructive Approximation},
  year      = {1999},
  volume    = {15},
  number    = {4},
  pages     = {499--522},
  issn      = {1432-0940},
  doi       = {10.1007/s003659900119}
}

@BOOK{Santambrogio2015-lp,
  title     = "Optimal transport for applied mathematicians",
  author    = "Santambrogio, Filippo",
  publisher = "Birkhauser",
  series    = "Progress in nonlinear differential equations and their
               applications",
  edition   =  1,
  month     =  oct,
  year      =  2015,
  address   = "Basel, Switzerland",
  copyright = "https://www.springernature.com/gp/researchers/text-and-data-mining",
  language  = "en"
}

@article{marzouk2023,
  author  = {Youssef Marzouk and Zhi (Robert) Ren and Sven Wang and Jakob Zech},
  title   = {Distribution Learning via Neural Differential Equations: A Nonparametric Statistical Perspective},
  journal = {Journal of Machine Learning Research},
  year    = {2024},
  volume  = {25},
  number  = {232},
  pages   = {1--61}
}

@book{Polyanskiy_Wu_2025, place={Cambridge}, title={Information Theory: From Coding to Learning}, publisher={Cambridge University Press}, author={Polyanskiy, Yury and Wu, Yihong}, year={2025}}

@book{hartmann2002,
author = {Hartman, Philip},
title = {Ordinary Differential Equations},
publisher = {Society for Industrial and Applied Mathematics},
year = {2002},
doi = {10.1137/1.9780898719222},
address = {},
edition   = {Second}
}

@misc{ehrhardt2025,
      title={Numerical and statistical analysis of NeuralODE with Runge-Kutta time integration}, 
      author={Emily C. Ehrhardt and Hanno Gottschalk and Tobias J. Riedlinger},
      year={2025},
      archivePrefix={arXiv},
      primaryClass={cs.LG},
      eprint={2503.10729}
}

@inproceedings{chen2019,
 author = {Chen, Ricky T. Q. and Rubanova, Yulia and Bettencourt, Jesse and Duvenaud, David K},
 booktitle = {Advances in Neural Information Processing Systems},
 editor = {S. Bengio and H. Wallach and H. Larochelle and K. Grauman and N. Cesa-Bianchi and R. Garnett},
 pages = {},
 publisher = {Curran Associates, Inc.},
 title = {Neural Ordinary Differential Equations},
 volume = {31},
 year = {2018}
}

@misc{marzouk2025,
      title={Distribution learning via neural differential equations: minimal energy regularization and approximation theory}, 
      author={Youssef Marzouk and Zhi Ren and Jakob Zech},
      year={2025},
      archivePrefix={arXiv},
      primaryClass={cs.LG},
      eprint={2502.03795} 
}

@article{Ma2009-rj,
  author    = {Ma, Tsoy-Wo},
  title     = {Higher Chain Formula Proved by Combinatorics},
  journal   = {Electronic Journal of Combinatorics},
  publisher = {The Electronic Journal of Combinatorics},
  volume    = {16},
  number    = {1},
  pages     = {21},
  month     = {6},
  year      = {2009}
}

@inproceedings{smolyak1963,
  author    = {Smolyak, Sergei Abramovich},
  title     = {Quadrature and Interpolation Formulas for Tensor Products of Certain Classes of Functions},
  booktitle = {Doklady Akademii Nauk},
  volume    = {148},
  pages     = {1042--1045},
  year      = {1963},
  organization = {Russian Academy of Sciences}
}

@article{wasilikowski1995,
  author    = {Wasilkowski, G. W. and Wozniakowski, H.},
  title     = {Explicit Cost Bounds of Algorithms for Multivariate Tensor Product Problems},
  journal   = {Journal of Complexity},
  year      = {1995},
  volume    = {11},
  number    = {1},
  pages     = {1--56},
  issn      = {0885-064X},
  doi       = {10.1006/jcom.1995.1001}
}

@article{belomestny2022,
title = {Simultaneous approximation of a smooth function and its derivatives by deep neural networks with piecewise-polynomial activations},
journal = {Neural Networks},
volume = {161},
pages = {242-253},
year = {2023},
issn = {0893-6080},
doi = {10.1016/j.neunet.2023.01.035},
author = {Denis Belomestny and Alexey Naumov and Nikita Puchkin and Sergey Samsonov}
}

@InProceedings{Novak1997,
author="Novak, Erich
and Ritter, Klaus",
editor="N{\"u}rnberger, G{\"u}nther
and Schmidt, Jochen W.
and Walz, Guido",
title="The Curse of Dimension and a Universal Method For Numerical Integration",
booktitle="Multivariate Approximation and Splines",
year="1997",
publisher="Birkh{\"a}user Basel",
address="Basel",
pages="177--187",
isbn="978-3-0348-8871-4"
}

@book{sullivan2015introduction,
  title={Introduction to uncertainty quantification},
  author={Sullivan, Timothy John},
  volume={63},
  year={2015},
  publisher={Springer}
}

@article{bungartz2004sparse,
  title={Sparse grids},
  author={Bungartz, Hans-Joachim and Griebel, Michael},
  journal={Acta numerica},
  volume={13},
  pages={147--269},
  year={2004},
  publisher={Cambridge University Press}
}

@book{davison1997bootstrap,
  title={Bootstrap methods and their application},
  author={Davison, Anthony Christopher and Hinkley, David Victor},
  year={1997},
  publisher={Cambridge university press}
}

@book{gamerman2006markov,
  title={Markov chain Monte Carlo: stochastic simulation for Bayesian inference},
  author={Gamerman, Dani and Lopes, Hedibert F},
  year={2006},
  publisher={Chapman and Hall/CRC}
}

@misc{ernst2025learning,
      title={Learning to Integrate}, 
      author={Oliver G. Ernst and Hanno Gottschalk and Toni Kowalewitz and Patrick Krüger},
      year={2025},
      eprint={2506.11801},
      archivePrefix={arXiv},
      primaryClass={math.NA}
}

@article{Hoeffding01031963,
author = {Wassily Hoeffding},
title = {Probability Inequalities for Sums of Bounded Random Variables},
journal = {Journal of the American Statistical Association},
volume = {58},
number = {301},
pages = {13--30},
year = {1963},
publisher = {ASA Website},
doi = {10.1080/01621459.1963.10500830}
}

@article{papamakarios2021normalizing,
  title={Normalizing flows for probabilistic modeling and inference},
  author={Papamakarios, George and Nalisnick, Eric and Rezende, Danilo Jimenez and Mohamed, Shakir and Lakshminarayanan, Balaji},
  journal={Journal of Machine Learning Research},
  volume={22},
  number={57},
  pages={1--64},
  year={2021}
}

@inproceedings{rezende2015variational,
  title={Variational inference with normalizing flows},
  author={Rezende, Danilo and Mohamed, Shakir},
  booktitle={International conference on machine learning},
  pages={1530--1538},
  year={2015},
  organization={PMLR}
}

@inproceedings{dinh2017density,
  title={Density estimation using Real NVP},
  author={Dinh, Laurent and Sohl-Dickstein, Jascha and Bengio, Samy},
  booktitle={International Conference on Learning Representations},
  year={2017}
}

@inproceedings{chan2023lu,
  title={Lu-net: Invertible neural networks based on matrix factorization},
  author={Chan, Robin and Penquitt, Sarina and Gottschalk, Hanno},
  booktitle={2023 International Joint Conference on Neural Networks (IJCNN)},
  pages={1--10},
  year={2023},
  organization={IEEE}
}

@misc{rochau2024new,
      title={New advances in universal approximation with neural networks of minimal width}, 
      author={Dennis Rochau and Robin Chan and Hanno Gottschalk},
      year={2024},
      eprint={2411.08735},
      archivePrefix={arXiv},
      primaryClass={cs.NE}
}

@inproceedings{lipman2023flow,
  title={Flow Matching for Generative Modeling},
  author={Lipman, Yaron and Chen, Ricky TQ and Ben-Hamu, Heli and Nickel, Maximilian and Le, Matt},
  booktitle={11th International Conference on Learning Representations, ICLR 2023},
  year={2023}
}

@book{shalev2014understanding,
  title={Understanding machine learning: From theory to algorithms},
  author={Shalev-Shwartz, Shai and Ben-David, Shai},
  year={2014},
  publisher={Cambridge university press}
}

@article{Sloan1978,
author={Sloan, Ian H.
and Smith, W. E.},
title={Product-integration with the Clenshaw-Curtis and related points},
journal={Numerische Mathematik},
year={1978},
month={12},
day={01},
volume={30},
number={4},
pages={415-428},
issn={0945-3245},
doi={10.1007/BF01398509}
}

@misc{drygala2025learning,
      title={Learning Brenier Potentials with Convex Generative Adversarial Neural Networks}, 
      author={Claudia Drygala and Hanno Gottschalk and Thomas Kruse and Ségolène Martin and Annika Mütze},
      year={2025},
      eprint={2504.19779},
      archivePrefix={arXiv},
      primaryClass={cs.LG}
}

@techreport{osti_1832293,
  author       = {Dalbey, Keith R. and Eldred, Michael S. and Geraci, Gianluca and Jakeman, John D. and Maupin, Kathryn A. and Monschke, Jason A. and Seidl, Daniel Thomas and Tran, Anh and Menhorn, Friedrich and Zeng, Xiaoshu},
  title        = {Dakota, A Multilevel Parallel Object-Oriented Framework for Design Optimization, Parameter Estimation, Uncertainty Quantification, and Sensitivity Analysis: Theory Manual (V.6.15)},
  institution  = {Sandia National Lab. (SNL-NM), Albuquerque, NM (United States)},
  doi          = {10.2172/1832293},
  place        = {United States},
  year         = {2021},
  month        = {11},
  note         = {Chapter 3: Stochastic Expansion Methods}
}

@article{Waldvogel2006,
author={Waldvogel, J{\"o}rg},
title={Fast Construction of the Fej{\'e}r and Clenshaw--Curtis Quadrature Rules},
journal={BIT Numerical Mathematics},
year={2006},
month={3},
day={01},
volume={46},
number={1},
pages={195-202},
issn={1572-9125},
doi={10.1007/s10543-006-0045-4}
}

@article{SOMMARIVA2013682,
title = {Fast construction of Fejér and Clenshaw–Curtis rules for general weight functions},
journal = {Computers \& Mathematics with Applications},
volume = {65},
number = {4},
pages = {682-693},
issn = {0898-1221},
doi = {10.1016/j.camwa.2012.12.004},
author = {Alvise Sommariva},
year={2013}
}

@article{glaubnitz2020,
author = {Glaubitz, Jan},
title = {Stable High Order Quadrature Rules for Scattered Data and General Weight Functions},
journal = {SIAM Journal on Numerical Analysis},
volume = {58},
number = {4},
pages = {2144-2164},
year = {2020},
doi = {10.1137/19M1257901},
URL = {    
        https://doi.org/10.1137/19M1257901
},
eprint = {  
        https://doi.org/10.1137/19M1257901
}
}

@misc{kazashi2025,
      title={Optimality of quasi-Monte Carlo methods and suboptimality of the sparse-grid Gauss--Hermite rule in Gaussian Sobolev spaces}, 
      author={Yoshihito Kazashi and Yuya Suzuki and Takashi Goda},
      year={2026},
      eprint={2509.18712},
      archivePrefix={arXiv},
      primaryClass={math.NA},
      url={https://arxiv.org/abs/2509.18712}, 
}
\newpage
\appendix
\section{Regularity Preserving Diffeomorphisms}
\label{app:diagonality_proof}
In this section we give a full proof to Proposition~\ref{prop:diagonal}, establishing that any $C^1$ diffeomorphism $\Phi \colon [0,1]^d \to [0,1]^d$ satisfying $f \circ \Phi \in C_{\mathrm{mix}}^k$ for all $f \in C_{\mathrm{mix}}^k$ is diagonal up to a permutation of coordinates.
\begin{lemma}\label{lem:bound}
Let $\Phi \colon [0,1]^d \to [0,1]^d$ be a $C^k$-diffeomorphism such
that $f \circ \Phi \in C^k_{\mathrm{mix}}([0,1]^d; \R)$ for every
$f \in C^k_{\mathrm{mix}}([0,1]^d; \R)$. Then there exists a constant
$C > 0$, depending only on $\Phi$, $k$, and $d$, such that
\[
\|f \circ \Phi\|_{C^k_{\mathrm{mix}}}
  \leq C \, \|f\|_{C^k_{\mathrm{mix}}}
  \quad\text{for all } f \in C^k_{\mathrm{mix}}([0,1]^d; \R).
\]
\end{lemma}
\begin{proof}
The composition operator
$T_\Phi \colon C^k_{\mathrm{mix}}([0,1]^d) \to C^k_{\mathrm{mix}}([0,1]^d)$,
$f \mapsto f \circ \Phi$, is linear, and $C^k_{\mathrm{mix}}([0,1]^d)$
is a Banach space under
$\|f\|_{C^k_{\mathrm{mix}}}
:= \max_{\valpha \in \{0, \dots, k\}^d} \|\partial^{\valpha} f\|_\infty$.
It suffices to verify that $T_\Phi$ has closed graph.
If $f_n \to f$ and $f_n \circ \Phi \to g$ in $C^k_{\mathrm{mix}}$,
then both convergences imply uniform convergence (since
$\|\cdot\|_\infty \leq \|\cdot\|_{C^k_{\mathrm{mix}}}$), so
$f_n \circ \Phi \to f \circ \Phi$ pointwise, hence $g = f \circ \Phi$.
The closed graph theorem yields the asserted bound.
\end{proof}
\begin{proof}[Proof of Proposition~\ref{prop:diagonal}]
We argue in dimension $d = 2$; the general case follows by selecting
arbitrary coordinate pairs as in Step~5 below.
\smallskip
\textit{Step 1: regularity of components.}
The coordinate projection $\pi_i(y) = y_i$ lies in
$C^k_{\mathrm{mix}}([0,1]^2; \R)$, with all mixed partial derivatives
of order $\geq 2$ vanishing identically. By hypothesis,
\[
\Phi_i = \pi_i \circ \Phi \in C^k_{\mathrm{mix}}([0,1]^2; \R)
\qquad \text{for } i \in \{1, 2\},
\]
so all mixed partial derivatives $\partial^{\vbeta} \Phi_i$ with
$\vbeta \in \{0, \dots, k\}^2$ exist and are bounded on $[0,1]^2$.
\smallskip
\textit{Step 2: norm bound.}
By Lemma~\ref{lem:bound}, there exists $C > 0$ such that
\begin{equation}
\label{eq:diag_proof_bound}
\|f \circ \Phi\|_{C^k_{\mathrm{mix}}}
\leq C\,\|f\|_{C^k_{\mathrm{mix}}}
\qquad \forall\, f \in C^k_{\mathrm{mix}}([0,1]^2; \R).
\end{equation}
For convenience, write $\partial_1 = \partial / \partial x_1$,
$\partial_2 = \partial / \partial x_2$.
\smallskip
\textit{Step 3: Faà di Bruno expansion.}
For $g \in C^{2k}([0,1]; \R)$ and the multi-index $\valpha = (k, k)$,
the multivariate Faà di Bruno formula \cite{Ma2009-rj} applied to the composition
$g \circ \Phi_1$ yields, using the regularity of $\Phi_1$ from Step~1,
\begin{equation}
\label{eq:diag_proof_fdb}
\partial_1^k \partial_2^k \big[ g \circ \Phi_1 \big](x)
= g^{(2k)}(\Phi_1(x))\,
  \big( \partial_1 \Phi_1(x) \big)^k \big( \partial_2 \Phi_1(x) \big)^k
+ R_g(x).
\end{equation}
The remainder $R_g$ collects all remaining terms with derivative order $\le 2k - 1$ on $g$. Grouping by the order of the outer derivative,
\[
R_g(x) = \sum_{\nu=1}^{2k-1} g^{(\nu)}(\Phi_1(x))\, \Psi_\nu(x),
\]
with each $\Psi_\nu$ a polynomial in the mixed derivatives
$\partial^{\vbeta} \Phi_1$ for $\vbeta \in \{0, \dots, k\}^2$ — in
particular bounded uniformly on $[0,1]^2$ by Step~1, with bound $M$
depending only on $\Phi$ and $k$.
\smallskip
\textit{Step 4: oscillating test functions.}
For $N \in \N$, set
\[
g_N^{\cos}(t) := N^{-k} \cos(N t),
\qquad
g_N^{\sin}(t) := N^{-k} \sin(N t).
\]
Both are $C^\infty$(hence in $C^{2k}([0,1])$, so the Faà di Bruno expansion
\eqref{eq:diag_proof_fdb} applies), with $q$-th derivative
$(g_N^{\cos})^{(q)}(t) = N^{q - k} \cos(N t + q \pi/2)$ and analogously
for $g_N^{\sin}$. Lifting to a function on $[0,1]^2$ via
$f_N^{\cos}(y) := g_N^{\cos}(y_1)$, the partial derivatives in
$y_2$-direction vanish identically, so
\[
\|f_N^{\cos}\|_{C^k_{\mathrm{mix}}}
= \|g_N^{\cos}\|_{C^k([0,1])}
= \max_{0 \leq q \leq k} N^{q - k}
\leq 1 \qquad \forall N \geq 1,
\]
and identically for $f_N^{\sin}$. By~\eqref{eq:diag_proof_bound},
\[
\big\| \partial_1^k \partial_2^k \big[ f_N^{\cos} \circ \Phi \big]
  \big\|_\infty
\leq \| f_N^{\cos} \circ \Phi \|_{C^k_{\mathrm{mix}}}
\leq C\,\| f_N^{\cos} \|_{C^k_{\mathrm{mix}}}
\leq C,
\]
and identically for $f_N^{\sin}$. Substituting $g = g_N^{\cos}$
into~\eqref{eq:diag_proof_fdb}, the leading term contributes
\[
(g_N^{\cos})^{(2k)}(\Phi_1(x))
  \big( \partial_1 \Phi_1(x) \big)^k \big( \partial_2 \Phi_1(x) \big)^k
= (-1)^k\, N^k\, \cos\big(N \Phi_1(x)\big)\, J(x),
\]
where we abbreviate $J(x) := (\partial_1 \Phi_1(x))^k\,
(\partial_2 \Phi_1(x))^k$. The remainder is bounded uniformly in $x$ and $N$, as
\[
| R_{g_N^{\cos}}(x) |
\leq \sum_{\nu=1}^{2k-1} N^{\nu - k}\,| \Psi_\nu(x) |
\leq M\,(2k - 1)\,N^{k - 1},
\]
since $\big| (g_N^{\cos})^{(\nu)}(t) \big| \leq N^{\nu - k}$ and the
largest exponent for $\nu \leq 2k - 1$ is $k - 1$.
Combining, we obtain pointwise on $[0,1]^2$:
\begin{align*}
\big| N^k\, J(x)\, \cos(N \Phi_1(x)) \big|
  &\leq C + M\,(2k - 1)\,N^{k - 1}, \\
\big| N^k\, J(x)\, \sin(N \Phi_1(x)) \big|
  &\leq C + M\,(2k - 1)\,N^{k - 1}.
\end{align*}
Squaring, adding, and using $\sin^2 + \cos^2 = 1$:
\[
N^{2k}\, J(x)^2
\leq 2 \big( C + M\,(2k - 1)\,N^{k - 1} \big)^2,
\]
which gives
\begin{equation}
\label{eq:diag_proof_J_bound}
J(x)^2 \leq 2\, \big( C\,N^{-k} + M\,(2k - 1)\,N^{-1} \big)^2.
\end{equation}
\smallskip
\textit{Step 5: vanishing of $J$ and conclusion.}
The right-hand side of~\eqref{eq:diag_proof_J_bound} tends to zero as
$N \to \infty$, while $J$ is independent of $N$. Hence $J \equiv 0$
on $[0,1]^2$, and since $k \geq 1$ this implies
\begin{equation}
\label{eq:diag_proof_row1}
(\partial_1 \Phi_1)(\partial_2 \Phi_1) \equiv 0
\qquad \text{on } [0,1]^2.
\end{equation}
Repeating Steps~3-4 with the test functions $g_N^{\cos}, g_N^{\sin}$
applied in the variable $y_2$ instead of $y_1$ yields analogously
\begin{equation}
\label{eq:diag_proof_row2}
(\partial_1 \Phi_2)(\partial_2 \Phi_2) \equiv 0
\qquad \text{on } [0,1]^2.
\end{equation}
By \eqref{eq:diag_proof_row1}-\eqref{eq:diag_proof_row2}, each row of
$D\Phi(x)$ contains at most one nonzero entry. Since $\Phi$ is a
diffeomorphism, $\det D\Phi(x) \neq 0$ for every $x \in [0,1]^2$, so
each row contains \emph{exactly} one nonzero entry. The set of points at which
$\partial_j \Phi_i$ is nonzero is open by continuity, and the disjoint
union over $j$ covers $[0,1]^2$, which is connected; hence the column
index of the nonzero entry is constant on $[0,1]^2$ for each row. The
resulting nonzero pattern is the support of a permutation matrix on
$\{1, 2\}$, i.e.\ either the identity (\emph{diagonal}: $\partial_2 \Phi_1
\equiv 0$ and $\partial_1 \Phi_2 \equiv 0$) or the transposition
(\emph{anti-diagonal}: $\partial_1 \Phi_1 \equiv 0$ and $\partial_2 \Phi_2
\equiv 0$). These two cases correspond exactly to the two elements of
the symmetric group $S_2$.
For arbitrary $d \geq 2$, the same argument applied to each pair
of coordinate indices $(\ell, r)$ with $\ell \neq r$ and to each
component $\Phi_m$ yields
$(\partial_\ell \Phi_m)(\partial_r \Phi_m) \equiv 0$ on $[0,1]^d$.
Hence each row of $D\Phi$ has at most one nonzero entry, and the
permutation structure follows from invertibility of $D\Phi$ and
connectedness of $[0,1]^d$.
\end{proof}
\section{Statistical Learning Theory for NeuralODE}
\label{app:relus}
In this appendix we collect the statistical learning machinery
underlying the learning-error bound of Theorem~\ref{thm:pac_learning_error_2}
in the general regime. The construction follows the framework of
Marzouk, Ren, Wang, and Zech~\cite{marzouk2023}, with a hypothesis space built
from $\mathrm{ReLU}^s$-networks.
Let $d \ge 1$ and let $\nu, \mu \in \mathcal M_1^+([0,1]^d)$ be probability measures on $[0,1]^d$ with
densities $f_\nu, f_\mu$ w.r.t.\ the Lebesgue measure.  The estimator is the
time-$1$ flow map $\Phi^\theta \colon [0,1]^d \to [0,1]^d$ of a neural
ODE whose vector field is a $\mathrm{ReLU}^s$-network with $s \ge 2$ and
$\eta_s(y) = \max(y, 0)^s \in C^{s-1}(\R)$.  For a measurable
$T \colon [0,1]^d \to [0,1]^d$ we write $T_*\nu$ for the push-forward.  The
Hellinger and total-variation distances are defined as
\[
\mathcal{H}^2(p,q) = \tfrac12 \int (\sqrt p - \sqrt q)^2\,\d x, \qquad
\TV(p,q) = \int |p-q|\,\d x,
\]
and satisfy $\TV(p,q) \le 2\sqrt 2\,\mathcal{H}(p,q)$ by Cauchy--Schwarz and the identity $(p-q) = (\sqrt{p} - \sqrt{q})(\sqrt{p} + \sqrt{q})$.
\subsection{Hypothesis Space and MLE}
\label{sec:companion_class}
Fix $s \ge 2$ and set $\eta_s(y) = \max(y, 0)^s \in C^{s-1}(\R)$, applied
coordinate-wise to vectors.  Following~\cite[Def.~4.1]{marzouk2023},
for integers $d_1, d_2, L, W, S \in \N$ and $B \geq 1$ let
$\Phi_s^{d_1, d_2}(L, W, S, B)$ denote the class of fully connected
networks $f_{\mathrm{NN}} \colon [0,1]^{d_1} \to \R^{d_2}$ of the form
\[
f_{\mathrm{NN}}(x)
= W^{(L)} \eta_s\!\bigl(W^{(L-1)} \eta_s(\cdots
\eta_s(W^{(1)} x + b^{(1)}) \cdots) + b^{(L-1)}\bigr) + b^{(L)},
\]
with depth $L$ (number of affine maps), weight matrices
$W^{(\ell)} \in \R^{d_\ell \times d_{\ell-1}}$ satisfying
$\max_\ell d_\ell \le W$, with at most $S$ non-zero entries across all
$(W^{(\ell)}, b^{(\ell)})$, each of absolute value at most $B$.  For
$1 \le \ell \le L$ we write $F_\ell$ for the network composed of the
first $\ell$ affine maps, so that $F_1(x) = W^{(1)} x + b^{(1)}$ and
$F_{\ell+1}(x) = W^{(\ell+1)} \eta_s(F_\ell(x)) + b^{(\ell+1)}$; the
output dimension of $F_\ell$ is the layer width $d_\ell \leq W$,
not necessarily $1$. We call such an $F_\ell$ an \emph{$\ell$-$\mathrm{ReLU}^s$ subnetwork}; throughout the appendix, all matrix are sup-norms and operator norms induced by the corresponding operator norm convention.
To enforce the no-flux boundary condition needed for the flow
$\Phi^\theta$ to preserve $[0,1]^d$, we follow~\cite[Def.~4.7]{marzouk2023}
and introduce the admissible class
\[
V := \bigl\{f \in C^1([0,1]^{d+1})\mid f_i(x, t) = 0 \ \forall i = 1, \dots, d \text{ and } (x, t) \in [0,1]^{d+1}; \, x_i \in \{0, 1\}\bigr\},
\]
together with the cut-off
$\chi_d(x) = \bigl(x_1(1 - x_1), \dots, x_d(1 - x_d)\bigr)^{\!\top}$ and the
$\mathrm{ReLU}^s$-ansatz
\begin{equation*}
\mathcal{F}^{\mathrm{ansatz}}_s(L, W, S, B)
:= \bigl\{(x, t) \mapsto f_{\mathrm{NN}}(x, t) \odot \chi_d(x) :
f_{\mathrm{NN}} \in \Phi_s^{d+1, d}(L, W, S, B)\bigr\},
\label{eq:companion_ansatz}
\end{equation*}
where $\odot$ denotes coordinate-wise multiplication.  By construction,
$\mathcal{F}^{\mathrm{ansatz}}_s \subset V$ and every element lies in
$C^{s-1}(\Omega, \R^d)$ with $\Omega := [0,1]^d \times [0,1]$.
For a $C^{s-1}$-truncation radius $r > 0$ we then define
\begin{equation*}
\mathcal{F}_s^{L, W, S, B, r}
:= \mathcal{F}^{\mathrm{ansatz}}_s(L, W, S, B) \cap
\bigl\{f \in C^{s-1}(\Omega) \;:\; \|f\|_{C^{s-1}(\Omega)} \le r,\ \|f\|_{W^{2,\infty}(\Omega)} \le r\bigr\}.
\label{eq:companion_hypothesis}
\end{equation*}
For $s \geq 3$, the trivial inclusion $C^{s-1}(\Omega) \subset C^2(\Omega) \subset W^{2,\infty}(\Omega)$ on the bounded Lipschitz domain $\Omega$ makes the $W^{2,\infty}$-constraint redundant; for $s = 2$ ($\mathrm{ReQU}$), every element of $\mathcal{F}^{\mathrm{ansatz}}_2(L, W, S, B)$ is piecewise polynomial in $x$ of degree depending on $L$ with Hessian bounded in terms of $(L, W, S, B)$ on compact domains, so the $W^{2,\infty}$-radius $r$ is a genuine additional truncation that is satisfiable for the approximation construction of Theorem~\ref{thm:pac_learning_error_2} with $r = \mathcal{O}(1)$. In both cases, \cite[Definition 4.7]{marzouk2023} is satisfied for every $s \ge 2$.
The estimator is an empirical MLE, given by
\[
\hat\theta_n \in \arg\min_{\theta \in [-B, B]^S} \hat L_n(\theta, \chi_n) = \arg\min_{\theta \in [-B, B]^S}-\frac{1}{n} \sum_{j=1}^n \log f_{\Phi^\theta_*\nu}(X_j),
\]
where $\theta \in [-B,B]^S$ enumerates the non-zero parameters of a fixed sparsity pattern in $\Phi_s^{d+1,d}(L,W,S,B)$, and the minimization is taken over all such patterns; cf.\ Theorem~\ref{thm:relu_s_entropy} below.
\subsection{Approximation and Metric Entropy for \texorpdfstring{$\mathrm{ReLU}^s$}{ReLU\textasciicircum s}}
\label{sec:companion_approx_entropy}
\subsubsection{Approximation Theory}
Marzouk, Ren, Wang, Zech state a universal approximation result for $\mathrm{ReLU}^{s}$-networks, based on spline constructions, cf.\ \cite{marzouk2023} that is already in the generality we need.
\begin{theorem}[$\mathrm{ReLU}^s$-approximation; {\cite[Cor.~4.5]{marzouk2023}}, re-indexed]
\label{thm:relu_s_approx}
Let $k, d_1, d_2$ and $s \in \N$ with $s \ge 2$ and $k \ge 1$.  There exists a
constant $C = C(d_1, d_2, k, s) > 0$ such that for every $f \in
C^k([0,1]^{d_1}; \R^{d_2})$ and every $N \in \N$ there is a network
$\tilde f \in \Phi_s^{d_1, d_2}(L, W, S, B)$ with $L \le C$, $W \le N$, $S \le
N$, $B \le C(\|f\|_{C^k} + N^{1/d_1})$, $\tilde f \in C^{s-1}$, and
\[
\|f_j - \tilde f_j\|_{W^{r,\infty}([0,1]^{d_1})}
\;\le\; C\,N^{-(k-r)/d_1}\,|f_j|_{C^k}
\qquad \forall\,r \in \{0, \dots, k\},\ j = 1, \dots, d_2.
\]
\end{theorem}
In Stage~1 of the main PAC bound, Theorem~\ref{thm:pac_learning_error_2}, we apply
Theorem~\ref{thm:relu_s_approx} with $r = 1$, $d_1 = d+1$, $d_2 = d$ to
the target velocity field $f^\mu \in C^k([0,1]^{d+1}, \R^d)$ corresponding to a flow $T^\mu$ with $T^\mu_*\nu = \mu$. On the bounded Lipschitz domain $[0,1]^{d_1}$, the Sobolev space $W^{r,\infty}$ is, for integer $r$, equivalent to the H\"older space $C^{r-1,1}$ and in particular controls the $C^r$-norm; we therefore freely interchange $W^{r,\infty}$ and $C^r$ bounds in the sequel up to dimension-only constants.
\subsubsection{Metric Entropy}
\label{rmk:universal_approx_generalization}
Although the approximation theory in \cite{marzouk2023} is stated for general $\mathrm{ReLU}^s$-networks, $s \geq 2$, the main statistical convergence result~\cite[Theorem 4.8]{marzouk2023} is only stated for the case $s = 2$ of $\mathrm{ReQU}$ networks.
We are however interested in a generalization
to the case of $\mathrm{ReLU}^s$ for general $s \geq 2$, since leveraging the smoothness of the underlying neural network architecture corresponds to leveraging the smoothness of resulting flows $\hat\Phi$, which will ultimately result in better convergence rates for numerical integration of the composed integrand $\qoi \circ \hat\Phi$, by Theorem~\ref{thm:quadrature_error}.
It turns out that it suffices to generalize the \emph{metric-entropy step} in the proof of \cite[Theorem 4.8]{marzouk2023} since the authors provide a general convergence result for ODE-MLEs, that covers the rest of the argument, cf.\ \cite[Theorem 2.2]{marzouk2023}.
The \emph{metric entropy step} is based on four auxiliary lemmas
(\cite{marzouk2023}~C.1--C.4); we repeat the same four steps with the appropriate
$s$-dependent constants.  Throughout, set
\[
M := W \vee B \vee d_1 \vee 1, \qquad
\beta_\ell := \frac{2(s^\ell - 1)}{s - 1}, \quad \ell = 1, \dots, L.
\]
We
record the elementary bounds, used repeatedly below,
\begin{gather}
\eta_s(y) \le |y|^s, \qquad
|\eta_s(y) - \eta_s(\tilde y)| \le s \max(|y|, |\tilde y|)^{s-1} |y - \tilde y|,
\label{eq:etas_lip}\\
|\eta_s'(y)| \le s |y|^{s-1}, \qquad
|\eta_s'(y) - \eta_s'(\tilde y)| \le s(s-1) \max(|y|, |\tilde y|)^{s-2} |y - \tilde y|.
\label{eq:etas_prime}
\end{gather}
\begin{lemma}[Uniform boundedness; cf.\ \cite{marzouk2023}~Lemma~C.1]
\label{lem:uniform_bd}
Let $s \ge 2$ and $L \ge 1$.  There exists a constant $C_L^{(s)} > 0$,
depending only on $L, s$, such that for every $1 \le \ell \le L$ and every
$F_\ell \in \Phi_{s, \ell}^{d_1, 1}(L, W, S, B)$,
\begin{equation}
\sup_{x \in [0,1]^{d_1}} \|F_\ell(x)\|_\infty
\;\le\; C_L^{(s)}\,M^{\beta_\ell}.
\label{eq:layer_bound}
\end{equation}
\end{lemma}
\begin{proof}
Set $a_\ell := \max(\sup_{x \in [0,1]^{d_1}} \|F_\ell(x)\|_\infty, 1)$ and
use $\eta_s(y) \le |y|^s$.  For $\ell = 1$, with $|x_j| \le 1$, entries of
$W^{(1)}$ bounded by $B$ and at most $d_1$ non-zero entries per row,
\[
\sup_x \|F_1(x)\|_\infty \le d_1 B + B \le 2 M^2,
\]
hence $a_1 \le 2 M^2$.  For $\ell \ge 1$, each row of $W^{(\ell+1)}$ has
at most $W$ non-zero entries bounded by $B$, so
\[
\sup_x \|F_{\ell+1}(x)\|_\infty
\;\le\; W B\,\|\eta_s(F_\ell)\|_\infty + B
\;\le\; M^2 a_\ell^s + M
\;\le\; 2 M^2 a_\ell^s,
\]
using $a_\ell, M \ge 1$ in the last step.  Taking logs and iterating the
affine recursion $\log a_{\ell+1} \le \log(2 M^2) + s \log a_\ell$ from the
base $\log a_1 \le \log(2 M^2)$ gives
\[
\log a_\ell
\;\le\; \log(2 M^2) \sum_{j=0}^{\ell-1} s^j
\;=\; \log(2 M^2)\,\frac{s^\ell - 1}{s - 1},
\]
whence $a_\ell \le (2 M^2)^{(s^\ell - 1)/(s-1)} = 2^{(s^\ell - 1)/(s-1)} M^{\beta_\ell}$,
which proves~\eqref{eq:layer_bound} with
$C_L^{(s)} := 2^{(s^L - 1)/(s-1)}$.
\end{proof}
\begin{lemma}[Uniform gradient bound; cf.\ \cite{marzouk2023}~Lemma~C.3]
\label{lem:uni_grad_bd}
Under the hypotheses of Lemma~\ref{lem:uniform_bd}, there exists a
constant $\tilde C_L^{(s)} > 0$, depending only on $L, s$, such that for
every $1 \le \ell \le L$ and every $F_\ell \in \Phi_{s, \ell}^{d_1, 1}(L, W, S, B)$,
\begin{equation}
\sup_{x \in [0,1]^{d_1}} \|\nabla F_\ell(x)\|_\infty
\;\le\; \tilde C_L^{(s)}\,M^{\beta_\ell}.
\label{eq:grad_bound}
\end{equation}
Here, $\|\nabla F_\ell(x)\|_\infty$ is interpreted as the operator $\infty$-norm $\max_i \sum_j |(\nabla F_\ell(x))_{ij}|$ for the inner-layer Jacobians, reducing to the vector sup-norm at the output layer where $d_2 = 1$.
\end{lemma}
\begin{proof}
For $\ell = 1$, $\nabla F_1 = W^{(1)}$ and
$\|\nabla F_1\|_\infty \le d_1 B \le M^2 = M^{\beta_1}$.  For $\ell \ge 2$,
the chain rule gives
\[
\nabla F_\ell(x)
= W^{(\ell)}\,\mathrm{diag}\bigl(\eta_s'(F_{\ell-1}(x))\bigr)\,\nabla F_{\ell-1}(x),
\]
and with $|\eta_s'(y)| \le s |y|^{s-1}$ from~\eqref{eq:etas_prime} together
with Lemma~\ref{lem:uniform_bd},
\begin{align*}
\|\nabla F_\ell(x)\|_\infty
&\;\le\; W B \cdot s \|F_{\ell-1}\|_\infty^{s-1} \|\nabla F_{\ell-1}(x)\|_\infty \\
&\;\le\; s M^2 \bigl(C_L^{(s)}\bigr)^{s-1} M^{(s-1)\beta_{\ell-1}}
\|\nabla F_{\ell-1}(x)\|_\infty.
\end{align*}
Iterating and telescoping exponents along
$\beta_\ell = 2 + s \beta_{\ell-1} = 2 + (s-1)\beta_{\ell-1} + \beta_{\ell-1}$ yields
\eqref{eq:grad_bound} with $\tilde C_L^{(s)} := s^L (C_L^{(s)})^{(s-1)L}$.
\end{proof}
\begin{lemma}[Parameter-to-function Lipschitz bound; cf.\ \cite{marzouk2023}~Lemmas~C.2 \& C.4]
\label{lem:param_lip}
Let $s \ge 2$ and $L = \mathcal{O}(1)$.  There exist an exponent $P_{L, s}
= \mathcal{O}(s^L)$ and a constant $A_L^{(s)} > 0$, depending only on $L, s$,
such that for any two networks
$F, \tilde F \in \Phi_s^{d_1, 1}(L, W, S, B)$ sharing the same sparsity
pattern and differing in parameter $\ell_\infty$-norm by at most
$\tau' \in (0, 1]$,
\begin{equation}
\|F - \tilde F\|_{C^1([0,1]^{d_1})}
\;\le\; A_L^{(s)}\,M^{P_{L, s}}\,\tau'.
\label{eq:param_lip}
\end{equation}
\end{lemma}
\begin{proof}
We unroll the proofs of~\cite[Lem.~C.2 \& C.4]{marzouk2023} layer by
layer, using the general activation bounds
\eqref{eq:etas_lip}--\eqref{eq:etas_prime} in place of the $s = 2$
estimates.  Throughout, set $M := W \vee B \vee d_1 \vee 1$ as in
Lemma~\ref{lem:uniform_bd}, and
\[
\Delta_\ell^{(0)}(x) := \|F_\ell(x) - \tilde F_\ell(x)\|_\infty,
\qquad
\Delta_\ell^{(1)}(x) := \|\nabla F_\ell(x) - \nabla \tilde F_\ell(x)\|_\infty.
\]
We prove by \emph{joint} induction on $\ell = 1, \dots, L$ that there
exist exponents $p_\ell$ with $p_1 = 1$ and $p_{\ell+1} = p_\ell + \beta_\ell(s+1)$, and constants $a_\ell > 0$
(depending only on $L, s$) with
\begin{equation}
\sup_{x \in [0,1]^{d_1}} \bigl(\Delta_\ell^{(0)}(x) + \Delta_\ell^{(1)}(x)\bigr)
\;\le\; a_\ell\,M^{p_\ell}\,\tau'.
\label{eq:joint_induction}
\end{equation}
This implies~\eqref{eq:param_lip} at level $\ell = L$.
\emph{Base case $\ell = 1$.}
$F_1 - \tilde F_1 = (W^{(1)} - \tilde W^{(1)}) x + (b^{(1)} - \tilde b^{(1)})$ and
$\nabla F_1 - \nabla \tilde F_1 = W^{(1)} - \tilde W^{(1)}$, so
$\Delta_1^{(0)} \le (d_1 + 1) \tau'$ and
$\Delta_1^{(1)} \le d_1 \tau'$.
\emph{Inductive step.}  Assume~\eqref{eq:joint_induction} at level
$\ell$.  Writing
$F_{\ell+1} = W^{(\ell+1)} \eta_s(F_\ell) + b^{(\ell+1)}$ for both networks
and subtracting,
\begin{equation*}
F_{\ell+1} - \tilde F_{\ell+1}
= W^{(\ell+1)} \bigl(\eta_s(F_\ell) - \eta_s(\tilde F_\ell)\bigr)
+ \bigl(W^{(\ell+1)} - \tilde W^{(\ell+1)}\bigr) \eta_s(\tilde F_\ell)
+ \bigl(b^{(\ell+1)} - \tilde b^{(\ell+1)}\bigr).
\label{eq:F_recursion}
\end{equation*}
The first summand is bounded by $W B \cdot s
\max(\|F_\ell\|_\infty, \|\tilde F_\ell\|_\infty)^{s-1}\,
\Delta_\ell^{(0)}$ via~\eqref{eq:etas_lip} and by
Lemma~\ref{lem:uniform_bd} hence by
$s\,M^2 (C_L^{(s)})^{s-1} M^{(s-1)\beta_\ell} \Delta_\ell^{(0)}$.  The
second summand is bounded by $W \tau' \|\eta_s(\tilde F_\ell)\|_\infty
\le W \tau' (C_L^{(s)})^s M^{s \beta_\ell}$, and the third by $\tau'$.
Adding the three contributions and invoking the induction hypothesis
on $\Delta_\ell^{(0)}$ yields
\begin{equation}
\Delta_{\ell+1}^{(0)}
\;\le\; a_\ell'\,M^{p_\ell + \beta_\ell(s+1)}\,\tau'
\label{eq:Delta0_recursion}
\end{equation}
for some $a_\ell' = \mathcal{O}_{L, s}(1)$.  Differentiating~\eqref{eq:F_recursion}
gives
\begin{align*}
\nabla F_{\ell+1} - \nabla \tilde F_{\ell+1}
&= W^{(\ell+1)} \bigl[\mathrm{diag}(\eta_s'(F_\ell)) \nabla F_\ell
- \mathrm{diag}(\eta_s'(\tilde F_\ell)) \nabla \tilde F_\ell\bigr] \\
&+ (W^{(\ell+1)} - \tilde W^{(\ell+1)}) \mathrm{diag}(\eta_s'(\tilde F_\ell))
\nabla \tilde F_\ell.
\end{align*}
Using the splitting
\begin{align*}
\eta_s'(F_\ell) \nabla F_\ell - \eta_s'(\tilde F_\ell) \nabla \tilde F_\ell
&= \eta_s'(F_\ell)(\nabla F_\ell - \nabla \tilde F_\ell)\\
&\quad + (\eta_s'(F_\ell) - \eta_s'(\tilde F_\ell)) \nabla \tilde F_\ell,
\end{align*}
we bound the first term by
$s\,\|F_\ell\|_\infty^{s-1} \Delta_\ell^{(1)}
\le s (C_L^{(s)})^{s-1} M^{(s-1)\beta_\ell}\,\Delta_\ell^{(1)}$
via~\eqref{eq:etas_prime} and Lemma~\ref{lem:uniform_bd}, and the
second by $s(s-1) \max(\|F_\ell\|, \|\tilde F_\ell\|)^{s-2}
\Delta_\ell^{(0)}\,\|\nabla \tilde F_\ell\|_\infty
\le s(s-1) (C_L^{(s)})^{s-2} \tilde C_L^{(s)}\,
M^{(s-2)\beta_\ell + \beta_\ell}\,\Delta_\ell^{(0)}$
via~\eqref{eq:etas_prime}, Lemma~\ref{lem:uniform_bd} and
Lemma~\ref{lem:uni_grad_bd}.  Multiplying by $\|W^{(\ell+1)}\| \le W B$
and adding the second term
$\|W^{(\ell+1)} - \tilde W^{(\ell+1)}\| \cdot s \|\tilde F_\ell\|^{s-1}
\|\nabla \tilde F_\ell\| \le W \tau' \cdot s (C_L^{(s)})^{s-1}
\tilde C_L^{(s)} M^{(s-1)\beta_\ell + \beta_\ell}$, we obtain via the
induction hypothesis on $\Delta_\ell^{(0)} + \Delta_\ell^{(1)}$
\[
\Delta_{\ell+1}^{(1)}
\;\le\; a_\ell''\,M^{p_\ell + \beta_\ell(s+1)}\,\tau'.
\]
Combining this with~\eqref{eq:Delta0_recursion} gives
\eqref{eq:joint_induction} at level $\ell + 1$ with
$a_{\ell+1} := a_\ell' + a_\ell''$ and $p_{\ell+1} := p_\ell + \beta_\ell(s+1)$. Since
$\beta_\ell = 2(s^\ell - 1)/(s-1) = \mathcal{O}(s^\ell)$, summing the recursion yields
$p_L = p_1 + \sum_{\ell=1}^{L-1} \beta_\ell(s+1) = \mathcal{O}(s^L)$, completing the induction and the proof.
\end{proof}
With the three lemmas in place, the $C^1$-metric entropy of the
network class follows by the standard covering-number argument of
\cite[Thm.~4.2]{marzouk2023}.
\begin{theorem}[$C^1$-metric entropy of $\mathrm{ReLU}^s$-networks; cf. {\cite[Thm.~4.2]{marzouk2023}}]
\label{thm:relu_s_entropy}
Let $s \ge 2$ be fixed and $d_1, d_2 \ge 1$.  Consider
$\Phi_s^{d_1, d_2}(L, W, S, B)$ with $L = \mathcal{O}(1)$, $W = \mathcal{O}(N)$, $S = \mathcal{O}(N)$,
$B = \mathcal{O}(N)$.  Then there exists a constant $c_{L, s, d_1, d_2} > 0$ depending only on $L, s, d_1, d_2$, such that
for every $\tau \in (0, 1]$,
\begin{equation}
H\!\bigl(\Phi_s^{d_1, d_2}(L, W, S, B),\,C^1([0,1]^{d_1}),\,\tau\bigr)
\;\le\; c_{L, s, d_1, d_2} \bigl(N \log(\tau^{-1}) + N \log N\bigr),
\label{eq:entropy_bd}
\end{equation}
where $H(\cdot, \cdot, \tau) := \log \mathcal N(\tau, \cdot, \cdot)$ denotes
the logarithm of the $\tau$-covering number.
\end{theorem}
\begin{proof}
Without loss of generality $d_2 = 1$; the general case follows by the
tensorization argument of~\cite[Cor.~4.3]{marzouk2023}.  Fix a
sparsity pattern of $\Phi_s^{d_1, 1}(L, W, S, B)$; the number of sparsity
patterns is bounded by $\binom{(W+1)^L}{S} \le (W+1)^{L S}$, see
\cite[proof of Thm.~4.2]{marzouk2023}.  For a fixed pattern, combine
Lemma~\ref{lem:param_lip} with a standard $\ell_\infty$-net of the
effective parameter space $[-B, B]^{\le S}$ of spacing
\[
\tau' \;:=\; \frac{\tau}{2 A_L^{(s)}\,M^{P_{L, s}}}.
\]
Following the arguments in \cite[Lemma 3]{Suzuki2018AdaptivityOD}
yields a $C^1$-cover of the sparsity-pattern restricted class of
cardinality at most
\[
\Bigl(\frac{2 B \cdot 2 A_L^{(s)}\,M^{P_{L, s}}}{\tau}\Bigr)^{\!S}.
\]
Taking logarithms, summing over the upper bound $(W+1)^{L S}$ on the number of sparsity patterns,
\[
H \;\le\; L S \log(W + 1) + S \log \tau^{-1}
+ S \bigl(P_{L, s} \log M + \log(4 B A_L^{(s)})\bigr).
\]
Inserting
$A_L^{(s)}, P_{L, s} = \mathcal{O}_{L, s}(1)$, $S, W, B = \mathcal{O}(N)$ and
$M = W \vee B \vee d_1 \vee 1 = \mathcal{O}(N)$ gives~\eqref{eq:entropy_bd}, with the
$s$-dependence sitting entirely in a prefactor $c_{L, s, d_1, d_2}$ and
not in the $N, \tau^{-1}$-scaling.
\end{proof}
\subsection{Main PAC Bound}
\label{sec:companion_pac}
With Theorems~\ref{thm:relu_s_approx} and~\ref{thm:relu_s_entropy}
available, \cite[Thm.~4.8]{marzouk2023} applies verbatim with
$\mathrm{ReLU}^s$ in place of $\mathrm{ReLU}^2$.
\begin{proof}[Proof of Proposition~\ref{thm:pac_learning_error_2}]
Fix $\mu$ satisfying Assumption~\ref{ass:general_regime}.  The proof of the
corresponding~\cite[Thm.~4.8]{marzouk2023} is split into
\emph{approximation error} and a \emph{metric entropy bound}; we show that only
minor modifications are needed to cover the case $s \ge 2$.
\smallskip
\textit{Approximation error.}
Analogously to the proof of~\cite[Thm.~4.8]{marzouk2023}, we apply
\cite[Thm.~3.1]{marzouk2023} to construct a $C^k$-velocity field
$v^\star \in V$ generating the Knothe--Rosenblatt flow from $\nu$ to
$\mu$, with $\|v^\star\|_{C^k(\Omega)}$ bounded by a constant depending only on
$(d, k, \kappa, \mathcal K, M)$; uniformity over $\mathcal T$ follows from the
uniform bound $\|f_\mu\|_{C^k} \le M$ in Assumption~\ref{ass:general_regime}~(A4).  Applying
Theorem~\ref{thm:relu_s_approx} with $r = 1$, $d_1 = d + 1$, $d_2 = d$ to
$v^\star$ produces, for every $N \in \N$, a network
$\hat v \in \Phi_s^{d+1, d}(L, W, S, B)$ with
$L = \mathcal{O}(1)$, $W, S = \mathcal{O}(N)$, $B = \mathcal{O}(\|v^\star\|_{C^k} + N^{1/(d+1)})$
and
\[
\|\hat v - v^\star\|_{C^1(\Omega)}
\;\le\; C_{d, k, s}\,N^{-(k-1)/(d+1)}.
\]
Applying Theorem~\ref{thm:relu_s_approx} with the auxiliary index $r = s-1$ (admissible since $s \leq k+1$ implies $s-1 \leq k$), the same network $\hat v$ also satisfies $\|\hat v - v^\star\|_{W^{s-1,\infty}(\Omega)} \le C N^{-(k-s+1)/(d+1)} \|v^\star\|_{C^k}$. Using the equivalence between $W^{s-1,\infty}$ and $C^{s-1}$ on the bounded Lipschitz domain $\Omega$ for integer $s-1$, this transfers to a $C^{s-1}$-bound. Combined with $\|v^\star\|_{C^{s-1}} \leq \|v^\star\|_{C^k}$ (which uses $s-1 \leq k$), we obtain $\|\hat v\|_{C^{s-1}(\Omega)} \leq r$ for $N$ sufficiently large with truncation radius $r = \mathcal{O}(\|v^\star\|_{C^k}) = \mathcal{O}(1)$. The additional $W^{2,\infty}$-bound required by the hypothesis class~\eqref{eq:hypothesis_main} for $s = 2$ is obtained analogously: Theorem~\ref{thm:relu_s_approx} with auxiliary index $r = 2$ (admissible since $k \ge 2$) gives $\|\hat v - v^\star\|_{W^{2,\infty}} \le CN^{-(k-2)/(d+1)}\|v^\star\|_{C^k}$, and together with $\|v^\star\|_{W^{2,\infty}} \le \|v^\star\|_{C^k} = \mathcal O(1)$ this yields $\|\hat v\|_{W^{2,\infty}(\Omega)} \le r$ for the same $r = \mathcal O(1)$; for $s \ge 3$ the $W^{2,\infty}$-bound is implied by the $C^{s-1}$-bound via $C^{s-1} \subset C^2 \subset W^{2,\infty}$ and no separate argument is needed. Hence $\hat v \in \mathcal{F}_s^{L, W, S, B, r}$ after multiplication by the cut-off $\chi_d$, and the approximation step controls the deterministic (model) error $\mathcal{H}^2(\mu, \Phi^{\hat v}_*\nu) \le C\,N^{-2(k-1)/(d+1)}$, by applying \cite[Lem.~2.6, Thm.~2.7, Thm.~2.8 and Lem.~A.3]{marzouk2023}.
\smallskip
\textit{Metric entropy bound.}
Theorem~\ref{thm:relu_s_entropy} replaces~\cite[Thm.~4.2]{marzouk2023}
inside the proof of~\cite[Thm.~4.8]{marzouk2023}; $s$-dependence
enters only through the prefactors
$C_L^{(s)}, \tilde C_L^{(s)}, A_L^{(s)}, P_{L, s}$ of
Lemmas~\ref{lem:uniform_bd}--\ref{lem:param_lip}, all finite for
$L = \mathcal{O}(1)$. Crucially, the rate exponent $\eta = 2(k-1)/(d+1+2(k-1))$ depends only on the density smoothness $k$ (entering via the approximation step above) and not on the activation order $s$ (which enters only via the prefactor constants), confirming that the statistical convergence rate is fully decoupled from the choice of $s$.
\smallskip
\textit{Balancing terms.}
Choosing the resolution parameter $N$ in
Theorem~\ref{thm:relu_s_approx} in order to balance the approximation error with the metric entropy term, as in the original statement,
\[
N \;\simeq\; n^{\frac{d+1}{d+1+2(k-1)}}.
\]
The general concentration theorem~\cite[Thm.~2.1]{marzouk2023}
combines the approximation bound with the
entropy bound to yield the concentration
inequality~\eqref{eq:pac_hellinger}, with
$\eta = 2(k-1)/(d+1+2(k-1))$ and constants $C_1, C_2 > 0$ depending only on
$(d, s, k, \kappa, \mathcal K, M)$.
\end{proof}
\section{Setup of the Numerical Illustrations}
\label{app:genz_reference}
This appendix specifies the target distributions, the integrands and the closed-form reference values used throughout Section~\ref{sec:numerics}.
We caution the reader that, to keep the notation light, this appendix overloads two symbols compared to the main body: $\Phi_{\mathcal N}$ denotes the standard normal CDF (instead of a transport map), and the bump location parameters of the Gaussian mixture are denoted $a_i$ to avoid clashing with the marginal targets $\mu_j$.
\subsection{Diagonal Regime}
\subsubsection{Target Densities}
\label{app:targets}
All marginal targets $\mu$ on $[0,1]$ are mixtures of a uniform floor and a finite number of Gaussian bumps clipped to $[0,1]$,
\begin{equation}
\label{eq:mixture_marginal_density}
f_\mu(x)
\;=\;
\frac{1}{Z_\mu}\!
\left(\,
c_0 \;+\; \sum_{i=1}^{K} w_i\,
\frac{1}{\sigma_i\sqrt{2\pi}}\,
\exp\!\Big(\!-\frac{(x-a_i)^2}{2\sigma_i^2}\Big)
\right),
\qquad
x\in[0,1],
\end{equation}
where $c_0 > 0$ is a uniform background, $\{w_i\}\subset[0,1]$ are mixture weights summing to one, and $\{(a_i,\sigma_i)\}$ are bump locations and widths. The normalizing constant
\begin{equation}
\label{eq:Z_mu_closed_form}
Z_\mu
\;=\;
c_0 \;+\; \sum_{i=1}^{K} w_i\,
\Big[\,\Phi_{\mathcal N}\!\big(\tfrac{1-a_i}{\sigma_i}\big)
       - \Phi_{\mathcal N}\!\big(\tfrac{-a_i}{\sigma_i}\big)\Big]
\end{equation}
is closed-form in the standard normal CDF $\Phi_{\mathcal N}$, as is the marginal CDF $F_\mu$. The uniform floor $c_0 > 0$ guarantees the lower bound $f_\mu \geq c_0/Z_\mu > 0$ required by Assumption~\ref{ass:diagonal_regime}~(A3) and hence by Theorem~\ref{thm:pac_lti_diagonal}.
For the one-dimensional comparison in Section~\ref{subsec:numerics_genz_1d} we use three representative parameter choices; cf. fig.~\ref{fig:target_densities}
\vspace{1em}
\begin{center}
\renewcommand{\arraystretch}{1.15}
\begin{tabular}{l c c c c}
\toprule
label & $c_0$ & $\{w_i\}$ & $\{a_i\}$ & $\{\sigma_i\}$ \\
\midrule
A: broad    & $0.50$ & $\{1\}$       & $\{0.50\}$       & $\{0.20\}$ \\
B: peaked   & $0.10$ & $\{1\}$       & $\{0.55\}$       & $\{0.08\}$ \\
C: bimodal  & $0.20$ & $\{0.5,0.5\}$ & $\{0.25,0.75\}$  & $\{0.08,0.08\}$ \\
\bottomrule
\end{tabular}
\end{center}
 \vspace{1em}
\begin{figure}
    \centering
    \includegraphics[width=0.9\linewidth]{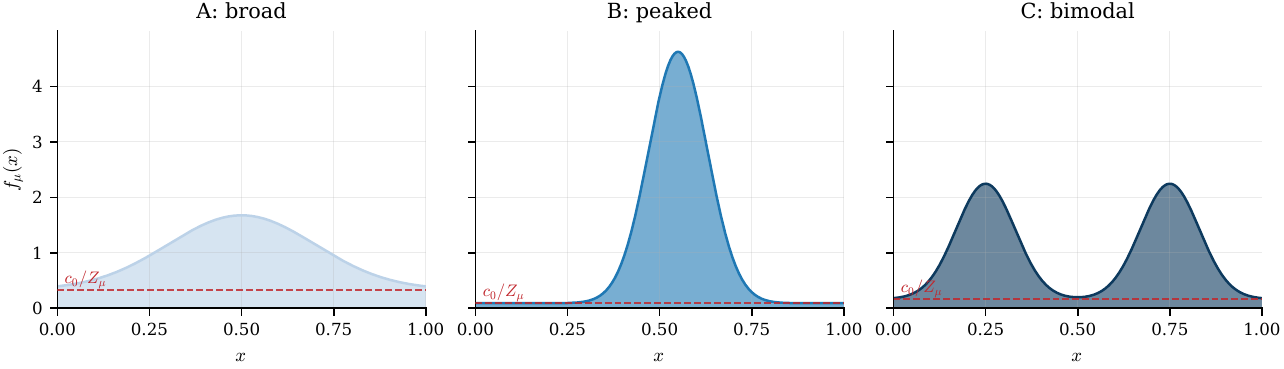}
    \caption{Target densites for the one-dimensional comparison in Section~\ref{subsec:numerics_genz_1d}}
    \label{fig:target_densities}
\end{figure}
For the multi-dimensional illustration in Section~\ref{subsec:numerics_genz_multid} we use the all-broad product target $\mu = \bigotimes_{j=1}^d \mu^{\textsc{a}}_j$ with each marginal $\mu^{\textsc{a}}_j$ given by row A above.
\subsubsection{Genz Test Integrands}
\label{app:integrands}
We adopt the integrands of \cite{Genz1984} adapted to $[0,1]$. The 1D versions used in Section~\ref{subsec:numerics_genz_1d} are
\begin{align}
\qoi_1(x) &= \cos\!\big(2\pi w + c\,x\big),
&& c=6,\;w=0.4,
&& \text{(oscillatory, $C^\infty$)},
\label{eq:f1_def}
\\
\qoi_4(x) &= \exp\!\big(-c^2(x-w)^2\big),
&& c=4,\;w=0.4,
&& \text{(Gaussian peak, $C^\infty$)},
\label{eq:f4_def}
\\
\qoi_6(x) &= \mathds{1}\{x>w\}\,e^{\,c(x-1)},
&& c=2,\;w=0.4,
&& \text{(jump, $L^1$)}.
\label{eq:f6_def}
\end{align}
The integrands $\qoi_1, \qoi_4$ are smooth and fall within the regularity hypothesis of Theorem~\ref{thm:pac_lti_diagonal}; $\qoi_6$ is included as a case outside the theoretical scope; but covered by \cite{bungartz2004sparse}.
For the multi-dimensional illustration we use the canonical separable extensions of $\qoi_1, \qoi_4$ to $[0,1]^d$,
\begin{equation}
\label{eq:fdef_dim}
\qoi_1(x) = \cos\!\Big(2\pi w + \sum_{j=1}^d c_j\,x_j\Big),
\qquad
\qoi_4(x) = \exp\!\Big(\!- \sum_{j=1}^d c_j^2\,(x_j - w)^2\Big),
\end{equation}
with $w=0.4$ and per-dimension parameter scaling
\begin{equation}
\label{eq:c_scaling}
\qoi_1: \quad c_j = \tfrac{6}{d},
\qquad
\qoi_4: \quad c_j = \tfrac{1}{\sqrt{d}}.
\end{equation}
The scaling keeps the supremum norm $\|\qoi\|_\infty = 1$ as $d$ varies. 
 
\subsubsection{Closed-Form Reference Values}
\label{app:closed_form_E_mu}
For each integrand, $\evdist{\mu}{\qoi}$ decomposes into a uniform and a Gaussian-bump contribution,
\begin{equation}
\evdist{\mu}{\qoi}
\;=\;
\frac{1}{Z_\mu}\!\left(
c_0\!\!\int_0^1\!\qoi(x)\,\d x
\;+\;
\sum_{i=1}^{K} w_i\!\!\int_0^1\!\qoi(x)\,
\frac{e^{-(x-a_i)^2/(2\sigma_i^2)}}{\sigma_i\sqrt{2\pi}}\,\d x
\right),
\end{equation}
in which both pieces admit closed-form expressions in the (real or complex) error function. The key identity for the bump contributions is the complete-the-square integral
\begin{equation}
\label{eq:gauss_explin_identity}
\begin{aligned}
\int_a^b e^{\beta x}\,
\frac{e^{-(x-a_0)^2/(2\sigma^2)}}{\sigma\sqrt{2\pi}}\,\d x
&\;=\;
\tfrac{1}{2}\,e^{\,a_0\beta + \sigma^2\beta^2/2}\,
\Big[\erf\!\big(\tfrac{b-z_\beta}{\sigma\sqrt{2}}\big)
     - \erf\!\big(\tfrac{a-z_\beta}{\sigma\sqrt{2}}\big)\Big], \\
& \qquad z_\beta := a_0 + \beta\sigma^2,
\end{aligned}
\end{equation}
valid for any $\beta\in\mathbb{C}$ (with the standard analytic continuation of $\erf$ for non-real $\beta$, and $\operatorname{Re}$ denoting the real part). Writing $N(a, \sigma^2)(x) := \frac{1}{\sigma\sqrt{2\pi}} \exp\!\big(-(x-a)^2/(2\sigma^2)\big)$ for the Gaussian bump density, the three integrand classes reduce to~\eqref{eq:gauss_explin_identity} as follows:
\begin{itemize}
\item For $\qoi_1$, take $\beta = ic$ and use
\[
\int_0^1 \cos(2\pi w + cx)\, N(a, \sigma^2)(x)\,\d x
= \operatorname{Re}\!\Big[\,e^{i (2\pi w + a c)} \int_0^1 e^{icx}\, N(a, \sigma^2)(x)\, \d x\,\Big].
\]
\item For $\qoi_4$, the product $e^{-c^2(x-w)^2}\,N(a, \sigma^2)(x)$ is itself a (rescaled) Gaussian, proportional to $N(\tilde a, \tilde\sigma^2)(x)$ with
\[
\tilde\sigma^{-2} = 2c^2 + \sigma^{-2}, \qquad
\tilde a = \tilde\sigma^2\big(2 c^2 w + a / \sigma^2\big).
\]
\item For $\qoi_6$, restrict the integral to $[w, 1]$ and take $\beta = c$:
\[
\int_0^1 \mathds{1}\{x>w\}\,e^{c(x-1)}\, N(a, \sigma^2)(x)\, \d x
= e^{-c}\!\int_w^1 e^{cx}\, N(a, \sigma^2)(x)\, \d x.
\]
\end{itemize}
 
In the multi-dimensional case the separable structure of $\qoi_4$ and the product structure of $\mu$ give $\evdist{\mu}{\qoi_4} = \prod_{j=1}^d \evdist{\mu_j}{g_j}$ with $g_j(x_j)=e^{-c_j^2(x_j-w)^2}$, so the full $d$-dimensional reference is the product of the 1D pieces. For the non-separable oscillatory integrand $\qoi_1$, factorization via the characteristic function gives
\begin{equation}
\evdist{\mu}{\qoi_1}
\;=\;
\operatorname{Re}\!\bigg(e^{2\pi i w}\prod_{j=1}^d \varphi_{\mu_j}(c_j)\bigg),
\qquad
\varphi_{\mu_j}(c) \;:=\; \evdist{\mu_j}{e^{i\,c\,X_j}},
\end{equation}
and each marginal characteristic function $\varphi_{\mu_j}$ is closed form via~\eqref{eq:gauss_explin_identity} with $\beta = ic$.
\subsubsection{Multivariate Sparse Grid Setup}
\label{app:sampling_and_grids}
Sparse grids in Section~\ref{subsec:numerics_genz_multid} use the Clenshaw--Curtis nodes with the closed nonlinear growth $m_l = 2^{l-1}+1$ for $l \ge 2$ and $m_1 = 1$, combined into the Smolyak rule of sparsity level $\ell$. The dimension-dependent sparsity ranges used are
\begin{equation}
\label{eq:sparsity_grids}
\ell \in
\begin{cases}
\{1,\dots,7\},   & d=2,\\
\{1,\dots,6\},   & d\in\{5,10\},\\
\{1,\dots,6\},   & d=15,
\end{cases}
\end{equation}
Empirical-quantile errors are reported as the median over $4$ independent runs of the empirical quantile transport; naive Monte Carlo errors are reported as the median over $50$ independent runs at each budget.
\subsection{General Regime}
\subsubsection{Training Setup for the Trained-Flow Experiment}
\label{app:training_setup_relus}
The target $\mu$ used in Section~\ref{subsec:numerics_trained} is a two-bump Gaussian mixture on the diagonal of $[0,1]^2$,
\[
f_\mu(x) \;=\; \frac{1}{Z_\mu}\Big(c_0 + \tfrac{1}{2} \mathcal{N}(\tfrac{3}{10}\mathbf{1}, \sigma^2 I_2)(x) + \tfrac{1}{2} \mathcal{N}(\tfrac{7}{10}\mathbf{1}, \sigma^2 I_2)(x)\Big),
\]
with floor $c_0 = 0.30$ and isotropic width $\sigma = 0.18$. The uniform floor enforces $f_\mu \geq c_0/Z_\mu \approx 0.25$, satisfying Assumption~\ref{ass:general_regime}~(A3). 
The vector field $v^\theta$ is a fully connected $\mathrm{ReLU}^s$ network of width $128$ and depth $3$, passed through a smooth saturation $g_S(z) = S\tanh(z/S)$ with $S = 20$ and the boundary cut-off $\chi_d(x) = (x_1(1-x_1), x_2(1-x_2))^\top$. The saturation uniformly bounds $\|v^\theta\|_\infty$ to prevent $\mathrm{float32}$ overflow from compounding $\mathrm{ReLU}^s$ pre-activations during training and, being $C^\infty$, preserves the $C^{s-1}$ regularity class of the resulting flow. We use the regularized NLL with per-sample soft cap
\[
\mathcal{L}(\theta;x) \;=\; \operatorname{soft\_cap}_{\kappa}\!\Big(\!-\tfrac{1}{2}\log\big(\det(J J^\top) + \varepsilon\big)\Big),
\]
$J = \nabla_x \Phi^{-\theta}(x)$, $\varepsilon = 10^{-3}$, $\kappa = 3$, $\operatorname{soft\_cap}_\kappa(y) = y$ for $y \leq \kappa$ and $\kappa + \log(1 + y - \kappa)$ otherwise. The regularization bounds the gradient of $\log|\det J|$ near singular Jacobians; the soft cap keeps the batch gradient informative when individual samples are nearly degenerate. Training runs $3000$ AdamW iterations (weight decay $10^{-4}$, gradient clip $1$) with cosine-decayed learning rate from $\eta_0 = 10^{-3}$ down to $\eta_0/10$, on fresh batches of $2048$ samples drawn from $\mu$ each iteration and reverse-time RK4 with ten steps. Among the periodic checkpoints we retain the one whose held-out eval NLL is lowest.
The retained flows are evaluated on Clenshaw--Curtis sparse grids of levels $\ell = 1, \ldots, 9$, corresponding to budgets $m \in \{5, 13, 29, 65, 145, 321, 705, 1537, 3329\}$. The LtI estimator on $\qoi_1, \qoi_4$ is compared against the median of $80$ naive Monte Carlo runs at matching $m$.

\subsubsection{Synthetic LtI Test on Random \texorpdfstring{$\mathrm{ReLU}^s$}{ReLUs} Neural-ODE Flows}
\label{app:synthetic_quadrature_details}
For each $s \in \{2, 3, 5, 8\}$ and seed index $i = 1, \ldots, 32$ we draw a one-hidden-layer $\mathrm{ReLU}^s$ vector field
\[
v^{\theta_i}(x, t) \;=\; \chi_d(x) \odot \sum_{k=1}^{K} a_k^{(i)} \, \max\!\big(0, w_k^{(i)\top} [x; t] - b_k^{(i)}\big)^s, \qquad x \in [0,1]^2,\ t \in [0, 1],
\]
with $K = 8$ features, $\chi_d(x) = (x_1(1-x_1), x_2(1-x_2))^\top$ the boundary cut-off, directions $w_k^{(i)} \in \mathbb{R}^3$ drawn uniformly from the unit sphere (sampled as $\tilde w / \|\tilde w\|$ with $\tilde w \sim \mathcal{N}(0, I_3)$), offsets $b_k^{(i)} = w_k^{(i)\top} z_k^{(i)}$ placed at random kink points $z_k^{(i)} \sim \mathrm{Uniform}([0,1]^3)$, and output weights $a_k^{(i)} \in \mathbb{R}^2$ standardized by $\mathcal{N}(0, K^{-1} I_2)$. This way we ensure that each kink hyperplane $\{(x, t) : w_k^{(i)\top} [x; t] = b_k^{(i)}\}$ passes through the cube interior, so that the worst-case $C^{s-1}$ regularity of the resulting feature is actually attained on the test domain.
The associated time-1 flow $\Phi^{\theta_i}$ on $[0,1]^2$ is computed by RK4 with $40$ steps. The reference integrals against the uniform measure
\[
I_{s,k}^{(i)} \;=\; \int_{[0,1]^2} \qoi_k\big(\Phi^{\theta_i}(z)\big)\, dz, \qquad k \in \{1, 4\},
\]
are computed by tensor product Gauss--Legendre with $n_{\mathrm{GL}} = 1000$ nodes per axis ($10^6$ total nodes) using the same RK4 flow, so that the temporal-discretisation error of the flow cancels exactly between the reference integrals and the Smolyak estimates. For each $(s, k)$ the corresponding panel of Figure~\ref{fig:synthetic_quadrature} is the median over the $32$ seeds.
\end{document}